\documentclass[11pt]{amsart}

\usepackage[dvipsnames,svgnames,x11names]{xcolor}
\usepackage{textgreek,bbm,url,graphicx,verbatim,amssymb,enumerate,stmaryrd,booktabs,lmodern,mathtools,microtype,mathabx}
\usepackage[backref=false,colorlinks,citecolor=Mahogany,linkcolor=Mahogany,urlcolor=Mahogany,filecolor=Mahogany]{hyperref}
\usepackage[capitalize]{cleveref}
\usepackage[mathscr]{euscript}
\usepackage{tikz,tikz-cd,circuitikz}
\usetikzlibrary{matrix,calc,positioning,arrows,decorations.pathreplacing,patterns,arrows}
\usepackage{mathabx}
\usepackage[a4paper,
            left=1in,
            right=1in,
            top=1in,
            bottom=1in,
            footskip=.25in]{geometry}
\usepackage{enumitem}
\usepackage{makecell}

\DeclareMathSymbol{\in}{\mathrel}{symbols}{"32}
\DeclareMathSymbol{\subset}{\mathrel}{symbols}{"1A}
\DeclareMathSymbol{\supset}{\mathrel}{symbols}{"1B}
\DeclareMathSymbol{\cup}{\mathbin}{symbols}{"5B}
\DeclareMathSymbol{\cap}{\mathbin}{symbols}{"5C}
\DeclareMathSymbol{\wedge}{\mathbin}{symbols}{"5E}
\DeclareMathSymbol{\vee}{\mathbin}{symbols}{"5F}
\DeclareMathSymbol{\land}{\mathbin}{symbols}{"5E}
\DeclareMathSymbol{\lor}{\mathbin}{symbols}{"5F}
\DeclareMathSymbol{\Leftarrow}{\mathrel}{symbols}{"28}
\DeclareMathSymbol{\Rightarrow}{\mathrel}{symbols}{"29}
\DeclareMathSymbol{\Leftrightarrow}{\mathrel}{symbols}{"2C}

\def\Xint#1{\mathchoice
{\XXint\displaystyle\textstyle{#1}}%
{\XXint\textstyle\scriptstyle{#1}}%
{\XXint\scriptstyle\scriptscriptstyle{#1}}%
{\XXint\scriptscriptstyle\scriptscriptstyle{#1}}%
\!\int}
\def\XXint#1#2#3{{\setbox0=\hbox{$#1{#2#3}{\int}$ }
\vcenter{\hbox{$#2#3$ }}\kern-.6\wd0}}

\def\dashint{\Xint-}

\newtheorem{theorem}{Theorem}[section]
\newtheorem{lemma}[theorem]{Lemma}
\newtheorem{proposition}[theorem]{Proposition}
\newtheorem{corollary}[theorem]{Corollary}
\newtheorem*{corollary*}{Corollary}

\theoremstyle{definition}
\newtheorem{definition}[theorem]{Definition}

\theoremstyle{remark}

\newtheorem{remark}[theorem]{Remark}

\newtheorem{question}[theorem]{Question}

\newcommand{\subref}[2]{\hyperref[#2]{\ref{#1}.\ref{#2}}}

\numberwithin{equation}{section}
\crefname{equation}{}{}

\newcommand{\sq}[1]{\widetilde{#1}}
\newcommand{\R}{\mathbb{R}}
\newcommand{\Z}{\mathbb{Z}}
\newcommand{\N}{\mathbb{N}}

\newcommand{\C}{\mathbb{C}}
\newcommand{\mc}[1]{\mathcal{#1}}

\newcommand{\ov}[1]{\overline{#1}}
\newcommand{\ang}[1]{\langle #1 \rangle}

\newcommand{\dist}{\textup{dist}}

\DeclareMathOperator*{\einf}{ess \, inf}
\DeclareMathOperator*{\esup}{ess \, sup}

\title{Medians, Oscillations, and Distance Functions}

\author{Marcus Pasquariello \and Ignacio Uriarte-Tuero}

\begin{document}

\begin{abstract}
\cite{Vasin_LimitSetOfFuchsianGroupAndDynkinsLemma} (for $n=1$) and \cite{anderson2022weakly} (for $n>1$) provided a geometric characterization of the sets $E \subset \R^n$ for which $w = \dist(\cdot, E)^{-\alpha}$ is a Muckenhoupt $A_1$ weight for some $\alpha > 0$. In this paper, we provide a geometric characterization of the sets $E \subset \R^n$ (which we call \emph{median porous sets}) such that $w = \dist(\cdot, E)^{-\alpha}$ is a Muckenhoupt $A_p$ weight for some $\alpha > 0$ (given any $1 < p \leq \infty$). Given $1 < p \leq \infty$, we also find the precise range of exponents $\alpha$ for which $w = \dist(\cdot, E)^{-\alpha} \in A_p$, in analogy to the $p=1$ case done in \cite{anderson2022weakly}. With our characterization, we prove that $\R^n \setminus E$ supports a Hardy-Sobolev inequality if $E$ is an appropriate median porous set. The proofs rely on a new median-value characterization of $BMO$: For a real-valued measurable function on $\R^n$ and constants $0 < s < t < 1$,
\[\|f\|_{BMO} \approx_{s, t, n} \sup_{Q}[M_t(f, Q) - M_s(f, Q)]\]
where $M_s(f, Q)$ denotes the $s$-median value of $f$ on $Q$.
\end{abstract}
    
\maketitle

\section{Introduction}\label{sec:intro}

\subsection*{Quick overview}

A main theme of this paper is the relationship between geometric notions of porosity for a set \(E\subset\mathbb R^n\)—that is, the presence of holes at various locations and scales—and the
membership of \(\operatorname{dist}(\cdot,E)^{-\alpha}\) in a Muckenhoupt class \(A_p\), a property with
important applications in harmonic analysis and PDEs.

Before giving the various definitions, recent partial results, and our solution to the problem, we will briefly recall the history of this problem, which also highlights some of its applications. 

Powers of distance functions that belong to Muckenhoupt classes arise in several areas of analysis. They
have been used in the study of Riesz capacities and Whitney decompositions
\cite{AikawaQuasiadditivityofRieszcapacity}, weighted Sobolev embeddings
\cite{HoriuchiImbeddingTheoremsForWeightedSobolevSpaces2}, regularity questions for polyharmonic functions \cite{AimarCarenaDuranToschi_PowersOfDistancesToLowerDimensionalSetsAsMuckenhouptWeights}, and similar questions for the 
divergence and Stokes equations
\cite{DuranLopezGarcia_SolsOfDivAndStokes}.

They also provide a useful approach to Hardy--Sobolev inequalities. Lehrb\"ack and V\"ah\"akangas \cite{LehrbackVahakangasInbetweenSobolevandHardy} obtained
such inequalities using $A_1$ properties of distance weights under Assouad-codimension conditions, while
Dyda, Ihnatsyeva, Lehrb\"ack, Tuominen, and V\"ah\"akangas established corresponding $A_p$ results and
global Hardy--Sobolev inequalities under a porosity assumption
\cite{dyda2017muckenhoupt}. More broadly, the problems of determining which domains support Hardy--Sobolev inequalities and
understanding the self-improvement of such inequalities have been studied through several methods,
including isoperimetric inequalities, Sobolev and Riesz capacity estimates, and hyperbolic fillings; see,
for example,
\cite{KoskelaZhongHardyInequalityAndBoundarySize,
KoskelaLehrback_WeightedPointwiseHardy,
Lehrbäck2013dimension,
LehrbackWeightedHardyInequalitiesSizeOfBoundary,
LehrbackVahakangasInbetweenSobolevandHardy,
HurriSyrjanenVahakangas_FractionalSobolevPoincareAndFractionalHardyUnboundedJohnDomains,
ErikssonBiqueLehrbackVahakangas_SelfImprovementOfWeightedPointwiseInequalitiesOnOpenSets,
ErikssonBiqueKlineSelfImprovementOfFractionalHardyInequalities}
and the references therein. The role of distance weights in this theory, including weighted Hardy--Sobolev
and Poincar\'e inequalities, is discussed systematically in the excellent monograph
 \cite{KinnunenLehrbackVahakangasBook}. The results described above connect these questions to the geometry
of $E$ through the Muckenhoupt properties of powers of $\dist(\cdot,E)$ and provide the starting point for
our applications. 

Finally, one particular case ($p=1$) of the porosity versus $A_p$ characterization problem was solved in two elegant papers. Indeed, the sets $E \subset \R^n$ for which $w = \dist(\cdot, E)^{-\alpha}$ is an $A_1$ weight for some $\alpha$ were characterized by \cite{Vasin_LimitSetOfFuchsianGroupAndDynkinsLemma} in dimension $1$ (the boundary of the unit disk) and by Anderson, Lehrb\"ack, Mudarra and V\"ah\"akangas in \cite{anderson2022weakly} in $\R^n$. In the latter manuscript, the precise range of exponents $\alpha$ satisfying the $A_1$ property is found.

There have been some partial results concerning the exponents $\alpha$ for which $w = \dist(\cdot, E)^{-\alpha}$ is an $A_p$ weight, but they all assume either porosity of the set $E$, or that $w = \dist(\cdot, E)^{-\alpha}$ is an $A_1$ weight for some $\alpha$. We will describe some of those in more detail momentarily. 

In this manuscript, we solve the $A_p$ (and the $A_\infty$) problems unconditionally, i.e.~with no side conditions. We characterize geometrically the sets $E \subset \R^n$ for which $w = \dist(\cdot, E)^{-\alpha}$ is an $A_p$ weight for some $\alpha$ and some $p$. Furthermore, given $1 < p \leq \infty$, we identify the precise range of exponents $\alpha$ for which $w = \dist(\cdot, E)^{-\alpha} \in A_p$, extending the work of \cite{anderson2022weakly} in the $p = 1$ case. We further recover, with a different proof, the geometric characterization of sets for $p=1$ in \cite{anderson2022weakly}.

We then apply our results to weighted Hardy--Sobolev inequalities. In the critical case
$q=p^{\ast}=\frac{np}{n-p}$, we prove that domains much more general than porous ones support weighted Hardy-Sobolev inequalities. In this case, we also prove matching necessary
conditions for the Hardy--Sobolev inequality. 

Regarding the proof of the characterization of sets $E$ for which $w = \dist(\cdot, E)^{-\alpha}$ is an $A_p$ weight for some $\alpha$ and some $p$, our approach uses a pointwise domination formula for functions that improves upon earlier formulas by Lerner \cite{Lerner2010} and Hyt\"onen \cite{Hytonen2012}. This pointwise domination formula then allows us to find a new characterization of BMO, which we use crucially in our characterization of $A_p$ distance functions.

As a byproduct of our arguments, we also prove that we have reached the limit of both the $A_p$ method and the widely used Riesz potential method to prove weighted Hardy-Sobolev inequalities, in that both methods are sharp in our proofs.

For the sake of easy reference, we summarize, with pointers to the corresponding statements, what we feel are the main contributions of this manuscript, that we have just described.

\begin{enumerate}

\item We characterize geometrically the nonempty sets $E\subset\R^n$ for which
$\dist(\cdot,E)^{-\alpha}\in A_p$ for some $\alpha>0$ and some
$1<p<\infty$. These are precisely the median porous sets
(\cref{def:DefVolumeQuantities,DefMedianPorousSets,thm:BMO_char_dist}).
This extends the corresponding $A_1$ characterization of weakly porous sets
from \cite{Vasin_LimitSetOfFuchsianGroupAndDynkinsLemma,anderson2022weakly}.

\item For each $1<p\leq\infty$, we determine the precise range of exponents
$\alpha\neq0$ for which $\dist(\cdot,E)^{-\alpha}\in A_p$. The range is
described in terms of the $p$-Muckenhoupt exponents $\textup{Mu}_p(E)$
introduced in this paper
(\cref{def:muck_exponent_p,thm:Ap_range,thm:Ap_range_2}).

\item The geometric characterization is obtained from a median sparse-domination
formula valid for all $0<s<t<1$
(\cref{thm:median_decomp,thm:median_decomp2}). As a consequence, we obtain
median-value characterizations of $BMO$ and $BLO$
(\cref{thm:BMO_char}).

\item In the critical case $q=np/(n-p)$, we apply the distance-weight
characterization to prove sufficient and necessary conditions for weighted
Hardy--Sobolev inequalities under median-porosity assumptions
(\cref{cor:hardy_sobolev}, \cref{FractionalAndFirstOrderHardySobolevInequalitiesViaRieszPotentials}, and \cref{NecessaryConditionsForHardySobolevThm}). We also obtain related
distance-weighted Poincar\'e inequalities
(\cref{rem:WeightedPoincareInequalities}).

\item Our methods also yield a characterization of the nonnegative H\"older
continuous functions $w$ for which $\log w\in BMO$
(\cref{thm:HolderContinuousCase}).

\item Finally, the quantitative characterization determines the range in which
the $A_p$-based arguments used here are available. We also show that the
testing condition underlying the Riesz-potential method fails for nonporous
median porous sets in the subcritical range
(\cref{rem:WeReachedTheLimitOfTheRieszPotentialMethod}).

\end{enumerate}

\subsection*{More detailed overview}

For a real-valued measurable function $f$ on a cube $Q$ and a constant $0 < s < 1$, let $M_s(f, Q)$ denote the upper $s$-median value of $f$ on $Q$. That is, $M_s(f, Q)$ is the largest value $\lambda$ satisfying
\[|\{x \in Q : f(x) < \lambda\}| \leq s|Q| \quad \text{and} \quad |\{x \in Q : f(x) > \lambda\}| \leq (1 - s)|Q|.\]
Several important spaces in harmonic analysis, such as $BMO$ and $BLO$ (definitions recalled in \cref{sec:prelim}) are defined for locally integrable functions $f$ in terms of their averages 
\[\ang{f}_Q := \dashint_{Q}f \, dx.\] 
In this paper, we find characterizations of these spaces in terms of their median values $M_s(f, Q)$. One advantage of these characterizations is that they allow us to define $BMO$ and $BLO$ for functions that are not a priori locally integrable. \par
The first kind of median value characterizations was obtained by John \cite{john1962} where it was proved that if there are constants $0 \leq a < 1/2$, $\lambda > 0$, and $\{c_Q\}$ such that
\[|\{x \in Q : |f(x) - c_Q| > \lambda\}| < a|Q|\]
for all cubes $Q$, or equivalently, 
\[\sup_Q\inf_{c \in \C}M_{1 - a}(|f - c|, Q) < \infty,\]
then $f \in BMO$. This result was extended to the $a = 1/2$ case in \cite{stromberg1979} which is sharp. \par Several years later, by adapting John's argument in a very efficient way, Lerner proved a ``local median oscillation'' formula which describes the oscillations of a function in terms of its median values over a sparse family of cubes. The formula was later improved by Hytönen which has the following form:  given a function $f$, there exists a sparse family $\mc{S}$ of subcubes of $Q_0$ such that
\[|f(x) - M_s(f, Q_0)| \leq 2\sum_{Q \in \mc{S}}\inf_{c \in \C}M_{1-a}(|f - c|, Q)\chi_Q(x)\]
a.e.~on $Q_0$ for $a = 2^{- n - 2}$ (\cite{Lerner2010} and \cite{Hytonen2012}). This result initiated the study of sparse domination which has been a very important tool in harmonic analysis for establishing sharp weighted norm inequalities. \par
In this paper, we establish an improvement of Lerner's formula. Our first theorem is
\begin{theorem}\label{thm:median_decomp}
    Let $f$ be a real-valued measurable function on a cube $Q_0$ and let $0 < s < t < 1$.  Then there is an $\eta$-sparse family $\mc{S} = \mc{S}(f, Q_0, s, t)$ of (not necessarily dyadic) subcubes of $Q_0$ such that
        \[|f(x) - M_s(f, Q_0)| \lesssim_{s, t, n} \sum_{Q \in \mc{S}}[M_t(f, Q) - M_s(f, Q)]\chi_Q(x)\]
    a.e.~on $Q_0$, where $\eta$ depends only on $s,t$, and $n$. 
\end{theorem}
In particular, this gives us new median-value characterizations of $BMO$ and $BLO$:
\begin{theorem}\label{thm:BMO_char}
    Let $f$ be a real-valued measurable function on $\R^n$ and fix $0 \leq s < t < 1$.
    \begin{enumerate}
        \item  \label{thm:BMO_char_BMO} For $s > 0$, we have
        \[f \in BMO \quad \text{if and only if} \quad \|f\|_{s,t} := \sup_{Q}[M_t(f, Q) - M_s(f, Q)] < \infty.\]
        Moreover,
        \[\|f\|_{s, t} \lesssim_{s, t} \|f\|_{BMO} \lesssim_{s, t, n}\|f\|_{s,t}.\]
        \item  \label{thm:BMO_char_BLO} For $s = 0$, we have
        \[f \in BLO \quad \text{if and only if} \quad \|f\|_{0,t} := \sup_{Q}[M_t(f, Q) - \einf_Q f] < \infty.\]
        Moreover,
        \[\|f\|_{0, t} \lesssim_{t} \|f\|_{BLO} \lesssim_{t, n}\|f\|_{0,t}.\]
    \end{enumerate}
\end{theorem}
We remark that if $t - s > 1/2$, then the theorem follows easily from previous results in the literature. Therefore the importance of the theorem is that the characterization holds for \emph{all} $0 < s < t < 1$. Also, somewhat surprisingly, if $t - s < 1/2$, then the theorem fails in the dyadic setting (see \cref{sec:dyadic_example}). \par 
As an application of this characterization, we characterize certain natural integrability properties of distance functions. More specifically, let $E \subset \R^n$, and consider the distance function 
\[d_E(x) := \dist(x, E).\]
The (vaguely stated) question we will discuss is:
\begin{question}
    Given an analytic property $A$ of $\dist(\cdot, E)$, what is the corresponding geometric property $G$ of $E$ such that 
    \[\dist(\cdot, E) \text{ has property } A \text{ if and only if } E \text{ has property } G?\]
\end{question}
It turns out that certain properties of $E$ defined in terms of the holes of $E$ will correspond to integrability properties of $\dist(\cdot, E)$. We say that $E \subset \R^n$ is \emph{porous} if for every cube $Q_0$, there is some subcube $Q \subset Q_0$ with $Q \cap E = \emptyset$ and $|Q| \geq c|Q_0|$ where $0 < c < 1$ is some universal constant. It is well known that $E$ is porous if and only if there is some $\alpha > 0$ such that 
\[\dashint_{Q}\dist(x, E)^{-\alpha} \, dx \lesssim l(Q)^{-\alpha}\]
for all cubes $Q$ that intersect $E$ (for example see \cite{Luukkainen1998Dimension} and \cite{Lehrbäck2013dimension}). This equivalence is closely related to the well-known fact that 
\begin{equation}\label{eqn:porosity_equiv_codim}
      \textup{\underline{co\,dim}}_{A}(E)>0
      \quad\text{if and only if}\quad
      E\text{ is porous}.
  \end{equation}
Here $\underline{\textup{co\,dim}}_A(E)$ denotes the Assouad codimension of $E$ (see \cref{sec:Ap_range} for the definition). See, for example, \cite{KinnunenLehrbackVahakangasBook}
and \cite[Section 5]{Luukkainen1998Dimension}. \par
The previous result was refined by defining a more general notion of porosity (weak porosity) and proving its equivalence to the corresponding integrability property (the Muckenhoupt $A_1$ property), first in \cite{Vasin_LimitSetOfFuchsianGroupAndDynkinsLemma} (for the unit circle) and then in \cite{anderson2022weakly} (for $\R^n$).
We say that the set $E \subset \R^n$ is \emph{weakly porous} if for every cube $Q_0$ that intersects $E$, there are dyadic subcubes $Q_1, ..., Q_k \subset Q_0$ such that the following three properties hold:
\begin{equation}\label{WeakPorosityDefIntrod}
Q_i \cap E = \emptyset, \quad |Q_i| \geq \delta \, \mc{V}_1(Q_0) , \quad \sum_{i = 1}^{k}|Q_i| \geq (1 - s)|Q_0|
\end{equation}
where $\mc{V}_1(Q_0)$ denotes the volume of a maximum dyadic subcube of $Q_0$ that doesn't intersect $E$ and $0 < s, \delta < 1$ are universal constants. It is clear that porous sets are weakly porous. Anderson, Lehrb\"ack, Mudarra and V\"ah\"akangas proved the following elegant theorem:
\begin{theorem}[\cite{anderson2022weakly}]\label{thm:weakly_porous_VIntrod}
    $E$ is weakly porous if and only if $\dist(\cdot, E)^{-\alpha} \in A_1$ for some $\alpha > 0$ (or equivalently $- \log \dist(\cdot, E) \in BLO$).
\end{theorem}
In other words, $E$ is weakly porous if and only if for all cubes $Q$,
\[\dashint_{Q}\dist(x, E)^{-\alpha} \, dx \lesssim \inf_{Q}\dist(x, E)^{-\alpha}.\]
\par 
Using the median value characterization we established for $BMO$ (\cref{thm:BMO_char}), we recover \cref{thm:weakly_porous_VIntrod} with a different proof (\cref{cor:OurProofOfWeakPorosityResult}), and we further refine the results of \cite{anderson2022weakly} by proving a characterization of $BMO$ for distance functions. Indeed, we completely answer two questions left open in \cite{anderson2022weakly} (see \cref{thm:BMO_char_dist} and \cref{thm:Ap_range}).

\begin{definition} 
    For a constant $0 < s < 1$ and a cube $Q_0$ that intersects $E$, let $\mc{V}_s(Q_0)$ denote the largest value $\delta > 0$ for which we can fill a $(1 - s)$-th proportion of $Q_0$ with $E$-free dyadic cubes $Q$ with measure $|Q| \geq \delta$. That is,
    \[\mc{V}_s(Q_0) := \sup\left\{\delta > 0 : \exists Q_1, ..., Q_k \in \mc{D}(Q_0),\, Q_i \cap E = \emptyset, \, |Q_i| \geq \delta, \, \sum_{i}|Q_i| \geq (1 - s)|Q_0|\right\}.\]
\end{definition}
\begin{definition}
    We say that $E$ is a \emph{median porous set} (or $E$ has median porosity) if there are constants $0 < s < t \leq 1$ and $0 < \delta < 1$ such that for every cube $Q_0$ that intersects $E$, there are dyadic subcubes $Q_1, ..., Q_k \in \mc{D}(Q_0)$ such that the following three properties hold
    \[Q_i \cap E = \emptyset, \quad |Q_i| \geq \delta \, \mc{V}_t(Q_0) , \quad \sum_{i = 1}^{k}|Q_i| \geq (1 - s)|Q_0|.\]
\end{definition}

This almost looks like the definition of weak porosity given in \eqref{WeakPorosityDefIntrod}. The difference is that $\mc{V}_1(Q_0)$ is replaced by the smaller quantity $\mc{V}_t(Q_0)$. Observe that if $E$ has measure zero, then the median porosity condition can equivalently be written as 
\[\mc{V}_t(Q_0) \leq \delta^{-1}\mc{V}_s(Q_0).\]
Our characterization is
\begin{theorem}\label{thm:BMO_char_dist}
    Let $E$ be a nonempty subset of $\R^n$. Then $E$ is a median porous set if and only if $\dist(\cdot, E)^{-\alpha} \in A_\infty = \bigcup_{1 \leq p < \infty}A_p$ for some $\alpha > 0$ (or equivalently $\log \dist(\cdot, E) \in BMO$). I.e.~for all cubes $Q$,
    \begin{equation}\label{ExpLogCharacterizationOfAInfty}
        \dashint_{Q} \dist(\cdot, E)^{-\alpha} \, dx \lesssim \exp\left(\dashint_Q \log \dist(\cdot, E)^{-\alpha} \, dx\right).
    \end{equation}

\end{theorem}

\begin{figure}[!ht]
\centering

\tikzset{every picture/.style={line width=0.75pt}} 

\tikzset{every picture/.style={line width=0.75pt}} 

\begin{tikzpicture}[x=0.75pt,y=0.75pt,yscale=-1,xscale=1]

\draw  [draw opacity=0] (251,88) -- (412,88) -- (412,249) -- (251,249) -- cycle ; \draw   (251,88) -- (251,249)(291,88) -- (291,249)(331,88) -- (331,249)(371,88) -- (371,249)(411,88) -- (411,249) ; \draw   (251,88) -- (412,88)(251,128) -- (412,128)(251,168) -- (412,168)(251,208) -- (412,208)(251,248) -- (412,248) ; \draw    ;
\draw   (50,88) -- (210,88) -- (210,248) -- (50,248) -- cycle ;
\draw [color={rgb, 255:red, 0; green, 0; blue, 0 }  ,draw opacity=1 ][line width=1.5]    (71,149) .. controls (125,188) and (99,133) .. (139,103) ;
\draw [line width=1.5]    (92,121) .. controls (137,161) and (152,151) .. (192,121) ;
\draw  [fill={rgb, 255:red, 155; green, 155; blue, 155 }  ,fill opacity=0.57 ] (251,208) -- (291,208) -- (291,248) -- (251,248) -- cycle ;
\draw  [fill={rgb, 255:red, 155; green, 155; blue, 155 }  ,fill opacity=0.57 ] (251,168) -- (291,168) -- (291,208) -- (251,208) -- cycle ;
\draw  [fill={rgb, 255:red, 155; green, 155; blue, 155 }  ,fill opacity=0.57 ] (371,168) -- (411,168) -- (411,208) -- (371,208) -- cycle ;
\draw  [fill={rgb, 255:red, 155; green, 155; blue, 155 }  ,fill opacity=0.57 ] (251,88) -- (291,88) -- (291,128) -- (251,128) -- cycle ;
\draw [color={rgb, 255:red, 0; green, 0; blue, 0 }  ,draw opacity=1 ][line width=1.5]    (272,149) .. controls (326,188) and (300,133) .. (340,103) ;
\draw [line width=1.5]    (301,216) .. controls (345,171) and (361,246) .. (401,216) ;
\draw [line width=1.5]    (293,121) .. controls (338,161) and (353,151) .. (393,121) ;
\draw   (450,88) -- (610,88) -- (610,248) -- (450,248) -- cycle ;
\draw [color={rgb, 255:red, 0; green, 0; blue, 0 }  ,draw opacity=1 ][line width=1.5]    (472,149) .. controls (526,188) and (500,133) .. (540,103) ;
\draw [line width=1.5]    (493,121) .. controls (538,161) and (553,151) .. (593,121) ;
\draw  [draw opacity=0] (450,88) -- (610,88) -- (610,248) -- (450,248) -- cycle ; \draw   (450,88) -- (450,248)(470,88) -- (470,248)(490,88) -- (490,248)(510,88) -- (510,248)(530,88) -- (530,248)(550,88) -- (550,248)(570,88) -- (570,248)(590,88) -- (590,248) ; \draw   (450,88) -- (610,88)(450,108) -- (610,108)(450,128) -- (610,128)(450,148) -- (610,148)(450,168) -- (610,168)(450,188) -- (610,188)(450,208) -- (610,208)(450,228) -- (610,228) ; \draw    ;
\draw [line width=1.5]    (100,216) .. controls (144,171) and (160,246) .. (200,216) ;
\draw [line width=1.5]    (501,216) .. controls (545,171) and (561,246) .. (601,216) ;
\draw  [fill={rgb, 255:red, 155; green, 155; blue, 155 }  ,fill opacity=0.57 ] (450,208) -- (470,208) -- (470,228) -- (450,228) -- cycle ;
\draw  [fill={rgb, 255:red, 155; green, 155; blue, 155 }  ,fill opacity=0.57 ] (450,228) -- (470,228) -- (470,248) -- (450,248) -- cycle ;
\draw  [fill={rgb, 255:red, 155; green, 155; blue, 155 }  ,fill opacity=0.57 ] (470,228) -- (490,228) -- (490,248) -- (470,248) -- cycle ;
\draw  [fill={rgb, 255:red, 155; green, 155; blue, 155 }  ,fill opacity=0.57 ] (470,208) -- (490,208) -- (490,228) -- (470,228) -- cycle ;
\draw  [fill={rgb, 255:red, 155; green, 155; blue, 155 }  ,fill opacity=0.57 ] (490,228) -- (510,228) -- (510,248) -- (490,248) -- cycle ;
\draw  [fill={rgb, 255:red, 155; green, 155; blue, 155 }  ,fill opacity=0.57 ] (510,228) -- (530,228) -- (530,248) -- (510,248) -- cycle ;
\draw  [fill={rgb, 255:red, 155; green, 155; blue, 155 }  ,fill opacity=0.57 ] (530,228) -- (550,228) -- (550,248) -- (530,248) -- cycle ;
\draw  [fill={rgb, 255:red, 155; green, 155; blue, 155 }  ,fill opacity=0.57 ] (550,228) -- (570,228) -- (570,248) -- (550,248) -- cycle ;
\draw  [fill={rgb, 255:red, 155; green, 155; blue, 155 }  ,fill opacity=0.57 ] (570,228) -- (590,228) -- (590,248) -- (570,248) -- cycle ;
\draw  [fill={rgb, 255:red, 155; green, 155; blue, 155 }  ,fill opacity=0.57 ] (590,228) -- (610,228) -- (610,248) -- (590,248) -- cycle ;
\draw  [fill={rgb, 255:red, 155; green, 155; blue, 155 }  ,fill opacity=0.57 ] (450,188) -- (470,188) -- (470,208) -- (450,208) -- cycle ;
\draw  [fill={rgb, 255:red, 155; green, 155; blue, 155 }  ,fill opacity=0.57 ] (470,188) -- (490,188) -- (490,208) -- (470,208) -- cycle ;
\draw  [fill={rgb, 255:red, 155; green, 155; blue, 155 }  ,fill opacity=0.57 ] (570,188) -- (590,188) -- (590,208) -- (570,208) -- cycle ;
\draw  [fill={rgb, 255:red, 155; green, 155; blue, 155 }  ,fill opacity=0.57 ] (590,188) -- (610,188) -- (610,208) -- (590,208) -- cycle ;
\draw  [fill={rgb, 255:red, 155; green, 155; blue, 155 }  ,fill opacity=0.57 ] (570,168) -- (590,168) -- (590,188) -- (570,188) -- cycle ;
\draw  [fill={rgb, 255:red, 155; green, 155; blue, 155 }  ,fill opacity=0.57 ] (550,168) -- (570,168) -- (570,188) -- (550,188) -- cycle ;
\draw  [fill={rgb, 255:red, 155; green, 155; blue, 155 }  ,fill opacity=0.57 ] (490,168) -- (510,168) -- (510,188) -- (490,188) -- cycle ;
\draw  [fill={rgb, 255:red, 155; green, 155; blue, 155 }  ,fill opacity=0.57 ] (530,168) -- (550,168) -- (550,188) -- (530,188) -- cycle ;
\draw  [fill={rgb, 255:red, 155; green, 155; blue, 155 }  ,fill opacity=0.57 ] (510,168) -- (530,168) -- (530,188) -- (510,188) -- cycle ;
\draw  [fill={rgb, 255:red, 155; green, 155; blue, 155 }  ,fill opacity=0.57 ] (590,168) -- (610,168) -- (610,188) -- (590,188) -- cycle ;
\draw  [fill={rgb, 255:red, 155; green, 155; blue, 155 }  ,fill opacity=0.57 ] (470,168) -- (490,168) -- (490,188) -- (470,188) -- cycle ;
\draw  [fill={rgb, 255:red, 155; green, 155; blue, 155 }  ,fill opacity=0.57 ] (450,168) -- (470,168) -- (470,188) -- (450,188) -- cycle ;
\draw  [fill={rgb, 255:red, 155; green, 155; blue, 155 }  ,fill opacity=0.57 ] (450,148) -- (470,148) -- (470,168) -- (450,168) -- cycle ;
\draw  [fill={rgb, 255:red, 155; green, 155; blue, 155 }  ,fill opacity=0.57 ] (530,148) -- (550,148) -- (550,168) -- (530,168) -- cycle ;
\draw  [fill={rgb, 255:red, 155; green, 155; blue, 155 }  ,fill opacity=0.57 ] (550,148) -- (570,148) -- (570,168) -- (550,168) -- cycle ;
\draw  [fill={rgb, 255:red, 155; green, 155; blue, 155 }  ,fill opacity=0.57 ] (570,148) -- (590,148) -- (590,168) -- (570,168) -- cycle ;
\draw  [fill={rgb, 255:red, 155; green, 155; blue, 155 }  ,fill opacity=0.57 ] (590,148) -- (610,148) -- (610,168) -- (590,168) -- cycle ;
\draw  [fill={rgb, 255:red, 155; green, 155; blue, 155 }  ,fill opacity=0.57 ] (470,128) -- (490,128) -- (490,148) -- (470,148) -- cycle ;
\draw  [fill={rgb, 255:red, 155; green, 155; blue, 155 }  ,fill opacity=0.57 ] (450,128) -- (470,128) -- (470,148) -- (450,148) -- cycle ;
\draw  [fill={rgb, 255:red, 155; green, 155; blue, 155 }  ,fill opacity=0.57 ] (450,108) -- (470,108) -- (470,128) -- (450,128) -- cycle ;
\draw  [fill={rgb, 255:red, 155; green, 155; blue, 155 }  ,fill opacity=0.57 ] (470,108) -- (490,108) -- (490,128) -- (470,128) -- cycle ;
\draw  [fill={rgb, 255:red, 155; green, 155; blue, 155 }  ,fill opacity=0.57 ] (450,88) -- (470,88) -- (470,108) -- (450,108) -- cycle ;
\draw  [fill={rgb, 255:red, 155; green, 155; blue, 155 }  ,fill opacity=0.57 ] (470,88) -- (490,88) -- (490,108) -- (470,108) -- cycle ;
\draw  [fill={rgb, 255:red, 155; green, 155; blue, 155 }  ,fill opacity=0.57 ] (490,88) -- (510,88) -- (510,108) -- (490,108) -- cycle ;
\draw  [fill={rgb, 255:red, 155; green, 155; blue, 155 }  ,fill opacity=0.57 ] (590,88) -- (610,88) -- (610,108) -- (590,108) -- cycle ;
\draw  [fill={rgb, 255:red, 155; green, 155; blue, 155 }  ,fill opacity=0.57 ] (570,88) -- (590,88) -- (590,108) -- (570,108) -- cycle ;
\draw  [fill={rgb, 255:red, 155; green, 155; blue, 155 }  ,fill opacity=0.57 ] (550,88) -- (570,88) -- (570,108) -- (550,108) -- cycle ;
\draw  [fill={rgb, 255:red, 155; green, 155; blue, 155 }  ,fill opacity=0.57 ] (550,108) -- (570,108) -- (570,128) -- (550,128) -- cycle ;
\draw  [fill={rgb, 255:red, 155; green, 155; blue, 155 }  ,fill opacity=0.57 ] (590,128) -- (610,128) -- (610,148) -- (590,148) -- cycle ;
\draw  [fill={rgb, 255:red, 155; green, 155; blue, 155 }  ,fill opacity=0.57 ] (510,88) -- (530,88) -- (530,108) -- (510,108) -- cycle ;

\draw (52,91) node [anchor=north west][inner sep=0.75pt]   [align=left] {$\displaystyle Q_{0}$};
\draw (146,171) node [anchor=north west][inner sep=0.75pt]   [align=left] {$\displaystyle E$};

\end{tikzpicture}
\caption{A set $E \subset Q_0$ and the procedure to compute $\mc{V}_{s}(Q_0)$ for $s = \frac{3}{4}$ and $s = \frac{3}{8}$. The grey cubes are $E$-free dyadic subcubes of $Q_0$. Therefore $\mc{V}_{\frac{3}{4}}(Q_0) = \frac{1}{16}|Q_0|$ and $\mc{V}_{\frac{3}{8}}(Q_0) = \frac{1}{64}|Q_0|$.}
\label{fig:volume_def}
\end{figure}
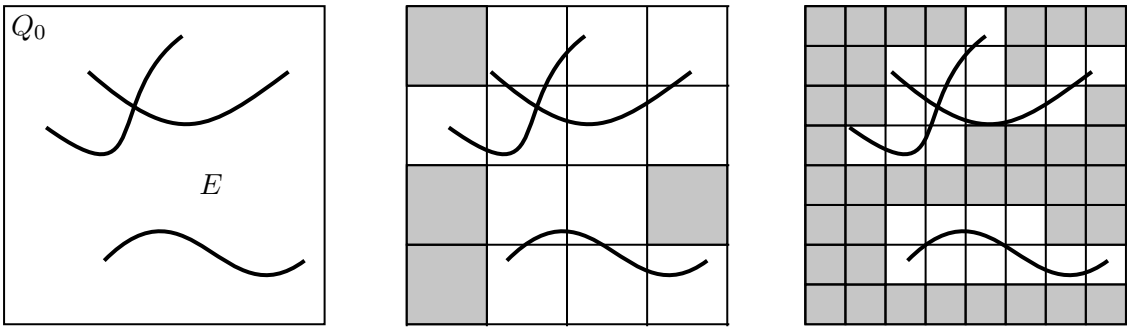

The key step in applying \cref{thm:BMO_char} to prove \cref{thm:BMO_char_dist} is a general statement relating the volume quantities $\mc{V}_s(Q_0)$ with the median values of the distance function (see \cref{prop:median_dist}). This proposition also allows us to interpret the $BLO$ statement of \cref{thm:BMO_char} as a generalization of the ``weakly porous if and only if $A_1$'' equivalence of \cite{anderson2022weakly} to \emph{any} measurable function and not only distance functions. For some concrete applications of this interpretation, see \cref{cor:OurProofOfWeakPorosityResult} for our own (different) proof of the ``weakly porous if and only if $A_1$'' equivalence of \cite{anderson2022weakly} (which generalizes to the BMO case in our proof of \cref{thm:BMO_char_dist}), and \cref{thm:HolderContinuousCase} for an application to H\"older continuous functions.
\par 

It follows from some basic facts about $A_p$ and $A_\infty$ (Proposition $\subref{prop:prelim_BMO_Ap}{prelim_BMO_Ap_4}$ and  Proposition $\subref{prop:prelim_BMO_Ap}{prelim_BMO_Ap_2}$)  that \cref{thm:BMO_char_dist} answers the question: 
\begin{question}\label{que:which_sets}
    \begin{center}
    Given $1 < p \leq \infty$, for which sets $E$ is $\dist(\cdot, E)^{-\alpha} \in A_p$ for some $\alpha > 0$?
\end{center}
\end{question}
Now given a set $E$, one could also ask the question (which already appears in \cite{anderson2022weakly}): 
\begin{question}\label{que:Ap_range}
    \begin{center}
        Given $1 < p \leq \infty$, for which values $\alpha \neq 0$ is $\dist(\cdot, E)^{-\alpha} \in A_p$?
\end{center}
\end{question}
This question has been given partial answers:
\begin{theorem}[\cite{dyda2017muckenhoupt}]
    Let $E$ be a nonempty subset of $\R^n$, $\alpha \neq 0$, and $1 < p < \infty$. Also assume that $E$ is porous. Then
    \begin{enumerate}
        \item $\dist(\cdot, E)^{-\alpha} \in A_1$ if and only if $0 \leq \alpha < \underline{\textup{co\,dim}}_A(E)$.
        \item $\dist(\cdot, E)^{-\alpha} \in A_p$ if and only if $-(p - 1)\underline{\textup{co\,dim}}_A(E) < \alpha < \underline{\textup{co\,dim}}_A(E)$.
    \end{enumerate} 
\end{theorem}
Here $\underline{\textup{co\,dim}}_A(E)$ denotes the Assouad codimension of $E$ (see \cref{sec:Ap_range} for the definition). This theorem gives an answer to \cref{que:Ap_range} under the assumption that $E$ is porous. Recall that this assumption implies $E$ is weakly porous, which in turn by  \cref{thm:weakly_porous_VIntrod}, is equivalent to the existence of some $\beta > 0$ such that $\dist(\cdot, E)^{-\beta} \in A_1$. This result was improved in \cite{anderson2022weakly} where the precise $\alpha$ were determined if $E$ is weakly porous. In fact, they gave a full solution to the $A_1$ question and a partial answer to the $A_p$ question (because it is under the assumption of weak porosity):
\begin{theorem}[\cite{anderson2022weakly}]\label{thm:A_p_weakly_porousVIntrod}
    Let $E$ be a nonempty subset of $\R^n$, $\alpha \neq 0$, and $1 < p < \infty$. Then
    \begin{enumerate}
        \item $\dist(\cdot, E)^{-\alpha} \in A_1$ if and only if $0 < \alpha < \textup{Mu}_1(E)$.
        \item If $E$ is weakly porous then $\dist(\cdot, E)^{-\alpha} \in A_p$ if and only if $-(p - 1)\textup{Mu}_1(E) < \alpha < \textup{Mu}_1(E)$.
    \end{enumerate} 
\end{theorem}
Here  $\text{Mu}_1(E)$, originally defined in \cite{anderson2022weakly} and denoted there as $\text{Mu}(E)$, denotes the ``Muckenhoupt exponent'' (see \cref{sec:Ap_range} for the relevant definitions). Motivated by our qualitative characterization, we are able to fully answer (i.e.~with no side conditions) \cref{que:Ap_range} for all $1 < p \leq \infty$.
\begin{theorem}\label{thm:Ap_range}
    Let $E$ be a nonempty subset of $\R^n$, $\alpha \neq 0$, and $1 < p < \infty$. Then
    \begin{enumerate}
        \item $\dist(\cdot, E)^{-\alpha} \in A_p$ if and only if $-(p - 1)\textup{Mu}_{p'}(E) < \alpha < \textup{Mu}_p(E)$.
        \item $\dist(\cdot, E)^{-\alpha} \in A_\infty$ if and only if $-\sup_{p > 1}(p - 1)\textup{Mu}_{p'}(E) < \alpha < \textup{Mu}_\infty(E)$.
    \end{enumerate}
\end{theorem}
For $1 < p \leq \infty$, $\textup{Mu}_p(E)$ is an ``Assouad dimension''-like quantity defined by terms similar to $\mc{V}_s(Q)$. See \cref{sec:Ap_range} for the precise definitions. \par
As an application of our quantitative characterization, we study, in the range $1 \leq p < q \leq np / (n - p) < \infty$, the Hardy-Sobolev inequalities:
\begin{equation}\label{eqn:HardySobolevIneqIntroduction}
\left(\int_{\R^n}|f|^q \dist(x, E)^{\beta_2} \, dx\right)^{1/q} \leq C\left(\int_{\R^n}|\nabla f|^p \dist(x, E)^{\beta_1} \, dx\right)^{1/p},
\end{equation}
where $\beta_1 \in \R$ and $\frac{q}{p}(n - p + \beta_1) -n = \beta_2$. 

For $\beta_1 \geq 0$, the existing necessary and sufficient conditions for
\cref{eqn:HardySobolevIneqIntroduction} are closely connected with porosity;
see \cite{LehrbackVahakangasInbetweenSobolevandHardy,dyda2017muckenhoupt}.
We focus instead on the case $\beta_1<0$. In this case,
$\beta_2<\beta_1<0$, so both distance weights in
\cref{eqn:HardySobolevIneqIntroduction} are singular near $E$. In the
critical case $q=p^*=np/(n-p)$, the results in
\cref{sec:hardy_sobolev} replace the porosity assumptions in previously
known sufficient and necessary conditions by median porosity. The final
section contains a detailed comparison with the cases $\beta_1=0$ and
$\beta_1>0$.

Moreover, since median porosity is \emph{equivalent} to the distance function $\dist(x, E)^{- \alpha}$ being an $A_p$ weight (for appropriate values of $\alpha$ and $p$ as previously discussed), we seem to have reached the limit of the $A_p$ method for the study of characterizations of weighted Hardy-Sobolev inequalities.

The proofs also identify a limitation of the Riesz potential method. For
nonporous median porous sets, the relevant testing condition holds at the
critical exponent but fails in the subcritical range. Thus this method
does not yield the corresponding subcritical inequalities under median
porosity. See \cref{rem:WeReachedTheLimitOfTheRieszPotentialMethod} for
the precise argument.

\par

\begin{remark}
    A very similar result to \cref{thm:BMO_char_dist} was very recently posted (in fact, it was announced in the arXiv mailing list just the previous day to our posting of the first version of the present manuscript) by Ignacio Gómez Vargas \cite{IgnacioGomezVargas_ApCharacterization2025}. Although the main result of \cite{IgnacioGomezVargas_ApCharacterization2025} is a special case of our main result, \cite{IgnacioGomezVargas_ApCharacterization2025} uses a more geometric argument, and has a probability interpretation for computing certain examples, which is of independent interest. See \cref{sec:history} and the end of \cref{sec:subsets} to contextualize \cite{IgnacioGomezVargas_ApCharacterization2025}.
\end{remark}

\subsection*{Outline of the article} In \cref{sec:prelim} we introduce notation and recall some results about median values and sparse families. In \cref{sec:main_results}, we give an overview of the proof of the median sparse domination (\cref{thm:median_decomp}) and we show how the $BMO$ and $BLO$ characterizations (\cref{thm:BMO_char}) follow. In \cref{sec:history}, we compare \cref{thm:median_decomp} and \cref{thm:BMO_char} to previous known results about median values. In particular, we show how these theorems are improvements of various well-known results in the literature. In \cref{sec:median_proofs}, we give the proof of the median sparse domination result (\cref{thm:median_decomp}). In \cref{sec:dyadic_example}, we show that the $BMO$ characterization fails in the dyadic setting. In \cref{sec:subsets}, we prove the qualitative result (\cref{thm:BMO_char_dist}), and in \cref{sec:Ap_range}, we prove the quantitative result (\cref{thm:Ap_range}) as well as providing the corresponding notation and definitions. In \cref{sec:example} we provide an example of a one-parameter family of sets that are median porous but not weakly porous. Finally, in \cref{sec:hardy_sobolev}, we apply our quantitative results to the study of Hardy-Sobolev inequalities.

\subsection*{Acknowledgments} M. Pasquariello has been partially supported by a UTEA (University of Toronto Excellence Award). I. Uriarte-Tuero has been partially supported by a grant from the National Research Council of Canada.

The authors want to thank the referees for a careful reading of the manuscript and their suggestions to improve it.

\section{Preliminaries}\label{sec:prelim}

\subsection{Basic Definitions}

Here and throughout the paper, the notation $A \lesssim_{x,y,z} B$ means that $A \leq CB$ for a constant $C$ depending only on $x$, $y$, and $z$.

All the cubes we will be discussing are of the form
\[Q = [a_1, b_1) \times \cdots \times [a_n, b_n) \subset \R^n\]
with $l(Q) := b_1 - a_1 = \cdots = b_n - a_n$. For a (not necessarily dyadic)  cube $Q_0 \subset \R^n$, we let $\mc{D}(Q_0)$ denote the collection of dyadic sub-cubes of $Q_0$. That is, 
\[\mc{D}(Q_0) = \bigcup_{j = 0}^{\infty}\mc{D}_j(Q_0)\]
where $\mc{D}_j(Q_0)$ is the unique collection of $2^{jn}$ disjoint subcubes of $Q_0$ such that $l(Q) = 2^{-j}l(Q_0)$ for all $Q \in D_j(Q_0)$. For $Q \in \mc{D}(Q_0)$, we let $\widehat{Q} \in \mc{D}(Q_0)$ denote the dyadic parent of $Q$ and we set $\widehat{Q_0} = Q_0$ by convention. For a cube $Q$ and a positive real number $r$, we let $rQ$ denote the cube with the same center as $Q$, but with side length $rl(Q)$. \par
If $f$ is a locally integrable function, we let 
\[\ang{f}_Q := \dashint_{Q}f(x) \, dx\]
denote the average of $f$ on $Q$. Let $BMO$ be the set of locally integrable functions on $\R^n$ satisfying
\[\|f\|_{BMO} := \sup_{Q}\dashint_{Q}|f(x) - \ang{f}_Q| \, dx < \infty.\]
Let $BLO$ be the set of locally integrable functions on $\R^n$ satisfying
\[\|f\|_{BLO} := \sup_{Q}[\ang{f}_Q - \text{ess inf}_{Q} \, f(x)] < \infty.\]
Let $A_1$ be the set of nonnegative locally integrable functions on $\R^n$ satisfying
\[[w]_{A_1} := \sup_{Q} \left(\dashint_{Q} w(x) \, dx\right)\|w^{-1}\|_{L^\infty(Q)} < \infty.\]
For $1 < p < \infty$, let $A_p$ be the set of nonnegative locally integrable functions on $\R^n$ satisfying
\[[w]_{A_p} := \sup_{Q} \left(\dashint_{Q} w(x) \, dx\right)\left(\dashint_{Q} w(x)^{-1 / (p - 1)} \, dx\right)^{p-1} < \infty.\]
Let $A_\infty$ be the set of nonnegative locally integrable functions on $\R^n$ satisfying
\[[w]_{A_\infty} := \sup_{Q} \left(\dashint_{Q} w(x) \, dx\right)\exp\left(\dashint_{Q} \log w \, dx\right)^{-1} < \infty.\]
The following proposition is well-known (see \cite{garcia1985weighted} for example).

\begin{proposition}\label{prop:prelim_BMO_Ap}
We have that
\begin{enumerate}
    \item \label{prelim_BMO_Ap_1} $\log w \in BLO$ if and only if $w^{\alpha} \in A_1$ for some $\alpha > 0$.
    \item \label{prelim_BMO_Ap_2} $\log w \in BMO$ if and only if $w^{\alpha} \in A_p$ for some $\alpha > 0$ (where $1 < p < \infty$ is given).
    \item \label{prelim_BMO_Ap_3} $\log w \in BMO$ if and only if $w^{\alpha} \in A_\infty$ for some $\alpha > 0$.
    \item \label{prelim_BMO_Ap_4} $A_\infty = \bigcup_{1 \leq p < \infty}A_p$.
\end{enumerate}
\end{proposition}

\subsection{Median Values}

\begin{definition}\label{def:DefUpperSMedianValue}
    Let $f$ be a real-valued measurable function on a cube $Q \subset \R^n$. For $0 \leq s < 1$, we define the upper $s$-median value
    \[M_s(f, Q) := \sup \{\lambda \in \R : |\{x \in Q : f(x) < \lambda\}| \leq s|Q|\},\]
 with the usual understanding that $\sup \emptyset = - \infty$.
    When $s = 1/2$, we denote $M_{1/2}(f, Q)$ simply by $M(f,Q)$.
\end{definition}

\begin{remark}\label{rem:MedianIsRealNumber} If $s>0$, $M_s(f, Q)$ is a real number since there always exists some value $\lambda$ such that 
\[|\{x \in Q : f(x) < \lambda\}| \leq s|Q|.\]
This follows immediately from the fact that
\[\lim_{\lambda \to -\infty} |\{x \in Q : f(x) < \lambda\}| = |\{x \in Q : f(x) = -\infty\}| = 0.\]
\end{remark}

\begin{remark}\label{MedianForSEquals0}
    When $s = 0$, $M_0(f, Q)$ becomes the essential infimum of $f$ over $Q$, $\einf_{Q}f$.
\end{remark}

We list some properties of median values in the following proposition, some of which can be found in \cite{Poelhuis2012}.

\begin{proposition}\label{prop:medians}
    Let $f$ and $g$ be real-valued measurable functions on a cube $Q$ and let $0 \leq s, t < 1$.
    \begin{enumerate}
        \item \label{prop:medians_1} $\lambda = M_s(f, Q)$ satisfies
        \begin{equation}\label{eqn:medians_1}
        |\{x \in Q : f(x) < \lambda\}| \leq s|Q| \text{ and } |\{x \in Q : f(x) > \lambda\}| \leq (1 - s)|Q|.
        \end{equation}
        \item \label{prop:medians_2} For $0 \leq s < t < 1$,
        \[M_s(f, Q) \leq M_t(f, Q).\]
        \item \label{prop:medians_3} For $c \in \R$
        \[M_s(f, Q) + c = M_{s}(f + c, Q).\]
        \item \label{prop:medians_4} If $f \leq g$ a.e., then
        \[M_s(f, Q) \leq M_s(g, Q).\]
        \item \label{prop:medians_5} For $0 < s < 1$,
        \[|M_s(f, Q)| \leq 
        \begin{cases}
            M_{s}(|f|, Q), & \text{ if } M_s(f, Q) \geq 0 \\
            M_{1 - s}(|f|, Q), & \text{ if } M_s(f, Q) \leq 0.
        \end{cases}\]
        \item \label{prop:medians_6} If $f \geq 0$ then
        \[M_s(f, Q) \leq \frac{1}{1 - s} \ang{f}_Q.\]
        \item \label{prop:medians_7} For $0 < s < 1$,
        \[\lim_{Q \ni x, Q \downarrow x}M_s(f, Q) = f(x)\]
        almost everywhere.
        \item \label{prop:medians_8} Suppose $f$ is locally integrable. If $0 < s < t < 1$ then
        \[M_t(f, Q) - M_s(f, Q) \leq  \left(\frac{1}{1 - t} + \frac{1}{s}\right)\dashint_{Q}|f - \ang{f}_Q| \, dx\]
        so in particular if $f \in BMO$ then 
        \[\sup_{Q}[M_t(f, Q) -M_s(f, Q)] \leq \left(\frac{1}{1 - t} + \frac{1}{s}\right)\|f\|_{BMO}.\]
        (The $s=0$ case): 
        Similarly, if $0 < t < 1$ then
        \[M_t(f, Q) - \einf_Q f \leq \frac{1}{1 - t}\dashint_{Q}[f - \einf_Q f] \, dx\]
        so in particular if $f \in BLO$ then
        \[\sup_{Q}[M_t(f, Q) - \einf_{Q}f] \leq \frac{1}{1 - t}\|f\|_{BLO}.\]
        \item \label{prop:medians_9} If $0 < s < 1$ then 
        \[M_s(f, Q) = \inf \{\lambda \in \R : |\{x \in Q : f(x) > \lambda\}| < (1 - s)|Q|\}.\]
        \item \label{prop:medians_10} Let $0 < s < 1$. If $g$ is continuous and strictly increasing then \[g(M_s(f, Q)) = M_s(g \circ f, Q).\]
        If $g$ is continuous and strictly decreasing then \[g(M_s(f, Q)) \leq M_{1 - s}(g \circ f, Q).\] Consequently, for $f > 0$ a.e.~and for all $p \in \R$
        \[M_s(f, Q)^p \lesssim_s \dashint_{Q}f^p \, dx.\]
        \item \label{prop:medians_11} Let $0 < s < 1$.
        If $g$ is continuous and strictly decreasing, then, for any $s' > s$, 
        \[
        g(M_s(f, Q)) \geq M_{1 - s'}(g \circ f, Q).
        \]
        
    \end{enumerate}
\end{proposition}

\begin{proof}
    For \eqref{prop:medians_1}, if $M_s(f, Q) = - \infty$ (which can only happen if $s=0$), there is nothing to prove, so assume wlog that $M_s(f, Q) > - \infty$ and
    let $\lambda < M_s(f, Q)$. We have
    \[|\{x \in Q : f(x) < \lambda\}| \leq s|Q|\]
    so
    \[|\{x \in Q : f(x) < M_s(f, Q)\}| = \lim_{\lambda \to M_s(f, Q)^-}|\{x \in Q : f(x) < \lambda\}| \leq s|Q|.\]
    Also, for any $\lambda > M_s(f, Q)$ we have
    \[|\{x \in Q : f(x) < \lambda\}| > s|Q|.\]
    Thus,
    \[|\{x \in Q : f(x) > M_s(f, Q)\}| = \lim_{\lambda \to  M_s(f, Q)^+}|\{x \in Q : f(x) > \lambda\}| \leq (1 - s)|Q|.\]
    Thus \eqref{eqn:medians_1} holds for $M_s(f, Q)$. \par 
    For \eqref{prop:medians_2}, we have, using \eqref{prop:medians_1} twice,
    \[|\{x \in Q : f(x) < M_s(f, Q)\}| \leq s|Q| < t|Q| \leq |\{x \in Q : f(x) \leq M_t(f, Q)\}|.\]
    It follows that $M_s(f, Q) \leq M_t(f, Q)$. \par
    \eqref{prop:medians_3} and \eqref{prop:medians_4} are trivial. The proofs of \eqref{prop:medians_5}, \eqref{prop:medians_6}, and \eqref{prop:medians_7} can be found in \cite[Prop 1.1 and Thm 2.1]{Poelhuis2012}. \par
    For \eqref{prop:medians_8}, suppose $s > 0$. The $s = 0$ case can be proved similarly (recall Remark \ref{MedianForSEquals0}). Observe that  by \eqref{prop:medians_6},
    \[M_t(f, Q) - \ang{f}_Q = M_{t}(f - \ang{f}_Q, Q) \leq \frac{1}{1 - t}\dashint_{Q}|f - \ang{f}_Q| \, dx\]
     and, by \cref{eqn:medians_1} and \eqref{prop:medians_6} again,
    \[\ang{f}_Q - M_s(f, Q) \leq \ang{f}_Q + M_{1-s}(-f, Q) \leq \frac{1}{s}\dashint_{Q}|f - \ang{f}_Q| \, dx.\]
    Summing the previous two estimates gives
    \[M_{t}(f, Q) - M_s(f, Q) \leq \left(\frac{1}{1 - t} + \frac{1}{s}\right)\dashint_{Q}|f - \ang{f}_Q| \, dx\]
    and the result follows. \par 
    The proof of \eqref{prop:medians_9} is essentially contained in the proof of \cite[Prop. 1.2]{Poelhuis2012}. Simply consider $\overline{\alpha} = \inf \{ \alpha \in \R : |\{ y \in Q : f(y) > \alpha\}| < s |Q| \}$. The first inequality ($M_s(f, Q) \leq \ov{\alpha}$) proved there carries over, and then for the opposite inequality, take $\overline{\alpha} - \varepsilon$ and note that $\{ f \geq \overline{\alpha} - \varepsilon\} \supset \{ f > \overline{\alpha} - \varepsilon\}$ and follow the definitions.
    \par
    For \eqref{prop:medians_10}, if $g$ is strictly increasing then 
    \[|\{x \in Q : f(x) \leq \lambda\}| \leq s|Q| \Leftrightarrow |\{x \in Q : g \circ f(x) \leq g(\lambda)\}| \leq s|Q|\]
    and thus taking the supremum over $\lambda$ gives $g(M_s(f, Q)) = M_{s}(g \circ f, Q)$. Now suppose $g$ is strictly decreasing. By definition,
    \[|\{x \in Q: f(x) < M_s(f, Q) + \epsilon\}| > s|Q|\]
    for each $\epsilon > 0$. Therefore 
    \[|\{x \in Q: g \circ f(x) < g(M_s(f, Q) + \epsilon)\}| < (1 - s)|Q|.\]
    It follows that $g(M_s(f, Q) + \epsilon) \leq M_{1-s}(g \circ f, Q)$. Since $\epsilon$ was arbitrary and $g$ is continuous, the result follows. The final assertions then follow from \eqref{prop:medians_6}. \par
    For \eqref{prop:medians_11}, note that by \eqref{prop:medians_10} 
    \[g(M_s(f, Q)) = -(-g(M_s(f, Q))) = -M_s(-g \circ f, Q).\]
    It then follows from \eqref{prop:medians_9} that 
    \[-M_s(h, Q) \geq M_{1 - s'}(-h, Q)\]
    for any $s' > s$ so by setting $h = -g \circ f$, we have the result.
\end{proof}

\begin{remark}\label{rmk:essinf}
    Proposition $\subref{prop:medians}{prop:medians_7}$ does not hold for $s = 0$. That is, it is not necessarily true that
    $\lim_{Q \ni x, Q \downarrow x}\einf_{Q}f(x) = f(x)$
   almost everywhere. Indeed, let $E \subset [0, 1]^n$ be a closed and nowhere dense set with positive measure. Then $\einf_{Q}\chi_E = 0$ for all cubes $Q$. However, $\chi_E$ is nonzero on a set of positive measure.
\end{remark}

\subsection{Sparse Families}

\begin{definition}
    Let $\mc{S}$ be a family of (not necessarily dyadic) cubes in $\R^n$. 
    \begin{enumerate}
        \item For $0 < \eta < 1$, we say that $\mc{S}$ is $\eta$-sparse if for each $Q \in \mc{S}$, there exists sets $E_Q \subset Q$ with $|E_Q| \geq \eta|Q|$ and the $E_Q$'s are pairwise disjoint.
        \item For $\Lambda > 1$, we say that $\mc{S}$ is $\Lambda$-Carleson if for every sub-collection $\mc{S}' \subset \mc{S}$,
        \[\sum_{Q \in \mc{S}'}|Q| \leq \Lambda \left|\bigcup_{Q \in \mc{S}'}Q\right|.\]
    \end{enumerate}
\end{definition}

\begin{remark}
    To emphasize when a sparse family consists of \emph{dyadic} cubes, we will denote it by $\mc{F}$ instead of $\mc{S}$.
\end{remark}

\begin{remark}\label{rmk:sparse_carleson}
    It was shown in \cite{2017Hanninen} that $\mc{S}$ is $\eta$-sparse if and only if $\mc{S}$ is $\eta^{-1}$-Carleson. This was already well-known to hold for dyadic families. In fact, the equivalence was shown to hold for arbitrary countable collections of Borel sets. In particular, the equivalence gives an easy proof that if $\mc{S}_j$ is $\eta_j$-sparse for each $j = 1, ..., m$, then $\bigcup_{j=1}^{m}\mc{S}_j$ is $(\sum_{j=1}^{m}\eta_j^{-1})^{-1}$-sparse. We will use this fact later.
\end{remark}

\section{Pointwise Sparse Bounds}\label{sec:main_results}

For a constant $0 < a \leq 1/2$ define the local mean oscillation of $f$ on $Q$ by
\[\omega_a(f, Q) := \inf_{c}(\chi_Q (f - c))^*(a |Q|)\]
where $g^*$ denotes the decreasing rearrangement of the function $g$. That is,
\[g^\ast (\lambda) = \inf \{ \alpha >0 : |\{ x \in \R^n : |g(x)| > \alpha \}| \leq \lambda \}\] for $0 < \lambda < \infty$. The following two ``local mean oscillation decompositions'' have been very useful in proving sharp norm estimates for Calderón-Zygmund operators.
\begin{theorem}\cite{Lerner2010}\label{thm:LernerSparseDominationFormula}
Let $f$ be a measurable function on a cube $Q_0$. Then there is an $\eta$-sparse family $\mc{F} = \mc{F}(f, Q_0, n) \subset \mc{D}(Q_0)$ such that
    \[|f(x) - M(f, Q_0)| \leq 4 M^{\sharp}_{1/4,Q_0} f(x)  + 4 \sum_{Q \in \mc{F}}\omega_{a}(f, \widehat{Q})\chi_Q(x)\]
    a.e.~on $Q_0$ where $a = 2^{-n-2}$ and $\eta$ does not depend on $f$ or $Q_0$.
\end{theorem}

Above, $\widehat{Q}$ denotes the dyadic parent of $Q$ and $M^{\sharp}_{1/4,Q_0}$ denotes the local sharp maximal function
relative to $Q_0$, defined by $M^{\sharp}_{\lambda,Q_0} f(x) = \sup_{x \in Q \subset Q_0} \omega_\lambda (f, Q)$.

This seminal result was later refined by Hyt\"onen:

\begin{theorem}\label{thm:local_decomp}\cite{Hytonen2012}
    Let $f$ be a measurable function on a cube $Q_0$. Then there is an $\eta$-sparse family $\mc{F} = \mc{F}(f, Q_0, n) \subset \mc{D}(Q_0)$ such that
    \[|f(x) - M(f, Q_0)| \leq 2 \sum_{Q \in \mc{F}}\omega_{a}(f, Q)\chi_Q(x)\]
    a.e.~on $Q_0$ where $a = 2^{-n-2}$ and $\eta$ doesn't depend on $f$ or $Q_0$.
\end{theorem}

We explain in \cref{sec:history} how \cref{thm:median_decomp} can be seen as an improvement of \cref{thm:local_decomp}. We restate \cref{thm:median_decomp} here. Let $d_{s,t}(f, Q)$ denote the ``median difference'':
\[d_{s, t}(f, Q) := M_t(f, Q) - M_s(f, Q).\]
\begin{theorem}\label{thm:median_decomp2}
    Let $f$ be a real-valued measurable function on a cube $Q_0$ and let $0 < s < t < 1$.  Then there is an $\eta$-sparse family $\mc{S} = \mc{S}(f, Q_0, s, t)$ of (not necessarily dyadic) sub-cubes of $Q_0$ such that
        \[|f(x) - M_s(f, Q_0)| \lesssim_{s,t, n} \sum_{Q \in \mc{S}}d_{s, t}(f, Q)\chi_Q(x)\]
    a.e.~on $Q_0$ where $\eta$ and $C$ depend only on $s,t$, and $n$. 
\end{theorem}
This theorem allows us to completely control the oscillation of a function by its median differences $d_{s,t}(f, Q)$ for fixed $0 < s < t < 1$. \cref{thm:median_decomp} will follow from a dyadic (martingale) analogue, namely \cref{thm:median_decomp_general}. We fix some notation for different measures of oscillations of $f$. For $Q \in \mc{D}(Q_0)$, let
\[\sigma^+_{s,t}(f, Q) := d_{s,t}(f, Q) + (M_s(f, Q) - M_s(f, \widehat{Q}))_+,\]
\[\sigma^-_{s,t}(f, Q) := d_{s,t}(f, Q) + (M_t(f, \widehat{Q}) - M_t(f, Q))_+,\]
\[\sigma_{s,t}(f, Q) := \sigma^+_{s,t}(f, Q) + \sigma^-_{s,t}(f, Q).\]
As the notation suggests, $\sigma^+_{s,t}(f, Q)$ will control the upper oscillations of $f$ and $\sigma^-_{s,t}(f, Q)$ will control the lower oscillations.
\begin{theorem}\label{thm:median_decomp_general}
    Let $f$ be a real-valued measurable function on a cube $Q_0$ and let $0 < s < t < 1$.  Then there is an $\eta$-sparse family $\mc{F} = \mc{F}(f, Q_0, s, t) \subset \mc{D}(Q_0)$ such that
        \[|f(x) - M_s(f, Q_0)| \leq 2\sum_{Q \in \mc{F}}\sigma_{s,t}(f, Q)\chi_Q(x)\]
    a.e.~on $Q_0$ where $\eta$ depends only on $s,t$, and $n$.
\end{theorem}

We postpone the proof of \cref{thm:median_decomp2} and \cref{thm:median_decomp_general} to \cref{sec:median_proofs}.

\begin{remark}
    \cref{thm:median_decomp_general} is valid if we replace the Lebesgue measure with any (not necessarily doubling) locally finite measure $\mu$. However, if we replace the Lebesgue measure with a locally finite measure $\mu$ in \cref{thm:median_decomp}, we must assume that $\mu$ is doubling for the conclusion to hold.
\end{remark}

As an immediate consequence of \cref{thm:median_decomp_general}, we obtain \cref{thm:BMO_char} which we restate here.

\begin{theorem}\label{thm:BMO_char2}
    Let $f$ be a real-valued measurable function on $\R^n$ and fix $0 \leq s < t < 1$.
    \begin{enumerate}
        \item  For $s > 0$, we have
        \[f \in BMO \quad \text{if and only if} \quad \|f\|_{s,t} := \sup_{Q}d_{s,t}(f, Q) < \infty.\]
        Moreover,
        \[\|f\|_{s, t} \lesssim_{s, t} \|f\|_{BMO} \lesssim_{s, t, n}\|f\|_{s,t}.\]
        \item For $s = 0$, we have
        \[f \in BLO \quad \text{if and only if} \quad \|f\|_{0,t} := \sup_{Q}d_{0,t}(f, Q) < \infty.\]
        Moreover,
        \[\|f\|_{0, t} \lesssim_{t} \|f\|_{BLO} \lesssim_{t, n}\|f\|_{0,t}.\]
    \end{enumerate}
\end{theorem}

\begin{remark}
    All of our results have extensions to the case when $f$ is complex-valued. I.e., we can define
    \[M_s(f, Q) := M_s(\text{Re} \, f, Q) + i M_s(\text{Im} \, f, Q)\]
    and then state \cref{thm:median_decomp_general}, \cref{thm:median_decomp}, and \cref{thm:BMO_char} with this definition. However, to simplify notation, we always work with real-valued functions.
\end{remark}

\begin{remark}
    The family $\mc{S}$ given by \cref{thm:median_decomp} is not necessarily dyadic. In fact, the theorem fails for certain $s$ and $t$ if $\mc{S}$ is replaced with a dyadic family. Indeed, assume that $\mc{S} \subset \mc{D}([0, 1])$ is \emph{any} dyadic family, let $0 < s < t < 1/2$, and set $f(x) = \chi_{[1/2, 1]} - \chi_{[0, 1/2]}$. Clearly
    \[M_t(f, Q) - M_s(f, Q) = 0\]
    for any $Q \subsetneq [0, 1]$ since $f$ is constant on any of these intervals. If $Q = [0, 1]$ then $M_s(f, Q) = M_t(f, Q) = -1$ so
    \[M_t(f, [0,1]) - M_s(f, [0,1]) = 0\]
    as well. Thus 
    \[\sum_{Q \in \mc{S}}[M_t(f, Q) - M_s(f, Q)]\chi_Q(x) = 0\]
    for any family $\mc{S} \subset \mc{D}(Q_0)$, let alone a sparse one. The problem is that the median differences don't detect the jump in $f$. Given this, it seems reasonable to conjecture that \cref{thm:BMO_char} fails as well in the dyadic setting. We show later that this is the case (see \cref{sec:dyadic_example}). 
\end{remark}

Now we show that \cref{thm:median_decomp} (restated as \cref{thm:median_decomp2}) implies \cref{thm:BMO_char} (restated as \cref{thm:BMO_char2}).

\begin{proof}[Proof of \cref{thm:BMO_char}]
    We start by proving \eqref{thm:BMO_char_BMO}.  The ``only if'' implication and bound $\|f\|_{s,t} \lesssim_{s,t} \|f\|_{BMO}$ is Proposition \subref{prop:medians}{prop:medians_8}. For the ``if" implication, fix a cube $Q_0$, integrate the inequality given by \cref{thm:median_decomp} over $Q_0$, and use \cref{rmk:sparse_carleson}:
    \[\int_{Q_0}|f - M_s(f, Q_0)| \, dx \lesssim_{s,t,n} \|f\|_{s, t}\sum_{Q \in \mc{S}}|Q| \leq \eta^{-1}\|f\|_{s, t}|Q_0|.\]
    The result follows since $\eta$ depends only on $s, t$, and $n$, using that 
    \[\| f \|_{BMO} \approx \sup_Q \inf_{a \in \C} \dashint_Q |f(x) - a|.\] For \eqref{thm:BMO_char_BLO}, again the ``only if'' implication and bound $\|f\|_{0, t} \lesssim_{t} \|f\|_{BLO}$ follow from Proposition \subref{prop:medians}{prop:medians_8}. For the converse, choose $s_0$ so that $0 < s_0 < t$, e.g. $s_0 = \frac{t}{2}$. Then by Proposition $\subref{prop:medians}{prop:medians_2}$ used twice, and part (1), and recalling that $\einf_{Q_0}f = M_0(f, Q_0)$,
    \begin{align*}
        \int_{Q_0}(f - \einf_{Q_0} f) \, dx & \leq \int_{Q_0}|f - M_{s_0}(f, Q_0)| \, dx + [ M_t(f, Q_0) - \einf_{Q_0}f ]|Q_0| \\
        & \lesssim_{t, n} \|f\|_{s_0, t}|Q_0| + [ M_t(f, Q_0) - \einf_{Q_0}f] |Q_0| \\
        & \lesssim _{t, n} \|f\|_{0, t}|Q_0|
    \end{align*}
    where in the second inequality, we used the conclusion of \eqref{thm:BMO_char_BMO}. Therefore $\|f\|_{BLO} \lesssim_{t, n} \|f\|_{0, t}$.
\end{proof}

\section{Comparison to Previous Work}\label{sec:history}

In this section, we discuss how \cref{thm:median_decomp} is an improvement of \cref{thm:local_decomp} and how \cref{thm:BMO_char} is related to a classical result of John. Consider the local mean oscillation term 
\[\omega_a(f, Q) = \inf_{c}(\chi_Q(f - c))^*(a |Q|)\]
for $a \leq 1/2$. First, using Proposition 1.2 in \cite{Poelhuis2012}, we can rewrite $\omega_a(f, Q)$ in terms of medians:
\begin{equation}\label{eqn:median_equality}
    \omega_a(f, Q) = \inf_{c}M_{1 - a}(|f - c|, Q).
\end{equation}
The following proposition summarizes the relationship between $\omega_a(f, Q)$ and the median differences $d_{s,t}(f, Q)$.
\begin{proposition}\label{prop:median_compare}
    Let $0 < s < t < 1$.
        We have
        \[\omega_{a}(f, Q) \lesssim d_{s, t}(f, Q) \lesssim \omega_{a'}(f, Q)\]
        for $a = 1 - (t - s)$ and $a' = \min(1 - t, s)$.
\end{proposition}

\begin{proof}
    For the first inequality, let $\lambda = d_{s,t}(f, Q) + \epsilon$ for $\epsilon > 0$. We have 
    \begin{align*}
        |\{x \in Q : & |f(x) - M_s(f, Q)|  > \lambda\}| \\ 
        & = |\{x \in Q : f(x) > M_s(f, Q) + \lambda\}| + |\{x \in Q : f(x) < M_s(f, Q) - \lambda\}| \\
        & \leq |\{x \in Q : f(x) > M_t(f, Q) + \epsilon\}| + |\{x \in Q : f(x) < M_s(f, Q)\}| \\
        & < (1 - (t - s))|Q|
    \end{align*}
    where in the last line we used Proposition \subref{prop:medians}{prop:medians_9}. By the same proposition, we have that 
    \[M_{t - s}(|f - M_s(f, Q)|, Q) \leq \lambda = d_{s,t}(f, Q) + \epsilon.\]
    This implies that 
    \[\omega_{a}(f, Q) \leq M_{t - s}(|f - M_s(f, Q)|, Q) \leq d_{s, t}(f, Q)\]
    where the first inequality holds by \cref{eqn:median_equality} as wanted. For the second inequality, we have by Propositions \subref{prop:medians}{prop:medians_3} and \subref{prop:medians}{prop:medians_4} and \subref{prop:medians}{prop:medians_5},
    \begin{align*}
        d_{s,t}(f, Q) & = M_t(f, Q) - M_s(f, Q) \\
        & = M_t(f - c, Q) - M_s(f - c, Q) \\
        & \leq M_t(|f - c|, Q) + M_{1-s}(|f - c|, Q) \\
        & \leq 2M_{\max(t, 1 - s)}(|f - c|, Q)
    \end{align*}
    where the first inequality holds by Proposition \subref{prop:medians}{prop:medians_4} and \subref{prop:medians}{prop:medians_5} and the second one holds by Proposition \subref{prop:medians}{prop:medians_2}. Taking the infimum over all $c$ and applying \eqref{eqn:median_equality} concludes the proof of the proposition.
\end{proof}
Therefore no matter what values $s$ and $t$ we choose, there is some value $a$ such that the median difference is dominated above or below by the local mean oscillation. However, $\omega_a(f, Q)$ is only useful for $a \leq 1/2$ by the next theorem. \par 
The classical results of \cite{john1962} and \cite{stromberg1979} (see also T. Anderson's Ph.D. thesis \cite[pp.30-31]{AndersonPhDThesis}) say that 

\begin{theorem}\label{thm:stromberg}
    Let $f$ be a measurable function and assume that there exists a constant $\lambda$ such that for every cube $Q$, there exists a constant $c_Q$ such that 
    \[|\{x \in Q : |f(x) - c_Q| > \lambda\}| < s|Q|\]
    where $s \leq 1/2$, then $f \in BMO$. If $s > 1/2$ then the conclusion fails.
\end{theorem}
It is clear that the hypothesis of this theorem is equivalent to 
\[\sup_{Q}\omega_a(f, Q) < \infty\]
for $a = s$. Therefore if $t - s > 1/2$, \cref{thm:BMO_char}(1) will follow from \cref{prop:median_compare}. However, \cref{thm:BMO_char} says that $d_{s,t}(f, Q)$ characterizes $BMO$ for \emph{any} $0 < s < t < 1$ which is not implied by \cref{thm:stromberg}. \par 
On the other hand, by \cref{thm:median_decomp}, we have the sparse bound 
\[ |f(x) - M_s(f, Q_0)| \lesssim_{s, t, n} \sum_{Q \in \mc{S}}d_{s,t}(f, Q)\chi_Q(x)\]
for any choice of $0 < s < t < 1$. By \cref{prop:median_compare}, for any  $a < 1/2$, we can choose $s = 1/2$ and $t = 1 - a$ so that $a = \min(1 - t, s)$. 
Then  by the previous sparse bound,
\[|f(x) - M(f, Q_0)| \lesssim_{a, n} \sum_{Q \in \mc{S}}\omega_{a}(f, Q)\chi_Q(x)\] 
which is valid for any choice of $a < 1/2$. This shows that \cref{thm:median_decomp} generalizes \cref{thm:local_decomp} (however, our proof of \cref{thm:median_decomp} is inspired by the proof of \cref{thm:local_decomp}). For the purpose of this paper, \cref{thm:median_decomp} is only used as a step to prove \cref{thm:BMO_char} (i.e.~the case where the $d_{s,t}(f, Q)$ are uniformly bounded). However, due to the importance of \cref{thm:local_decomp} to recent advancements in the subject of weighted norm inequalities, we believe that \cref{thm:median_decomp} can have applications elsewhere.

\begin{remark}\label{rem:AllParametersInMedianFormula}
    At first sight, by \cref{prop:median_compare}, it may look like \cref{thm:median_decomp} is only a mild generalization of \cref{thm:local_decomp}, in that we allow for more values of the parameter $a$, and get two parameters ($s$ and $t$) instead of a single one ($a$), which might not seem like a substantial improvement. 
    
    However, the appearance of the two parameters $s$ and $t$ in \cref{thm:median_decomp} and subsequent theorems, instead of one single parameter $a$, has far reaching consequences down the road. Indeed, having two parameters at our disposal is crucial to be able to define median porous sets (Definition \ref{DefMedianPorousSets}), which is what precisely characterizes when $ \dist(\cdot, E)^{-\alpha} \in A_\infty$ (\cref{thm:BMO_char_dist}).

    Moreover, the fact that the range of both parameters $s$ and $t$ is $(0,1)$ as opposed to a strict subset of $(0,1)$ (as would be the case, e.g.~with the restriction $t - s > 1/2$, when deriving \cref{thm:BMO_char}(1) from \cref{prop:median_compare}, or with the parameter $a = 2^{-n-2}$ in \cref{thm:local_decomp}) will allow for substantial further flexibility when checking that a set is median porous, since one can then choose any value of the parameter $s \in (0,1)$ for that purpose (see \cref{rem:AnyValueOfSWorksForDefinitionOfMedianPorous}). Indeed, it is not hard to see in concrete examples that some values of the parameter $s$ make checking median porosity much easier than other such values (see e.g.~\cref{sec:example}).
    
\end{remark}

The main result of \cite{IgnacioGomezVargas_ApCharacterization2025} can be seen as a subset of our results. More specifically, it would follow from the $t - s > 1/3$ case when $n = 1$ and the $t - s > 1/2$ case when $n > 1$. Therefore, when $n > 1$, it would also follow from the classical result of John (\cite{john1962}) by the reasoning discussed in this section. We will give a detailed comparison at the end of \cref{sec:subsets} when we focus our attention on distance functions. Note however, as mentioned above, that \cite{IgnacioGomezVargas_ApCharacterization2025} uses a more geometric argument for the proof of the main result, and has a probability interpretation for computing certain examples, which is of independent interest.

\section{Proof of Theorems \ref{thm:median_decomp} and \ref{thm:median_decomp_general}}\label{sec:median_proofs}

In this section, we prove Theorem \ref{thm:median_decomp} and its dyadic version, Theorem \ref{thm:median_decomp_general}. To prove \cref{thm:median_decomp} from \cref{thm:median_decomp_general}, we will need some way to compare medians on cubes of different sizes. The following lemma will allow us to do this in terms of the median differences $d_{s,t}(f, Q)$. We remark that the following lemma, somewhat inspired by \cite{anderson2022weakly}, relies heavily on the doubling property of the Lebesgue measure.
\begin{lemma}\label{lem:BMO_char}
    Let $0 \leq s < t < 1$.
    \begin{enumerate}
        \item \label{lem:BMO_char_1} Let $P \subset P'$ be cubes. If $|P| > (\frac{1-t}{1-s})|P'|$ then \[M_s(f, P) \leq M_t(f, P')\]
        and if $|P| \geq (\frac{s}{t})|P'|$ then \[M_s(f, P') \leq M_t(f, P).\]
        \item \label{lem:BMO_char_2}  Let $Q \subset Q'$ be cubes with $|Q'| = 2^n|Q|$ (note $Q$ need not be a dyadic child of $Q'$). Then there is some natural number $k = k(n, s , t)$ and cubes
        \[Q = Q_0 \subset Q_1 \subset \cdots \subset Q_{k-1} \subset Q_k = Q'\]
        such that for $r = \max(\frac{1-t}{1-s}, \frac{s}{t}) < 1$,
        \[|Q_j| < r^{-1}|Q_{j-1}|\]
        for all $1 \leq j \leq k$ and 
        \[\max(M_s(f, Q) - M_s(f, Q'), M_t(f, Q') - M_t(f, Q)) \leq \sum_{j = 1}^{k}[M_t(f, Q_j) - M_s(f, Q_j)].\]
    \end{enumerate}
\end{lemma}

\begin{proof}
    For \eqref{lem:BMO_char_1}, we have
    \begin{align*}
        |\{x \in P : f(x) > M_t(f, P')\}| & \leq |\{x \in P' : f(x) > M_t(f, P')\}| \\
        & \leq (1 - t)|P'| \\
        & < (1-s)|P|.
    \end{align*}
    It follows from Proposition \subref{prop:medians}{prop:medians_9} that $M_s(f, P) \leq M_t(f, P')$. The second assertion is proved in a similar way using Proposition \subref{prop:medians}{prop:medians_1} and \cref{def:DefUpperSMedianValue}. \par
    Now we prove \eqref{lem:BMO_char_2}. Let 
    \[Q = Q_0 \subset \cdots \subset Q_k = Q'\] 
    be a sequence of cubes such that $|Q_{j-1}| > \max(\frac{1-t}{1-s}, \frac{s}{t})|Q_j|$ for each $j$. Note we can take 
    \[k \lesssim_{n} \left(\log \min\left(\frac{1-s}{1-t}, \frac{t}{s}\right)\right)^{-1}.\] 
    Then by \eqref{lem:BMO_char_1}, 
    \[M_s(f, Q) - M_s(f, Q') = \sum_{j = 1}^{k}[M_s(f, Q_{j-1}) - M_s(f, Q_j)] \leq \sum_{j = 1}^{k}[M_t(f, Q_j) - M_s(f, Q_j)]\]
    and similarly 
    \[M_t(f, Q') - M_t(f, Q) \leq \sum_{j = 1}^{k}[M_t(f, Q_j) - M_s(f, Q_j)]\]
    as wanted.
\end{proof}

The idea to prove \cref{thm:median_decomp_general} is to handle the upper oscillation and lower oscillations of $f$ separately. The next proposition is a pointwise sparse bound on the upper oscillations of $f$. A similar result can be proved for the lower oscillations.

\begin{proposition}\label{prop:upper_osc}
    Let $f$ be a real-valued measurable function on a cube $Q_0$ and let $0 < s < t < 1$. Then there is an $\eta^+$-sparse family $\mc{F}^+ \subset \mc{D}(Q_0)$ such that
    \begin{enumerate}
        \item $M_s(f, Q) \geq M_s(f, \widehat{Q})$ for every $Q \in \mc{F}^+$.
        \item \[[f(x) - M_s(f, Q_0)]_+ \leq 2\sum_{Q \in \mc{F}^+}[M_t(f, Q) - M_s(f, \widehat{Q})]\chi_Q(x) = 2\sum_{Q \in \mc{F}^+}\sigma_{s,t}^+(f, Q)\chi_Q(x)\]
        a.e.~on $Q_0$ where $\eta^+$ depends only on $s,t$, and $n$.
    \end{enumerate}
\end{proposition}

\begin{proof}
    We assume that $s, t$, and $f$ are fixed and omit them from the notation so that $\sigma^+(Q) = \sigma^+_{s,t}(f, Q)$. For $Q' \in \mc{D}(Q_0)$, let $\mc{B}_{Q'}$ be the maximal subfamily of 
    \[\{Q \in \mc{D}(Q') : M_s(f, Q) - M_s(f, Q') > \sigma^+(Q')\}.\]
    Set $\mc{B}^0_{Q_0} = \{Q_0\}$, $\mc{B}^1_{Q_0} = \mc{B}_{Q_0}$ and for $k \geq 1$, inductively define
    \[\mc{B}^k_{Q_0} = \bigcup_{Q \in \mc{B}^{k-1}_{Q_0}}\mc{B}_Q.\]
    We claim that \cref{prop:upper_osc} holds with $\mc{F}^+ := \bigcup_{k = 0}^{\infty}\mc{B}^k_{Q_0}$. We write
    \begin{align*}
        f(x) - M_s(f, Q_0) & = \chi_{Q_0 \setminus \bigcup \mc{B}_{Q_0}}(x)[f(x) - M_s(f, Q_0)] \\
        & + \sum_{Q \in \mc{B}_{Q_0}}\chi_Q(x)[M_s(f, Q) - M_s(f, Q_0)] \\
        & + \sum_{Q \in \mc{B}_{Q_0}}\chi_Q(x)[f(x) - M_s(f, Q)]
    \end{align*}
    For $x \in Q_0 \setminus \bigcup \mc{B}_{Q_0}$, we have that 
    \[M_s(f, Q) - M_s(f, Q_0) \leq \sigma^+(Q_0)\] 
    for all dyadic cubes $Q$ containing $x$. It follows from Proposition \subref{prop:medians}{prop:medians_7} that 
    \[f(x) - M_s(f, Q_0) \leq \sigma^+(Q_0)\] 
    a.e.~on $Q_0 \setminus \bigcup \mc{B}_{Q_0}$. Also, for $Q \in \mc{B}_{Q_0}$, by definition of $\sigma^+(Q)$ and the fact that $Q$ is maximal, we have that
    \[M_s(f, Q) - M_s(f, Q_0) \leq \sigma^+(Q) + M_s(f, \widehat{Q}) - M_s(f, Q_0) \leq \sigma^+(Q) + \sigma^+(Q_0).\]
    Therefore
    \begin{align*}
        f(x) - M_s(f, Q_0) & \leq \chi_{Q_0 \setminus \bigcup \mc{B}_{Q_0}}(x)\sigma^+(Q_0) \\
        & + \sum_{Q \in \mc{B}_{Q_0}}\chi_Q(x)[\sigma^+(Q) + \sigma^+(Q_0)] \\
        & + \sum_{Q \in \mc{B}_{Q_0}}\chi_Q(x)[f(x) - M_s(f, Q)] \\
        & = \chi_{Q_0}(x)\sigma^+(Q_0) + \sum_{Q \in \mc{B}_{Q_0}}\chi_Q(x) \sigma^+(Q) + \sum_{Q \in \mc{B}_{Q_0}}\chi_Q(x)[f(x) - M_s(f, Q)].
    \end{align*}
    Applying the same argument inductively to $[f(x) - M_s(f, Q)]$ in the rightmost term yields
    \begin{align*}
        f(x) - M_s(f, Q_0) & \leq 2\sum_{k = 0}^{m-1}\sum_{Q \in \mc{B}^k_{Q_0}}\chi_Q(x) \sigma^+(Q) \\
        & + \sum_{Q \in \mc{B}^m_{Q_0}}\chi_Q(x) \sigma^+(Q)+ \sum_{Q \in \mc{B}^m_{Q_0}}\chi_Q(x)[f(x) - M_s(f, Q)]
    \end{align*}
    for all $m \geq 0$. The factor $2$ appears since each $Q \in \mc{B}^k_{Q_0}$ appears twice: once at stage $k$ and once at stage $k+1$. Next, note that if $Q \in \mc{B}_{Q'}$, then
    \[|\{x \in Q : f(x) \leq \sigma^+(Q') + M_s(f, Q')\}| \leq |\{x \in Q : f(x) < M_s(f, Q)\}| \leq s|Q|.\]
    Then since $\sigma^+(Q') \geq M_t(f, Q') - M_s(f, Q')$,
    \begin{align*}
        \sum_{Q \in \mc{B}_{Q'}}|Q| & \leq \frac{1}{1 - s}\sum_{Q \in \mc{B}_{Q'}}|\{x \in Q : f(x) - M_s(f, Q') > \sigma^+(Q')\}| \\
        & \leq  \frac{1}{1 - s}|\{x \in Q' : f(x) - M_s(f, Q') > M_t(f, Q') - M_s(f, Q')\}| \\
        & \leq \frac{1-t}{1-s}|Q'|.
    \end{align*}
    It follows by induction that
    \[\sum_{Q \in \mc{B}^m_{Q_0}}|Q| \leq \left(\frac{1-t}{1-s}\right)^m|Q_0|.\]
    Therefore by letting $m \to \infty$ in the estimate for $f(x) - M_s(f, Q_0)$, we have that 
    \[f(x) - M_s(f, Q_0) \leq 
    2\sum_{k = 0}^{\infty}\sum_{Q \in \mc{B}^k_{Q_0}}\chi_Q(x) \sigma^+(Q)
    \]
    almost everywhere which is the desired estimate. The family $\mc{F}^+$ is sparse since for $Q \in \mc{F}^+$, we can take $E_Q = Q \setminus \bigcup \mc{B}_Q$ which satisfies $|E_Q| \geq \frac{t - s}{1 - s}|Q|$. Therefore $\mc{F}^+$ is $\eta^+$-sparse with
    \[\eta^+ = \frac{t - s}{1 - s}.\]
    This proves that, in assertion (2), the leftmost term is bounded by the rightmost term. For assertion (1), if $Q \in \mc{F}^+$ with $Q \neq Q_0$, then $Q$ satisfies an estimate of the form 
    \[M_s(f, Q) > M_s(f, Q') + \sigma^+(Q')\]
    where $Q'$ is the cube such that $Q \in \mc{B}_{Q'}$. Since the parent $\widehat{Q}$ satisfies the opposite estimate, we must have $M_s(f, Q) \geq M_s(f, \widehat{Q})$ as wanted.
    Furthermore, the equality in assertion (2) between the rightmost term and the middle term now follows easily, since for $Q \in \mc{F}^+$,
    \begin{align*}
        \sigma^+(Q) & = M_t(f, Q) - M_s(f, Q) + [M_s(f, Q) - M_s(f, \widehat{Q})]_+ \\
        & = M_t(f, Q) - M_s(f, Q) + M_s(f, Q) - M_s(f, \widehat{Q}) \\
        & = M_t(f, Q) - M_s(f, \widehat{Q}).
    \end{align*}
    
\end{proof}
By using a similar argument with the stopping time:
\[M_t(f, Q') - M_t(f, Q) > \sigma^-(Q'),\]
we can obtain a family of $\eta^--$sparse cubes with $\mc{F}^- \subset \mc{D}(Q_0)$ and
\[\eta^- = \frac{t - s}{t},\]
such that $M_t(f, \widehat{Q}) \geq M_t(f, Q)$ for all $Q \in \mc{F}^-$ and
\[[M_t(f, Q_0) - f(x)]_+ \leq 2\sum_{Q \in \mc{F}^-}\sigma_{s,t}^-(f, Q)\chi_Q(x).\]
\begin{proof}[Proof of \cref{thm:median_decomp_general}]
We apply \cref{prop:upper_osc} and the remarks following it to obtain the sparse families $\mc{F}^+$ and $\mc{F}^-$. By \cref{rmk:sparse_carleson}, we see that $\mc{F}:= \mc{F}^+ \cup \mc{F}^-$ is $\eta$-sparse with
\[\eta = \left(\left(\frac{t-s}{1-s}\right)^{-1} + \left(\frac{t-s}{t}\right)^{-1}\right)^{-1} = \frac{t - s}{t - s + 1}.\]
Strictly speaking, the cube $Q_0$ is counted twice in $\mc{F}$, once for $\mc{F}^+$ and once for $\mc{F}^-$, but the reader can easily verify that this does not affect the proof.
Finally, recalling that $\sigma_{s,t}(f, Q) = \sigma_{s,t}^+(f, Q) + \sigma_{s,t}^-(f, Q)$, we have
\begin{align*}
    |f(x) - M_s(f, Q_0)| & = [f(x) - M_s(f, Q_0)]_+ + [M_s(f, Q_0) - f(x)]_+ \\
    & \leq [f(x) - M_s(f, Q_0)]_+ + [M_t(f, Q_0) - f(x)]_+ \\
    & \leq 2\sum_{Q \in \mc{F}^+}\chi_Q(x) \sigma_{s,t}^+(f, Q) + 2\sum_{Q \in \mc{F}^-}\chi_Q(x) \sigma_{s,t}^-(f, Q) \\
    & \leq 2\sum_{Q \in \mc{F}}\chi_Q(x)\sigma_{s,t}(f, Q)
\end{align*}
as wanted.
\end{proof}

We are now ready to prove \cref{thm:median_decomp}, which we had restated as \cref{thm:median_decomp2}.

\begin{proof}[Proof of \cref{thm:median_decomp}]
    We start by applying \cref{thm:median_decomp_general} to obtain the $\eta$-sparse family $\mc{F} \subset \mc{D}(Q_0)$. Set 
    \[r = \max\left(\frac{1-t}{1-s}, \frac{s}{t}\right) < 1.\]
    For a cube $Q \in \mc{F}$ with $Q \neq Q_0$, we apply \cref{lem:BMO_char} to the cube $Q$ and its parent $\widehat{Q}$. We obtain a family of cubes 
    \[Q \subset Q_1 \subset \cdots Q_{k-1} \subset Q_k = \widehat{Q}\]
    such that 
    \[|Q_j| < r^{-1}|Q_{j-1}| < \cdots < r^{-j}|Q|\]
    and
    \[\max(M_s(f, Q) - M_s(f, \widehat{Q}), M_t(f, \widehat{Q}) - M_t(f, Q)) \leq \sum_{j = 1}^{k}[M_t(f, Q_j) - M_s(f, Q_j)].\]
    Therefore we have established 
    \[|f(x) - M_s(f, Q_0)| \leq 4 \sum_{Q \in \mc{F}}\left(d_{s,t}(f, Q) + \sum_{j = 1}^{k}d_{s,t}(f, Q_j)\right)\chi_Q(x).\]
    To conclude the proof of the theorem, we show that 
    \[\{Q, Q_1, ..., Q_k : Q \in \mc{F}\}\]
    is sparse. For $j = 1, ..., k$, set 
    \[
    \mc{S}_j = \{Q_j : Q \in \mc{F}\}, \quad \mc{S}_0 = \mc{F}, \quad \mc{S} = \bigcup_{j = 0}^{k}\mc{S}_j.\]
    We show that each $\mc{S}_j$ is Carleson. Indeed, let $\mc{S}_j' \subset \mc{S}_j$ and write $\mc{S}_j' = \{Q_j : Q \in \mc{F}'\}$ where $\mc{F}' \subset \mc{F}$. We see that 
    \[\sum_{Q_j \in \mc{S}_j'}|Q_j| = \sum_{Q \in \mc{F}'}|Q_j| \leq r^{-j}\sum_{Q \in \mc{F}'}|Q| \leq r^{-j}\eta^{-1}\left|\bigcup_{Q \in \mc{F}'}Q\right| \leq r^{-j}\eta^{-1}\left|\bigcup_{Q_j \in \mc{S}_j'}Q_j\right|\]
    where we used the fact that $\mc{F}$ is $\eta^{-1}$-Carleson (by \cref{rmk:sparse_carleson}) in the second inequality and that $Q \subset Q_j$ in the last inequality. Therefore $\mc{S}_j$ is $r^{-j}\eta^{-1}$-Carleson.
    Then $\mc{S}$ is $\left( \sum_{j=0}^{k} r^{-j}\eta^{-1} \right)$-Carleson, and thus, by \cref{rmk:sparse_carleson} again, $\eta'$-sparse with $\eta'$ depending only on $s, t$, and $n$.
\end{proof}

\section{Theorem \ref{thm:BMO_char} fails in the dyadic setting} \label{sec:dyadic_example}

We showed in \cref{thm:BMO_char} that the quantities $d_{s,t}(f, Q)$ characterize $BMO$ for \emph{any} $0 < s < t < 1$. In this section, we show that (for certain $s$ and $t$), the characterization fails for dyadic $BMO$. Define 
\[\|f\|_{BMO_{\text{d}}} := \sup_{Q \in \mc{D}}\dashint_{Q}|f - \ang{f}_Q| \, dx\]
where $\mc{D}$ is the standard dyadic lattice in $\R^n$ and
\[BMO_{\text{d}} = \{f \in L^1_{loc}(\R^n) : \|f\|_{BMO_{\text{d}}} < \infty\}.\]
It is immediate from Proposition \subref{prop:medians}{prop:medians_8} that if $f \in BMO_d$, then 
\[\sup_{Q \in \mc{D}}[M_t(f, Q) - M_s(f, Q)] \lesssim_{s, t} \|f\|_{BMO_{\text{d}}}\] 
for any $0 < s < t < 1$. On the other hand, if $s$ and $t$ are sufficiently separated, the converse inequality holds. 
Indeed, e.g.~for $n=1$, if $t - s \geq 1/2$, then $\frac{1}{2} > r = \max\left(\frac{1-t}{1-s}, \frac{s}{t}\right)$, and thus the chain of cubes constructed in \cref{lem:BMO_char}(2) consists only of $Q$ and $Q'$, so the sparse family obtained in the proof of \cref{thm:median_decomp} actually ends up being dyadic. In $\R^n$ the analogous condition is $\frac{1}{2^n} > r = \max\left(\frac{1-t}{1-s}, \frac{s}{t}\right)$, which is satisfied if and only if $t - s > (2^n - 1) \max \{s, 1-t\}$.
We prove that the characterization fails on $\R$ for certain values of $s$ and $t$. Of course we can use a similar line of reasoning to find counterexamples in $\R^n$.
\begin{theorem}
    If $0 < s < t < 1$ are such that $1/4 < s < t < 1 / 2$, then on $\R$, the median differences $\{d_{s,t}(f, I)\}_{I \in \mc{D}}$ do not characterize $BMO_{\text{d}}$. In other words, there exists some $f \in L^1_{loc}$ such that 
    \[\sup_{I \in \mc{D}}d_{s,t}(f, I) < \infty\]
    but $f \not \in BMO_{d}$.
\end{theorem}

\begin{proof}
    For $I \in \mc{D}$, if $I_-$ (resp.~$I_+$) denotes the left (resp.~right) child of $I$, let $h_I$ denote the Haar function 
    \[h_I = \chi_{I_-} - \chi_{I_+}.\]
    The key observation is that for such functions, the median differences $d_{s,t}$ are all zero. That is, the median differences do not detect the jump in $f$. Let $I_j = [j, j + 1)$. We define 
    \[f = \sum_{j = 0}^{\infty}(h_{I_{2j}} + (j + 1)h_{I_{2j + 1}}).\]
    Let $I \in \mc{D}$. If $l(I) \leq 1$, then $f\chi_I$ is either constant or a constant multiple of a Haar function. In either case $d_{s,t}(f, I) = 0$. Now suppose $l(I) \geq 2$. Write 
    \[I = \bigcup_{j = m}^{n}(I_{2j} \cup I_{2j + 1}).\]
    Then on $(1/4)$-th of $I$, $f = 1$, on $(1/4)$-th of $I$, $f = -1$, on $(1/4)$-th of $I$, $f \leq -(m + 1)$ and on $(1/4)$-th of $I$, $f \geq m + 1$. Therefore 
    \[d_{s,t}(f, I) = M_t(f, I) - M_s(f, I) = -1 - (-1) = 0.\]
    It follows that 
    \[\sup_{I \in \mc{D}}[M_t(f, I) - M_s(f, I)] = 0 < \infty.\]
    However, if $s' < 1/4$ and $t' > 3/4$, then for $I = [2m, 2m + 2)$,
    \[d_{s',t'}(f, I) = M_{t'}(f, I) - M_{s'}(f, I) = m + 1 - (-(m + 1)) = 2(m + 1).\]
    Since this holds for any $m > 0$, we see that 
    \[\sup_{I \in \mc{D}}[M_{t'}(f, I) - M_{s'}(f, I)] = \infty\]
    and thus $f \not \in BMO_{d}$ by the remarks at the start of the section.
\end{proof}

\section{Applications to Subsets of \texorpdfstring{$\R^n$}{}}\label{sec:subsets}

For the remainder of the paper, we focus our attention on the study of integrability properties of distance functions $\dist(\cdot, E)$ where $E \subset \R^n$. In this section, we will prove \cref{thm:BMO_char_dist}. Set
\[d(x) = d_E(x) := \dist(x, E).\]
We start with some definitions.
\begin{definition}\label{def:DefVolumeQuantities}
    Let $E$ be a nonempty subset of $\R^n$.
    \begin{enumerate}
        \item We say that a set $A$ is $E$-free if $A \cap E = \emptyset$. We say that a cube $Q$ is maximal $E$-free if $Q$ is $E$-free but $2Q$ is not. For a cube $Q_0$ that intersects $E$, let 
        \[\mc{V}_1(E, Q_0) = \mc{V}_1(Q_0)\] 
        be the volume of a maximum sub-cube $Q \in \mc{D}(Q_0)$ that is $E$-free. I.e.,
        \[\mc{V}_1(Q_0) = \max\{|Q| : Q \in \mc{D}(Q_0), Q \cap E = \emptyset\}.\]
        \item For $\delta > 0$, and a cube $Q_0$ that intersects $E$, let $\mc{S}_\delta(Q_0)$ be the maximal subfamily of \[\{Q \in \mc{D}(Q_0) : Q \text{ is $E$-free and } |Q| \geq \delta\}.\]
        \item For $0 < s < 1$ and a cube $Q_0$ that intersects $E$, let \[\mc{V}_s(E, Q_0) = \mc{V}_s(Q_0) = \sup\left\{\delta > 0 : \sum_{Q \in \mc{S}_\delta(Q_0)}|Q| \geq (1 - s)|Q_0|\right\}.\]
    \end{enumerate}
\end{definition}

The following proposition gives the relationship between the median values of $d := \dist(\cdot, E)$ and the volume quantities $\mc{V}_s(Q_0)$. 

\begin{proposition}\label{prop:median_dist}
    Let $E$ be a nonempty subset of $\R^n$ with $E \neq \R^n$. Then for any $0 < s < t < 1$ and cube $Q_0$ that intersects $E$, 
    \begin{equation}\label{eqn:ComparabilityMedianRaisedToNAndVolume}
        M_s(d, Q_0)^n \lesssim_{n} \mc{V}_s(Q_0) \lesssim_{n, s, t} M_t(d, Q_0)^n.
    \end{equation}
    As a consequence, if $\mc{L}_s(Q_0) := \mc{V}_s(Q_0)^{1/n}$, then
    \[\mc{L}_s(Q_0)^{p} \lesssim_{p, n, s} \dashint_{Q_0}\dist(x, E)^{p} \, dx\]
    for all $p \neq 0$.
\end{proposition}

\begin{proof}
    For the first inequality, if $x \in Q_0$ and $\dist(x, E) \geq \lambda$ then $x$ is in some $E$-free cube $Q \in \mc{D}(Q_0)$ with $|Q| \geq (\lambda / \sqrt{n})^n$. Thus
    \[|\{x \in Q_0 : d(x) \geq \lambda\}| \leq \sum_{Q \in \mc{S}_{(\lambda / \sqrt{n})^n}(Q_0)}|Q|.\]
    This implies that
    \[M_s(d, Q_0)^n \leq (\sqrt{n})^{n}\mc{V}_s(Q_0)\]
    which is the first inequality in the statement. \par
    For the second inequality, fix some $\delta > 0$. Let $\eta < 1$ be some small constant that will be fixed later. For $x \in \eta Q$ where $Q \in \mc{S}_\delta(Q_0)$, we have 
    \[\dist(x, E) \geq \frac{(1 - \eta)}{2} l(Q) \geq \frac{(1 - \eta)}{2} \delta^{1/n} > \frac{(1 - \eta)}{(2 + \varepsilon)} \delta^{1/n}, \]
    for any $\varepsilon > 0$, so that for $\delta$ admissible in \cref{def:DefVolumeQuantities}(3) we have that
    \[
    (1-s) |Q_0| \eta^{n} \leq \eta^{n}\sum_{Q \in \mc{S}_\delta(Q_0)}|Q| = \sum_{Q \in \mc{S}_\delta(Q_0)}|\eta Q| \leq |\{x \in Q_0 : d(x) > \frac{(1 - \eta)}{(2 + \varepsilon)}\delta^{1/n}\}|.
    \]
    By Proposition \subref{prop:medians}{prop:medians_9}, this implies that
    \[
    M_{1 - \eta^n(1 - s)}(d, Q_0) > \frac{(1 - \eta)}{(2 + \varepsilon)}\delta^{1/n},
    \]
    which in turn implies that 
    \[\mc{V}_s(Q_0) \leq \left[\frac{2}{(1 - \eta)}\right]^n M_{1 - \eta^n(1 - s)}(d, Q_0)^n.\]
    If we choose $\eta$ so that $1 - \eta^n(1 - s) = t$, then we have the implication on the right. \par 
    For the last assertion, we apply Proposition \subref{prop:medians}{prop:medians_10}. Fix some $0 < s < 1$ and choose some $t$ so that $s < t < 1$ . If $p > 0$, then 
    \[\mc{L}_s(Q_0)^p \lesssim_{p, n, s} M_t(d, Q_0)^p\]
    and if $p < 0$ then 
    \[\mc{L}_s(Q_0)^p \lesssim_{p} M_{s}(d, Q_0)^p.\]
    The result then follows from Proposition \subref{prop:medians}{prop:medians_10}.
\end{proof}

\begin{definition}\label{DefMedianPorousSets}
    Let $E$ be a nonempty subset of $\R^n$ with $E \neq \R^n$.
    \begin{enumerate}
        \item Let $0 < s \leq 1$ and $Q_0$ be a cube that intersects $E$. Define
        \[\mc{S}_{\delta}^s(Q_0) := \mc{S}_{\delta \cdot \mc{V}_s(Q_0)}(Q_0).\] 
        I.e., $\mc{S}_\delta^s(Q_0)$ is the maximal subfamily of 
        \[\{Q \in \mc{D}(Q_0) : Q \text{ is $E$-free and } |Q| \geq \delta \cdot \mc{V}_s(Q_0)\}.\]
        \item We say that $E$ is weakly porous if there are constants $0 < \delta, s < 1$ such that for all cubes $Q_0$ which intersect $E$, 
        \[\sum_{Q \in \mc{S}_{\delta}^1(Q_0)}|Q| \geq (1 - s)|Q_0|.\]
        \item We say that $E$ is a median porous set if there are constants $0 < \delta < 1$ and $0 < s < t \leq 1$ such that for all cubes $Q_0$ which intersect $E$, 
        \[\sum_{Q \in \mc{S}_{\delta}^t(Q_0)}|Q| \geq (1 - s)|Q_0|.\]
    \end{enumerate}
\end{definition}

The definition of a weakly porous set was initially given in \cite{Vasin_LimitSetOfFuchsianGroupAndDynkinsLemma} for $\mathbb{T}$ (the boundary of the unit disk), and in $\R^n$ in \cite{anderson2022weakly}. In \cite{anderson2022weakly}, the following theorem (restated in the introduction of this manuscript as \cref{thm:weakly_porous_VIntrod}) was proved:
\begin{theorem}[\cite{anderson2022weakly}]\label{thm:weakly_porous}
    $E\subset \R^n$ is weakly porous if and only if $\dist(\cdot, E)^{-\alpha} \in A_1$ for some $\alpha > 0$ (or equivalently, by \cref{prop:prelim_BMO_Ap}, $- \log \dist(\cdot, E) \in BLO$).
\end{theorem}

\begin{remark}\label{WhyBLOCorrespondsToT=1AndNot0}
The weak porosity condition for a set $E$ corresponds to the parameter $t = 1$ in Definition \ref{DefMedianPorousSets}, which corresponds to the parameter $s=0$ in \cref{thm:BMO_char}. This is due to the fact that BMO is a vector space, so studying whether $\log \dist(\cdot, E) \in BMO$ is equivalent to whether $- \log \dist(\cdot, E) \in BMO$. But BLO is a cone, so the sign matters, and the introduction of the minus sign reverses the direction of the parameters.
\end{remark}

As mentioned in the Introduction,  \cref{thm:weakly_porous} completely characterizes for which sets $E \subset \R^n$ there exists an $\alpha > 0$ such that $\dist(\cdot, E)^{-\alpha} \in A_1$. However, the corresponding characterization for $A_p$ remained open. Indeed, some partial results were proved in \cite{anderson2022weakly}, \cite{dyda2017muckenhoupt}, and \cite{HoriuchiImbeddingTheoremsForWeightedSobolevSpaces2}, but all these assumed either that the set $E$ was weakly porous in the case of \cite{anderson2022weakly} i.e.~that there exists a $\beta > 0$ such that $\dist(\cdot, E)^{-\beta} \in A_1$, or stronger conditions, namely that the set $E$ is porous, in the cases of \cite{dyda2017muckenhoupt}, and \cite{HoriuchiImbeddingTheoremsForWeightedSobolevSpaces2}. (In \cite{HoriuchiImbeddingTheoremsForWeightedSobolevSpaces2} the explicit assumption was the $P(s)$ property, but that was proven later in \cite{LehrbackVahakangasInbetweenSobolevandHardy} to be equivalent to strictly positive Assouad codimension, and hence to porosity.)

We are now ready to prove \cref{thm:BMO_char_dist} which generalizes the $BLO$ result, namely \cref{thm:weakly_porous}, and, as explained in the Introduction, it completely answers the $A_p$ question, to wit,  \cref{que:which_sets}, when combining it with some basic facts about $A_p$ and $A_\infty$ weights, specifically Proposition $\subref{prop:prelim_BMO_Ap}{prelim_BMO_Ap_4}$ and  Proposition $\subref{prop:prelim_BMO_Ap}{prelim_BMO_Ap_2}$.

\begin{proof}[Proof of \cref{thm:BMO_char_dist}]
    We will assume that $t<1$, i.e.~that the set $E$ is not weakly porous, and, equivalently, that there is no $\alpha>0$ such that $\dist(\cdot, E)^{-\alpha} \in A_1$ (i.e.~that $- \log \dist(\cdot, E) \notin BLO$), since those two properties were proven to be equivalent in \cref{thm:weakly_porous}. See however \cref{cor:OurProofOfWeakPorosityResult}.

    First assume $E$ is a median porous set. The $BMO$ condition in particular implies that 
    \[|E^c \cap Q_0| \geq (1 - s)|Q_0|\]
    for all cubes $Q_0$. Therefore $E$ has measure zero and consequently, $\log d$ is real-valued. We show that 
    \[\sup_{Q}d_{s',t'}(\log d, Q) < \infty\]
    for some $0 < s' < t' < 1$ and it would follow from \cref{thm:BMO_char} that $\log d \in BMO$. Let $Q_0$ be an arbitrary cube. First assume that $Q_0$ intersects $E$. By definition, the $BMO$ condition on $E$ implies that 
\begin{equation}\label{ComparisonV1sAndVtForBMOSets}
    \mc{V}_t(Q_0) \leq \delta^{-1}\mc{V}_s(Q_0).
\end{equation}
    This implies 
    \[M_{t'}(d, Q_0) \lesssim M_{s'}(d, Q_0)\]
    for some $0 < s' < t' < 1$ by \cref{prop:median_dist}. By applying $\log$ to this inequality and using the fact that $\log(M_s(f, Q_0)) = M_{s}(\log f, Q_0)$ (which follows from Proposition \subref{prop:medians}{prop:medians_10}), we then have
    \[M_{t'}(\log d, Q_0) \leq C + M_{s'}(\log d, Q_0).\]
    Therefore 
    \[\sup_{Q, Q \cap E \neq \emptyset}d_{s',t'}(\log d, Q) < \infty.\]
    It remains to show the same for cubes that do not intersect $E$. Suppose now that $Q_0$ doesn't intersect $E$. Let $0 < z < 1$. On $(1-z)^{1/n}Q_0 = Q$, 
    \[\dist(x, E) \approx_{z,n} \max(\dist(Q, E), l(Q))\]
    (by considering the cases $2Q_0 \cap E = \emptyset$ and $2Q_0 \cap E \neq \emptyset$ separately). Since  $(1 - z)^{1/n} Q_0$ has measure $(1- z)|Q_0|$, it follows that 
    \begin{equation}\label{MedianComparableToDistanceToQOrSidelengthOfQ}
        M_{z}(d, Q_0) \approx_{z,n} \max(\dist(Q, E), l(Q)).
    \end{equation}
    In particular, by taking two different values of $z$, $M_{t'}(d, Q_0) \approx_{s',t',n} M_{s'}(d, Q_0)$ and 
    \[\sup_{Q, Q \cap E = \emptyset}d_{s',t'}(\log d, Q) < \infty\]
    follows as before. We conclude that $\log d \in BMO$, and thus, equivalently, by \cref{prop:prelim_BMO_Ap}, that $\dist(\cdot, E)^{-\alpha} \in A_\infty = \bigcup_{1 \leq p < \infty}A_p$ for some $\alpha > 0$. \par
    For the converse, suppose $\log d \in BMO$ and fix some $0 < s < t < 1$. Then by \cref{thm:BMO_char}, 
    \[\sup_{Q}d_{s, t}(\log d, Q) = K < \infty.\]
    Fix a cube $Q_0$ that intersects $E$. Therefore since 
    \[M_{t}(\log d, Q_0) - M_s(\log d, Q_0) \leq K,\]
    using Proposition \subref{prop:medians}{prop:medians_10}, \cref{rem:MedianIsRealNumber}, 
    and the fact that $\log d$ is real-valued a.e., since we assume $\log d \in BMO$, gives 
    \begin{equation}\label{ComparabilityOfMsAndMt}
        0 < M_t(d, Q_0) \leq e^K M_s(d, Q_0).
    \end{equation}
    Then by \cref{prop:median_dist}, 
    \begin{equation}\label{ComparabilityOfVsprimeAndVtprime}
    0 < \mc{V}_{t'}(Q_0) \leq K' \mc{V}_{s'}(Q_0)
    \end{equation}
    for some $0 < s'=s < t' < 1$. Indeed, to verify the strict positivity of $\mc{V}_{t'}(Q_0)$, choose $\varepsilon > 0$ such that $s + \varepsilon = t' < t$. Then \cref{prop:median_dist} yields that $0 < M_{s+ \varepsilon}(d, Q_0) \lesssim \mc{V}_{s+ \varepsilon}(Q_0) \lesssim M_t(d, Q_0)$, where the strict positivity of $M_{s+ \varepsilon}(d, Q_0)$ follows from the same reasoning as the one for $M_t(d, Q_0)$ in \eqref{ComparabilityOfMsAndMt}. 
    In other words, \eqref{ComparabilityOfVsprimeAndVtprime} says that
    \[\sum_{Q \in \mc{S}_{(K')^{-1}}^{t'}(Q_0)}|Q| \geq (1 - s')|Q_0|.\]
    But this is exactly the median porosity condition for $E$.

    Finally, \eqref{ExpLogCharacterizationOfAInfty} is one of the characterizations of an $A_\infty$ weight (see e.g. \cite[p.218, 6.3]{SteinFatBook}).
\end{proof}

While we used \cref{thm:weakly_porous} in our proof of \cref{thm:BMO_char_dist} to deal with the weakly porous case, the same proof of \cref{thm:BMO_char_dist} for the non-weakly porous case, properly understood in the limiting case (the weakly porous case), allows us to recover \cref{thm:weakly_porous}, which then makes \cref{thm:BMO_char_dist} self-contained in this manuscript. We present it next as a corollary of the proof. Recall Remark \ref{WhyBLOCorrespondsToT=1AndNot0}.

\begin{corollary}[Second proof of \cref{thm:weakly_porous}]\label{cor:OurProofOfWeakPorosityResult}
$E\subset \R^n$ is weakly porous if and only if $\dist(\cdot, E)^{-\alpha} \in A_1$ for some $\alpha > 0$ (or equivalently, by \cref{prop:prelim_BMO_Ap}, $- \log \dist(\cdot, E) \in BLO$).
\end{corollary}

\begin{proof}
Note that the equivalent condition to \cref{prop:median_dist} for $s = 0$ is that there exist constants $c_1, c_2 \in \R$ such that
\begin{equation}\label{ComparisonV1AndM0}
c_1 - M_{0}(-\log d, Q_0) \leq \log \mc{V}_1(Q_0)^{\frac{1}{n}} \leq c_2 - M_{0}(-\log d, Q_0) ,
\end{equation}
which is trivially true by \cref{def:DefVolumeQuantities}(1) and Remark \ref{MedianForSEquals0}. 

First assume the set $E$ is weakly porous. Then, by definition, the weak porosity condition on $E$ implies the analogue of \eqref{ComparisonV1sAndVtForBMOSets}, namely that
\begin{equation}\label{ComparisonV1sAndVtForWeaklyPorousSets}
    \mc{V}_1(Q_0) \leq \delta^{-1}\mc{V}_s(Q_0).
\end{equation}
In turn, \cref{prop:median_dist}, yields an $s'$ with  $0 < s < s' < 1$, such that, combining \eqref{ComparisonV1AndM0} with \eqref{ComparisonV1sAndVtForWeaklyPorousSets} and Proposition \subref{prop:medians}{prop:medians_11} for $f = \dist(\cdot, E)$ and $g(x) = - \log x$ (which gives an $s''$ with $0 < s' < s'' < 1$), we get
\begin{equation}\label{ComparisonM0AndM1MinusSPrime}
M_{0}(-\log d, Q_0) \geq - \log \mc{V}_1(Q_0)^{\frac{1}{n}} - C' \geq - \log M_{s'}(d, Q_0) - C \geq  M_{1-s''}(-\log d, Q_0) - C .
\end{equation}
    Therefore 
    \[\sup_{Q, Q \cap E \neq \emptyset}d_{0,1-s''}(-\log d, Q) < \infty.\]
It remains to show the same for cubes that do not intersect $E$. So assume $Q_0 \cap E = \emptyset$. Then, 
\[
\esup_{x \in Q_0} \dist(x, E) \approx \max(\dist(Q, E), l(Q_0)),
\]
which implies that there exist constants $c_1, c_2 \in \R$ such that
\begin{align*}
    c_1 + \max(\log \dist(Q, E), \log l(Q)) & \leq 
- M_{0}(-\log d, Q_0) = \esup_{x \in Q_0} \log \dist(x, E) \\
& \leq c_2 + \max(\log \dist(Q, E), \log l(Q)).
\end{align*}

In other words, there exist constants $c_1, c_2 \in \R$ such that
\begin{equation}\label{M0ComparableToMinMinusLogDistance}
    c_1 + \min(- \log \dist(Q, E), - \log l(Q)) \leq 
 M_{0}(-\log d, Q_0) \leq c_2 + \min(- \log \dist(Q, E), - \log l(Q)).
\end{equation}

On the other hand, noting that \eqref{MedianComparableToDistanceToQOrSidelengthOfQ} also holds true in the current situation by the same proof, taking $ - \log$, and using again Proposition \subref{prop:medians}{prop:medians_11}, we get, for any $0 < z < 1$ that there exist constants $c_1, c_2 \in \R$, and $z'$ with $0 < z < z' < 1$ such that
\begin{equation}\label{M1MinusZComparableToMinMinusLogDistance}
 M_{1-z'}(-\log d, Q_0) \leq c_2 + \min(- \log \dist(Q, E), - \log l(Q)).
\end{equation}

Combining \eqref{M0ComparableToMinMinusLogDistance} with \eqref{M1MinusZComparableToMinMinusLogDistance}, we get that
\[\sup_{Q, Q \cap E = \emptyset}d_{0, 1-s'}(-\log d, Q) < \infty\]
    follows as before. In conclusion, $- \log \dist(\cdot, E) \in BLO$. \par
    The converse also follows exactly the proof of \cref{thm:BMO_char_dist}, with the minor modifications already introduced in the implication just proved. Namely, taking $- \log$, using Proposition \subref{prop:medians}{prop:medians_10} (instead of Proposition \subref{prop:medians}{prop:medians_11}) to get that $- \log M_{1-t}(d, Q_0) \leq  M_{t}(-\log d, Q_0)$, and using \eqref{ComparisonV1AndM0}. Note also that in the present implication, $0 < \mc{V}_{1}(Q_0)$ because $E^c \cap Q_0$ contains an open ball, since otherwise, $\mc{V}_{1}(Q_0) = 0$ and $-\log d \equiv +\infty$ a.e.~in $Q_0$. We leave the details to the reader.
\end{proof}

\begin{remark}
    The reasoning in this section allows us to prove an analogous characterization for non-negative Hölder continuous functions. For a non-negative $\alpha$-Hölder continuous function $w$ and a constant $\eta > 0$, we say that a cube $Q$ is $(w, \eta)$-free if 
    \[\eta \cdot l(Q)^\alpha \leq \inf_{x \in Q}w(x).\]
    Let $\mc{S}_{\delta, \eta}(Q_0)$ be the maximal subfamily of 
    \[\{Q \in \mc{D}(Q_0) : Q \text{ is } (w, \eta) \text{-free and } |Q| \geq \delta\}\]
    and define the quantities $\mc{V}_{s, \eta}(Q_0)$ and collections $\mc{S}_{\delta, \eta}^s(Q_0)$ from $\mc{S}_{\delta, \eta}(Q_0)$ as we did for distance functions. The analogue of \cref{thm:BMO_char_dist} for Hölder continuous functions is the following:
    \begin{theorem}\label{thm:HolderContinuousCase}
        Let $w$ be a non-negative $\alpha$-Hölder continuous function on $\R^n$. Then $\log w \in BMO$ if and only if the following condition is satisfied: There exists constants $0 < s < t < 1$ and $\eta > 0$ such that for all cubes $Q_0$ that are not $(w, \eta)$-free, 
        \[\sum_{Q \in \mc{S}_{\delta, \eta}^t(Q_0)}|Q| \geq (1 - s)|Q_0|.\]
    \end{theorem}

        The proof is almost the same as the proof of \cref{thm:BMO_char_dist}. Indeed, if $w$ satisfies that 
        \[
        |w(z) - w(z')| \leq C_w |z-z'|^\alpha ,
        \]
        then the analogue of \cref{eqn:ComparabilityMedianRaisedToNAndVolume} is that, for $0 < s < 1$ and a cube $Q_0$ which is not $(w, \eta)-$free, 
        \[M_s(w, Q_0)^{\frac{n}{\alpha}} \approx_{n, \eta, \alpha, C_w} \mc{V}_s(Q_0).\]
        (The proof is completely analogous to that of \cref{eqn:ComparabilityMedianRaisedToNAndVolume}.        
        Note that there is no need for $t>s$, since the definition of a $(w, \eta)$-free cube already incorporates a strictly positive lower bound for $w$ in all the cube, so 
        for the analogue of the second inequality in \cref{eqn:ComparabilityMedianRaisedToNAndVolume}, there is no need to shrink the cube.)

        Then, following the proof of \cref{thm:BMO_char_dist}, if we first assume that the condition on cubes $Q_0$ that are not $(w, \eta)$-free holds, the set $E = \{ x \in \R^n : w(x) = 0\}$ has measure zero. Moreover, if $Q_0$ is $(w, \eta)$-free, then $w(x) \approx \inf_{x \in Q_0} w(x)$, by an analogous proof to that of \eqref{MedianComparableToDistanceToQOrSidelengthOfQ}. The rest of the proof of \cref{thm:BMO_char_dist} carries over as well, so we omit the details.
    
\end{remark}

\begin{remark}
Now we take some time to contextualize the main result of \cite{IgnacioGomezVargas_ApCharacterization2025} and explain how it compares to our results. In particular, we show that the main result of \cite{IgnacioGomezVargas_ApCharacterization2025} is a special case of \cref{thm:BMO_char_dist}. Note however that \cite{IgnacioGomezVargas_ApCharacterization2025} uses a more geometric argument, and has a probability interpretation for computing certain examples, which is of independent interest.

As we have discussed in this manuscript several times, we can suppose $E$ is closed without loss of generality. G\'omez Vargas defines (using different notation), for a cube $Q_0 \subset \R^n$, the collection 
\[\mc{D}_E(Q_0) = \{Q \in \mc{D}(Q_0) : Q \text{ maximal } E\text{-free}\}\]
and for $0 < a < 1$, the quantities 
\[\mc{A}_a(Q_0) := \sup \left\{L \geq 0 : \sum_{Q \in \mc{D}_E(Q_0), l(Q) \geq L}|Q| \geq a|Q_0|\right\}\]
and 
\[\mc{B}_a(Q_0) := \inf \left\{L \geq 0 : \sum_{Q \in \mc{D}_E(Q_0), l(Q) \leq L}|Q| \geq a|Q_0|\right\}.\]
The main result is then 
\begin{theorem}[\cite{IgnacioGomezVargas_ApCharacterization2025}]\label{thm:IgnacioThm}
    Let $E \subset \R^n$ be nonempty. Then the following statements are equivalent:
    \begin{enumerate}
        \item \label{eqn:IgnacioCondition} There is some $0 < a < (1 + 2^n)^{-1}$ such that $\mc{A}_a(Q_0) \lesssim \mc{B}_a(Q_0)$ for all cubes $Q_0$.
        \item \label{eqn:IgnacioCondition2} For any $p > 1$, there is some $\alpha \neq 0$ such that $\dist(\cdot, E)^\alpha \in A_p$ (or equivalently $\log \dist(\cdot, E) \in BMO$).
    \end{enumerate}
\end{theorem}
This is another characterization of $A_p$ for distance functions given in terms of similar quantities. Indeed, one can see from the definitions that $\mc{A}_a(Q_0) = \mc{V}_{1 - a}(Q_0)^{1/n}$. Also, as pointed out in \cite{IgnacioGomezVargas_ApCharacterization2025}, $\mc{B}_a(Q_0) \leq \mc{A}_{1 - a}(Q_0)$. Therefore \eqref{eqn:IgnacioCondition} implies in particular that $\mc{V}_{1 - a}(Q_0) \lesssim \mc{V}_{a}(Q_0)$ and thus $E$ is median porous. However, the definition of median porosity allows $\mc{V}_{t}(Q_0) \lesssim \mc{V}_{s}(Q_0)$ for any parameters $0 < s < t < 1$ and not just parameters of the form $s = a$, $t = 1 - a$ with $a$ small.

By setting $t = 1 - a$ and $s = a$, we see that 
\[t - s = 1 - 2a > 1 - 2(1 + 2^n)^{-1}.\]
Now, choose $s'$ so that $0 < s < s' < t < 1$, and $t - s' > 1 - 2(1 + 2^n)^{-1}$. Then by following the reasoning in the proof of \cref{thm:BMO_char_dist}, one would have 
\[M_t(\log d, Q_0) - M_{s'}(\log d, Q_0) \lesssim 1.\]
When $n \geq 2$, then 
\[t - s' > 1 - 2(1 + 2^n)^{-1} \geq \frac{1}{2}\]
and the fact that $\log d \in BMO$ follows from \cref{prop:median_compare} and \cref{thm:stromberg}. So in this case, the new characterization of $BMO$ for all $0 < s < t < 1$ is not needed, and implication $(1) \Rightarrow (2)$ of \cref{thm:IgnacioThm} follows when $n \geq 2$.

If $n=1$, the same reasoning gives the condition
\[
t - s' > \frac{1}{3} ,
\]
which then is not covered by \cref{thm:stromberg}, but is included in \cref{thm:BMO_char_dist}. 

The implication $(2) \Rightarrow (1)$ of \cref{thm:IgnacioThm} is also included in \cref{thm:BMO_char_dist}. 
\end{remark}

\section{The Quantitative Characterization}\label{sec:Ap_range}

In this section, we prove \cref{thm:Ap_range}. More specifically, we find the precise range of $\alpha$ values for which $\dist(\cdot, E)^{-\alpha} \in A_p$ for any given $1 < p \leq \infty$, extending the corresponding result in \cite{anderson2022weakly} for $A_1$. The quantitative characterization will be in terms of an ``Assouad dimension''-like quantity. Recall for a cube $Q$ that intersects $E$ and $0 < s \leq 1$, we defined the length quantity 
\[\mc{L}_s(Q, E) = \mc{L}_s(Q) := \mc{V}_s(Q)^{1/n}.\]
Now we recall the various notions of dimension and also give our own notion of dimension that we will use to answer \cref{que:Ap_range}. For $r > 0$, let
\[E_r = \{x \in \R^n : \dist(x, E) < r\}.\]
\begin{definition}\label{def:AssouadAndMuckenhoupt1Exponent}
    Let $E \subset \R^n$ be nonempty.
    \begin{enumerate}
        \item Let the (lower) Assouad co-dimension $\textup{\underline{co\,dim}}_A(E)$ be the supremum of all $\alpha \geq 0$ such that there is some constant $C$ (which could depend on $\alpha$) such that for all cubes $Q$ that are centered at a point of $E$ and all $0 < r < l(Q)$,
        \[\frac{|E_r \cap Q|}{|Q|} \leq C\left(\frac{l(Q)}{r}\right)^{-\alpha}.\] 
        \item Let the $1$-Muckenhoupt exponent $\textup{Mu}_1(E)$ be the supremum of all $\alpha \geq 0$ such that there is some constant $C$ (which could depend on $\alpha$) such that for all cubes $Q$ that are centered at a point of $E$ and all $0 < r < \mc{L}_1(Q)$,
        \[\frac{|E_r \cap Q|}{|Q|} \leq C\left(\frac{\mc{L}_1(Q)}{r}\right)^{-\alpha}.\]
        If $\mc{L}_1(B) = 0$ for some ball $B$, we set $\textup{Mu}_1(E) = 0$.
    \end{enumerate}
\end{definition}

Recall that the lower Assouad co-dimension and the upper Assouad dimension are related by
\[
    \textup{\underline{co\,dim}}_A(E) = n - \overline{\dim}_A(E),
\]
where $\overline{\dim}_A(E)$ denotes the upper Assouad dimension of $E$.

\begin{remark}
    The definition of the $1$-Muckenhoupt exponent was first introduced in \cite{anderson2022weakly}, where it was called ``Muckenhoupt exponent" and denoted as $\textup{Mu}(E)$.
    It was defined in terms of balls and not cubes but as the reader will easily check, the two definitions agree (comparability constants between radii of balls and side lengths of cubes of comparable size can be absorbed into the constant $C$).
\end{remark}

\begin{remark}\label{rmk:muck_codim}
    If $E$ is porous then $\mc{L}_1(Q) \approx l(Q)$ for all cubes $Q$ so that $\textup{Mu}_1(E) = \text{\underline{co\,dim}}_A(E)$.
\end{remark}

The $1$-Muckenhoupt exponent was used to find the precise range of $\alpha$ values for which $\dist(\cdot, E)^{-\alpha} \in A_1$: 

\begin{theorem}\cite{anderson2022weakly}\label{thm:muck_weakly_porous}
    A nonempty subset $E \subset \R^n$ is weakly porous if and only if $\textup{Mu}_1(E) > 0$.
\end{theorem}

\begin{theorem}\cite{anderson2022weakly}\label{thm:A_p_weakly_porous}
    Let $E$ be a nonempty subset of $\R^n$, $\alpha \neq 0$, and $1 < p < \infty$. Then
    \begin{enumerate}
        \item \label{thm:A_p_weakly_porous_1} $\dist(\cdot, E)^{-\alpha} \in A_1$ if and only if $0 < \alpha < \textup{Mu}_1(E)$.
        \item \label{thm:A_p_weakly_porous_2} If $E$ is weakly porous then $\dist(\cdot, E)^{-\alpha} \in A_p$ if and only if $-(p - 1)\textup{Mu}_1(E) < \alpha < \textup{Mu}_1(E)$.
        \item \label{thm:A_p_weakly_porous_3} If $E$ is weakly porous then $\dist(\cdot, E)^{-\alpha} \in A_\infty$ if and only if $\max(\alpha, 0) < \textup{Mu}_1(E)$.
    \end{enumerate} 
\end{theorem}

Analogous to the dimensional quantities defined above, we give the following definitions.

\begin{definition}\label{def:muck_exponent_p}
Let $E \subset \R^n$ be a nonempty set. 
\begin{enumerate}
    \item For $1 < p < \infty$, let $\text{Mu}_p(E)$ be the supremum of all $\alpha \geq 0$ such that there is some constant $C$ and $0 < s < 1$ (which could depend on $\alpha$) such that for all cubes $Q$ that are centered at a point of $E$,
    \begin{equation*}
    \frac{|E_r \cap Q|}{|Q|} \leq C\left(\frac{\mc{L}_s(Q)}{r}\right)^{-\alpha} \quad \text{for all} \quad r < \mc{L}_s(Q)
    \end{equation*}
    and
    \begin{equation*}
    \frac{|(E_r)^c \cap Q|}{|Q|} \leq C\left(\frac{\mc{L}_s(Q)}{r}\right)^{\frac{\alpha}{p-1}} \quad \text{for all} \quad \mc{L}_s(Q) < r < l(Q).
    \end{equation*}
    \item We also let $\text{Mu}_\infty(E)$ be the supremum of all $\alpha \geq 0$ such that there is some constant $C$ and $0 < s < 1$ (which could depend on $\alpha$) such that for all cubes $Q$ that are centered at a point of $E$,
    \begin{equation*}
        \frac{|E_r \cap Q|}{|Q|} \leq C\left(\frac{\mc{L}_s(Q)}{r}\right)^{-\alpha} \quad \text{for all} \quad r < \mc{L}_s(Q).
    \end{equation*}
\end{enumerate}
If $\mc{L}_s(Q) = 0$ for some $0 < s < 1$ and some cube $Q$, we set $\text{Mu}_p(E) = 0$ and $\text{Mu}_\infty(E) = 0$. 
\end{definition}

We are now ready to prove \cref{thm:Ap_range} which we restate here:

\begin{theorem}\label{thm:Ap_range_2}
    Let $E$ be a nonempty subset of $\R^n$, $\alpha \neq 0$, and $1 < p < \infty$. Then
    \begin{enumerate}
        \item \label{item:Ap_range_Ap} $\dist(\cdot, E)^{-\alpha} \in A_p$ if and only if $-(p - 1)\textup{Mu}_{p'}(E) < \alpha < \textup{Mu}_p(E)$.
        \item \label{item:Ap_range_Ainfty} $\dist(\cdot, E)^{-\alpha} \in A_\infty$ if and only if $-\sup_{p > 1}(p - 1)\textup{Mu}_{p'}(E) < \alpha < \textup{Mu}_\infty(E)$.
    \end{enumerate}
\end{theorem}

\begin{proof}
    Set $w = \dist(\cdot, E)^{-\alpha}$. We start by proving \eqref{item:Ap_range_Ap}. First assume that $\alpha > 0$. \par 
    Suppose that $w \in A_p$. Fix a cube $Q$ and let $r < \mc{L}_s(Q)$. We estimate
    \begin{equation*}
    \dashint_{Q}\dist(y, E)^{-\alpha} \, dy \geq \frac{1}{|Q|}\int_{E_r \cap Q}\dist(y, E)^{-\alpha} \, dy \geq r^{-\alpha}\frac{|E_r \cap Q|}{|Q|}
    \end{equation*}
    and
    \[\mc{L}_s(Q)^\alpha \lesssim \left(\dashint_{Q} \dist(y, E)^{\alpha / (p - 1)} \, dy\right)^{p - 1}\]
    by \cref{prop:median_dist}.
    Therefore
    \begin{equation*}
    \frac{|E_r \cap Q|}{|Q|}\left(\frac{\mc{L}_s(Q)}{r}\right)^\alpha \lesssim C\left(\dashint_{Q}\dist(y, E)^{-\alpha} \, dy\right)\left(\dashint_{Q}\dist(y, E)^{\frac{\alpha}{p-1}} \, dy\right)^{p-1} \lesssim 1.
    \end{equation*}
    On the other hand, by \cref{prop:median_dist},
    \begin{equation*}
        \mc{L}_s(Q)^{-\alpha} \lesssim \dashint_{Q}\dist(y, E)^{-\alpha} \, dy
    \end{equation*}
    and
    \begin{equation*}
    \left(\dashint_{Q}\dist(y, E)^{\frac{\alpha}{p-1}} \, dy\right)^{p-1} \geq \left(\frac{1}{|Q|}\int_{E_r^c \cap Q}\dist(y, E)^{\frac{\alpha}{p-1}} \, dy\right)^{p-1} \geq r^\alpha \left(\frac{|(E_r)^c \cap Q|}{|Q|}\right)^{p-1}.
    \end{equation*}
    Thus for $\mc{L}_s(Q) < r < l(Q)$,  
    \begin{equation*}
    \frac{|(E_r)^c \cap Q|}{|Q|}\left(\frac{\mc{L}_s(Q)}{r}\right)^{\frac{-\alpha}{p-1}} \lesssim C\left(\dashint_{Q}\dist(y, E)^{-\alpha} \, dy\right)^{\frac{1}{p-1}}\left(\dashint_{Q}\dist(y, E)^{\frac{\alpha}{p-1}} \, dy\right) \lesssim 1
    \end{equation*}
    so that $\alpha \leq \text{Mu}_p(E)$. By the self-improvement of $A_p$ weights, there is some $\epsilon > 0$ such that $w^{1+\epsilon} \in A_p$. Following the reasoning above gives 
    \[\alpha < \alpha(1 + \epsilon) \leq \text{Mu}_p(E)\]
    so that $\alpha < \text{Mu}_p(E)$ as wanted. \par
    Conversely, suppose that $\alpha < \text{Mu}_p(E)$. We want to show that $w \in A_p$. Note that if $Q$ is an $E$-free cube, then $\dist(\cdot, E)$ satisfies the $A_1$ condition on $Q$ (see the proof of Theorem 6.5 in \cite{anderson2022weakly}, in particular the case of $P_0$ being an $E$-free cube in \cite[Lemma 5.3]{anderson2022weakly}). We can therefore assume that $Q$ intersects $E$. By enlarging $Q$ by a factor of at most $2$, we may assume that $Q$ is centered at a point of $E$ (recall that \cref{def:muck_exponent_p} uses cubes centered at points of $E$). Since we are assuming $\text{Mu}_p(E) > 0$ we have that $\mc{L}_s(Q) > 0$ for all cubes $Q$. Suppose that $\alpha < \lambda < \text{Mu}_p(E)$. Since $|\ov{E}| = 0$ (which can be proven as in \cite{anderson2022weakly} Page 20, Eq 17), we have that 
    \begin{align}\label{EstimatingApConditionFactor1}
        \int_{Q}\dist(x, E)^{-\alpha} \, dx & = \sum_{- \infty < j \leq \log_2 \mc{L}_s(Q)}\int_{Q \cap (E_{2^j} \setminus E_{2^{j-1}})}\dist(\cdot, E)^{-\alpha} \, dx \\ \nonumber
        & + \int_{Q \cap \{x:\dist(x, E) \geq \mc{L}_s(Q)\}} \dist(x, E)^{-\alpha} \, dx \\ \nonumber
        & \lesssim \sum_{- \infty< j \leq \log_2 \mc{L}_s(Q)}|E_{2^j} \cap Q|2^{-j\alpha} + |Q|\mc{L}_s(Q)^{-\alpha} \\ \nonumber
        & \lesssim |Q| \sum_{- \infty < j \leq \log_2 \mc{L}_s(Q)}\mc{L}_s(Q)^{-\lambda}2^{j(\lambda - \alpha)} + |Q|\mc{L}_s(Q)^{-\alpha} \\ \nonumber
        & \lesssim |Q|\mc{L}_s(Q)^{-\alpha}.
    \end{align}
    Similarly, we can show that 
    \begin{equation}\label{EstimatingApConditionFactor2}
        \dashint_{Q}\dist(x, E)^{\frac{\alpha}{p-1}} \, dx \lesssim \mc{L}_s(Q)^{\frac{\alpha}{p-1}}.
    \end{equation}
    Finally,
    \begin{equation*}
    \left(\dashint_{Q} \dist(x, E)^{-\alpha} \, dx\right)\left(\dashint_{Q} \dist(x, E)^{\frac{\alpha}{p-1}} \, dx\right)^{p-1} \lesssim \mc{L}_s(Q)^{-\alpha}(\mc{L}_s(Q)^{\frac{\alpha}{p-1}})^{p - 1} = 1
    \end{equation*}
    so that $w \in A_p$ as wanted. \par
    If $\alpha < 0$, then the result follows from the $\alpha > 0$ case and the basic fact that $w \in A_p$ if and only if $w^{-1 / (p - 1)} \in A_{p'}$. \par 
    Next we prove \eqref{item:Ap_range_Ainfty}. Suppose that $w \in A_\infty$. Then $w \in A_p$ for some $p > 1$. From part \eqref{item:Ap_range_Ap}, we have that 
    \[-(p - 1)\text{Mu}_{p'}(E) < \alpha < \text{Mu}_p(E).\]
    It is clear from the definition that $\text{Mu}_\infty(E) \geq \text{Mu}_q(E)$ for all $q \geq 1$ and thus 
    \[-\sup_{p > 1}(p - 1)\text{Mu}_{p'}(E) < \alpha < \text{Mu}_{\infty}(E).\]
    as wanted. Conversely, suppose that \[-\sup_{p > 1}(p - 1)\text{Mu}_{p'}(E) < \alpha < \text{Mu}_{\infty}(E).\] 
    We show that $w \in A_\infty$. First assume that $\alpha > 0$. Take $0 < \alpha < \alpha' < \text{Mu}_\infty(E)$. Following the proof of \eqref{item:Ap_range_Ap}, we have that
    \[\dashint_{Q} \dist(x, E)^{-\alpha'} \, dx \lesssim \mc{L}_s(Q)^{-\alpha'}\]
    so by \cref{prop:median_dist}, we have 
    \[\dashint_{Q} d^{-\alpha'} \, dx \lesssim \left(\dashint_Q ({d^{-\alpha'}})^{\alpha/\alpha'} \, dx\right)^{\alpha'/\alpha} = \left(\dashint_Q d^{-\alpha} \, dx\right)^{\alpha'/\alpha}.\]
    Therefore $d^{-\alpha}$ satisfies the reverse Hölder inequality and we can conclude that $w \in A_\infty$. Now assume $\alpha < 0$. Then $\alpha > -\sup_{p > 1}(p - 1)\text{Mu}_{p'}(E)$. We can then choose some $p_0 > 1$ so that $\alpha > -(p_0 - 1)\text{Mu}_{p_0'}(E)$. Then by \eqref{item:Ap_range_Ap}, we can conclude that $w \in A_{p_0} \subset A_\infty$.
\end{proof}

\begin{remark}\label{rem:AnyValueOfSWorksForDefinitionOfMedianPorous}
    As a consequence of the proof of \cref{thm:Ap_range}, we can use \emph{any} $0 < s < 1$ in \cref{def:muck_exponent_p}. Another way to see this is that $E$ is median porous if and only if $\mc{L}_s(Q) \approx \mc{L}_t(Q)$ for any $0 < s < t < 1$. This can be seen as a consequence of \cref{thm:BMO_char} and \cref{prop:median_dist}.\par Analogously, also by \cref{thm:BMO_char} and \cref{prop:median_dist}, and by \eqref{ComparisonV1AndM0}, we get that $E$ is weakly porous if and only if $\mc{L}_s(Q) \approx \mc{L}_1(Q)$ for any $0 < s < 1$. Consequently, \emph{once we know a set $E$ is weakly porous}, in \cref{def:AssouadAndMuckenhoupt1Exponent}(2), any $\mc{L}_s(Q)$ can be used instead of $\mc{L}_1(Q)$ to find the value of $\textup{Mu}_1(E) > 0$. (But be careful: for a given $\alpha > 0$, just checking \cref{def:AssouadAndMuckenhoupt1Exponent}(2) with  $\mc{L}_s(Q)$ for a fixed $0 < s < 1$, instead of $\mc{L}_1(Q)$, simply proves that $E$ is median porous, by \cref{def:muck_exponent_p}(2), not that $E$ is weakly porous.)
\end{remark}

\begin{corollary}
    $\textup{Mu}_\infty(E) > 0$ if and only if $f = \log \dist(\cdot, E) \in BMO$.
\end{corollary}

\begin{proof}
    If $\text{Mu}_\infty(E) > 0$ then for $0 < \alpha < \text{Mu}_\infty(E)$, $\dist(\cdot, E)^{-\alpha} \in A_\infty$ so that $f \in BMO$. If $f \in BMO$ then $\dist(\cdot, E)^{-\alpha} \in A_\infty$ for some $\alpha > 0$ and thus $\text{Mu}_\infty(E) > \alpha > 0$.
\end{proof}

\begin{corollary}\label{PropertiesOfMu_pFunction} Given a set $E$, the following properties hold for $\textup{Mu}_p(E)$ as a function of $p$.

\begin{enumerate}
    \item \label{PropertiesOfMu_pFunction_1} If $E$ is weakly porous, then $\textup{Mu}_1(E) = \textup{Mu}_p(E) = \textup{Mu}_\infty(E) >0$ for all $1 < p < \infty$.

    \item \label{PropertiesOfMu_pFunction_2} $[1, \infty] \ni p \to \textup{Mu}_p(E) \in [0, \infty]$ is non-decreasing.

    \item \label{PropertiesOfMu_pFunction_3} For all $1 < p \leq \infty$, $\textup{Mu}_p(E) = \sup_{1 \leq q < p} \textup{Mu}_q(E)$.

    \item \label{PropertiesOfMu_pFunction_4} $[1, \infty] \ni p \to \textup{Mu}_p(E) \in [0, \infty]$ is left continuous.

    \item \label{PropertiesOfMu_pFunction_5} Given a set $E$, if there exists $1 < p < \infty$ such that $\textup{Mu}_p(E) > 0$, then for all $1 < q \leq \infty$, $\textup{Mu}_q(E) > 0$.

    \item \label{PropertiesOfMu_pFunction_6} $\textup{Mu}_\infty(E) \leq n$. Moreover, if $E$ is bounded and the Minkowski (or box) dimension of $E$ exists (denoted $\dim_M(E)$), then $\textup{Mu}_\infty(E) \leq n - \dim_M(E)$.

\end{enumerate}
\end{corollary}

\begin{proof}

For \eqref{PropertiesOfMu_pFunction_1}, combine the statements of Theorem $\subref{thm:A_p_weakly_porous}{thm:A_p_weakly_porous_2}$ and Theorem $\subref{thm:A_p_weakly_porous}{thm:A_p_weakly_porous_3}$ with those of Theorem $\subref{thm:Ap_range_2}{item:Ap_range_Ap}$ and Theorem $\subref{thm:Ap_range_2}{item:Ap_range_Ainfty}$.

For \eqref{PropertiesOfMu_pFunction_2}, since $\text{Mu}_p(E)$ is the supremum of nonnegative quantities, it is nonnegative. Furthermore, since for $q < p$, $A_q \subset A_p$, for $q \in [1, \infty), p \in (1, \infty]$, then the rest of \eqref{PropertiesOfMu_pFunction_2}, keeping in mind \cref{rem:AnyValueOfSWorksForDefinitionOfMedianPorous}, follows from: for $1 < q < p < \infty$, Theorem $\subref{thm:Ap_range_2}{item:Ap_range_Ap}$; for $1 < q < p = \infty$, Definition \ref{def:muck_exponent_p}; for $q = 1$, from Theorem \ref{thm:muck_weakly_porous} (for the case $\text{Mu}_1(E) =0$) and from \eqref{PropertiesOfMu_pFunction_1}.

For \eqref{PropertiesOfMu_pFunction_3}, the inequality $(\geq)$ follows from \eqref{PropertiesOfMu_pFunction_2}. Regarding the inequality $(\leq)$, if $\text{Mu}_p(E) = 0$, there is nothing to prove by the inequality $(\geq)$. So we may assume that $\text{Mu}_p(E) > 0$. Then for every $0 < \alpha < \text{Mu}_p(E)$, the weight $w_\alpha = \dist(\cdot, E)^{-\alpha} \in A_p$. But for each such $\alpha$, $w_\alpha \in A_q$ for some $q < p$, since $A_p = \bigcup_{q < p}A_q$ for $1 < p \leq \infty$. Thus $\alpha < \text{Mu}_q(E)$. 

\eqref{PropertiesOfMu_pFunction_4} follows from \eqref{PropertiesOfMu_pFunction_2} and \eqref{PropertiesOfMu_pFunction_3}.

\eqref{PropertiesOfMu_pFunction_5} follows from Theorem $\subref{thm:Ap_range_2}{item:Ap_range_Ap}$, since if there exists an $\alpha > 0$ such that $w_\alpha = \dist(\cdot, E)^{-\alpha} \in A_p$, then there exists a $\beta > 0$ such that $w_\beta = \dist(\cdot, E)^{-\beta} \in A_q$, by Proposition $\subref{prop:prelim_BMO_Ap}{prelim_BMO_Ap_2}$.

The first assertion in \eqref{PropertiesOfMu_pFunction_6} 
follows from \cref{def:muck_exponent_p}(2), since, if $Q$ is centered at $x \in E$, and $0 < r \lesssim l(Q)$, then $r^n \lesssim |E_r \cap Q|$. 

Regarding the second assertion in \eqref{PropertiesOfMu_pFunction_6}, we somewhat follow the reasoning in \cite[Remark 6.8]{anderson2022weakly}. If the Minkowski (or box) dimension of $E$ exists, then \cite[Theorem 4.4]{ZubrinicMinkowskiContents} gives that $\int_{E_r}\dist(y, E)^{-\alpha} \, dy < \infty$ if and only if $\alpha < n - \dim_M(E)$. Since $E$ is assumed to be bounded, $\int_{E_r}\dist(y, E)^{-\alpha} \, dy < \infty$ is equivalent to local integrability of $\dist(y, E)^{-\alpha}$, which is a necessary condition for $\dist(y, E)^{-\alpha}$ to be an $A_\infty$ weight. The conclusion follows from Theorem \subref{thm:Ap_range_2}{item:Ap_range_Ainfty}.

\end{proof}

\begin{remark}
Corollary $\subref{PropertiesOfMu_pFunction}{PropertiesOfMu_pFunction_5}$ is false for $q = 1$, since there exists a set $E$ such that $\text{Mu}_1(E) = 0 < \text{Mu}_p(E)$ for some $p>1$ (and hence for all $p>1$ by Corollary $\subref{PropertiesOfMu_pFunction}{PropertiesOfMu_pFunction_5}$), as shown in the example in Section 8 of \cite{anderson2022weakly} or in the example in \cref{sec:example}. In fact, for the set $E \subset \R$ constructed in Section 8 of \cite{anderson2022weakly}, it was shown that $\dist(\cdot, E)^{-\alpha} \in A_p$ for all $0 < \alpha < 1$ and $1 < p \leq \infty$. Therefore, since $\text{Mu}_p(E) \leq n$ for $1 \leq p \leq \infty$, we can conclude that $\text{Mu}_p(E) = 1$ for all $p > 1$ by \cref{thm:Ap_range}.

\end{remark}

We summarize the various porosity conditions and their corresponding properties in the following table (the proof that the prototypical example of median porous set in the table is indeed median porous is precisely the content of \cref{sec:example}): 
\begin{center}
    \begin{tabular}{ |c|c|c|c| } 
    \hline
            & Porous & Weakly porous & Median porous \\ 
     \hline
     \makecell{Dimension \\ characterization} & $\text{\underline{co\,dim}}_{A}(E) > 0$ & $\text{Mu}_1(E) > 0$ & $\text{Mu}_\infty(E) > 0$ \\ 
     \hline
     \makecell{Distance \\ function \\ characterization} & \makecell{$\dashint_{Q}\dist(x, E)^{-\alpha} \, dx \lesssim l(Q)^{-\alpha}$ \\ for some $\alpha > 0$ and all \\ cubes $Q$ that intersect $E$} & \makecell{$\dist(\cdot, E)^{-\alpha} \in A_1$ \\ for some $\alpha > 0$} & \makecell{$\dist(\cdot, E)^{-\alpha} \in A_\infty$ \\ for some $\alpha > 0$} \\ 
     \hline
     \makecell{Prototypical \\ example} & $E \subsetneq \R^n$ a subspace & $\Z$ & $\{\pm m^\gamma : m \in \Z\}$, $0 < \gamma < 1$ \\ 
     \hline
    \end{tabular}
\end{center}

\begin{remark}
    We conclude by remarking that we could equivalently define $\text{Mu}_p(E)$ for $1 < p \leq \infty$ with balls instead of cubes. Indeed, for a ball $B$, we could consider the length quantity 
    \[\mc{L}_s(E, B) = \mc{L}_s(B) :=\] 
    \[\sup \left\{\delta > 0 : \exists E\text{-free, disjoint balls } B_1, ..., B_k \subset B \text{ satisfying } r(B_i) \geq \delta, \sum_{i = 1}^{k}|B_i| \geq (1 - s)|B|\right\}\]
    and replace all the definitions with this quantity in the obvious way. There is a (slightly more complicated) analogue of \cref{prop:median_dist} for balls, and the proof of \cref{thm:Ap_range} then goes through without change. We leave the details to the interested reader. 
\end{remark}

\section{An example of a median porous set}\label{sec:example}

In \cite{anderson2022weakly}, the authors provided examples of a family of weakly porous sets that are not porous, and also an example of a set $E \subset \R$ which is not weakly porous, but such that $\dist(\cdot, E)^{- \alpha} \in A_p\setminus A_1$ for all $0 < \alpha < 1$ and all $1 <p <\infty$ (in the language of this manuscript, a median porous set which is not weakly porous). Here we provide another family of examples of sets that are median porous but are not weakly porous.

Let $0 < \gamma < 1$ be some fixed value and consider the set 
\[E = E_\gamma = \{ \pm m^{\gamma} : m \in \N\} \subset \R.\]
See \cref{fig:set}. In this section, we provide a geometric argument to show that $E$ is median porous but isn't weakly porous. It would follow immediately from \cref{thm:BMO_char_dist} that \[-\log \dist(\cdot, E) \in BMO \setminus BLO.\]
\begin{theorem}\label{thm:example}
    The set $E$ is median porous but not weakly porous. In other words, 
    \[-\log \dist(\cdot, E) \in BMO \setminus BLO.\]
\end{theorem} 

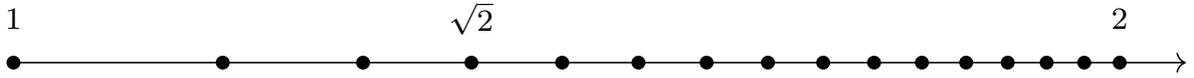
\begin{figure}[!ht]
    \centering
    \resizebox{1\textwidth}{!}{%
    \begin{tikzpicture}[scale=2]
    \tikzstyle{every node}=[font=\SMALL]
    \draw[->] (5,8.5) to[short] (10.3,8.5);
    \node [font=\SMALL] at (5,8.7) {$1$};
    \node at (5,8.5) [circ] {};
    \node at (5.946,8.5) [circ] {};
    \node at (6.58,8.5) [circ] {};
    \node at (7.07,8.5) [circ] {};
    \node [font=\SMALL] at (7.07,8.7) {$\sqrt{2}$};
    \node at (7.48,8.5) [circ] {};
    \node at (7.825,8.5) [circ] {};
    \node at (8.133,8.5) [circ] {};
    \node at (8.41,8.5) [circ] {};
    \node at (8.66,8.5) [circ] {};
    \node at (8.89,8.5) [circ] {};
    \node at (9.106,8.5) [circ] {};
    \node at (9.306,8.5) [circ] {};
    \node at (9.494,8.5) [circ] {};
    \node at (9.67,8.5) [circ] {};
    \node at (9.84,8.5) [circ] {};
    \node at (10,8.5) [circ] {};
    \node [font=\SMALL] at (10,8.7) {$2$};
    \end{tikzpicture}
    }%
    \caption{The set $E_\gamma = \{\pm m^\gamma : m \in \N\}$ for $\gamma = 0.25$.}
    \label{fig:set}
\end{figure}

\begin{remark}
   Let $f$ be an even function on $\R$. To show that $f \in BMO$, it suffices to show that 
   \[\sup_{I \subset [0, \infty)}\dashint_{I}|f - f_I| = C < \infty.\]
   Indeed, suppose this supremum is finite. Given an arbitrary interval $I$, write $I = I^- \cup I^+$ where $I^- \subset (-\infty, 0]$ and $I^+ \subset [0, \infty)$. Without loss of generality, assume $-I^- \subset I^+$. Then since $f$ is even,
   \[\int_{I}|f - \ang{f}_{I_+}| \, dx \leq 2 \int_{I_+}|f - \ang{f}_{I^+}| \, dx \leq 2C|I^+| \leq 2C|I|\]
   and thus $f \in BMO$. As a consequence of this remark, to show that $E$ is median porous, we only need to verify the median porosity condition on subsets of $[0, \infty)$. To see this recall that in the proof of \cref{thm:BMO_char}, the $BMO$ condition on an interval $I$ depends only on the median differences of subsets of $I$. Therefore, by the proof of \cref{thm:BMO_char_dist}, if we show the  median porosity condition on subsets of $[0, \infty)$, it will follow that with $f = \log \dist(\cdot, E)$,
   \[\sup_{I \subset [0, \infty)}\dashint_{I}|f - f_I| = C < \infty.\]
   Since $f$ is even, it will follow that $f \in BMO$ and so $E$ is median porous.
\end{remark}

Instead of computing $\mc{L}_s(I)$, we will compute the more or less equivalent quantities 
\[\sq{\mc{L}}_s(E, I) = \sq{\mc{L}}_s(I) :=\] 
    \[\sup \left\{\delta > 0 : \exists E\text{-free, disjoint open intervals } I_1, ..., I_k \subset I \text{ s.t. } |I_j| \geq \delta, \sum_{j = 1}^{k}|I_j| \geq (1 - s)|I|\right\}\]
for $0 < s < 1$ and 
\[\sq{\mc{L}}_1(E, I) = \sq{\mc{L}}_1(I) = \sup\{|\sq{I}| : \sq{I} \subset I \text{ is an } E\text{-free open interval}\}.\] 
The only difference between $\mc{L}_s(I)$ and $\sq{\mc{L}}_s(I)$ is that we allow arbitrary sub-intervals of $I$ instead of just dyadic ones and the subintervals are open. This will be more convenient to work with because of the non-dyadic structure of $E$. Note that we have an analogue of \cref{prop:median_dist} for $\sq{\mc{L}}$:
\begin{proposition}\label{prop:median_dist_intervals}
    Let $E$ be a nonempty subset of $\R$ with $E \neq \R$. Then for any $0 < s < t < 1$ and interval $I_0$ that intersects $E$, 
    \[M_s(d, I_0) \lesssim \sq{\mc{L}}_s(I_0) \lesssim_{s, t} M_t(d, I_0).\]
\end{proposition}

\begin{proof}
    The proof is almost identical to the proof of \cref{prop:median_dist} so we omit the details.
\end{proof}

Therefore to show that $\log \dist(\cdot, E_\gamma) \in BMO$, it suffices to show that there are some $0 < s < t < 1$ such that
\begin{equation}\label{eqn:example_hypothesis}
    \sq{\mc{L}}_t(I) \lesssim \sq{\mc{L}}_s(I)
\end{equation}
for all open intervals $I$ that intersect $E$.\par 
     
Before proving \cref{thm:example}, we introduce some notation. 
For an interval $I = (a, b) \subset [0, \infty)$, we define
\[n_l(I) := \min\{n \in \N : n^\gamma \geq a\}, \quad n_r(I) := \max\{n \in \N : n^\gamma \leq b\}.\]
For each $k \geq 0$, define the standard $E$-free intervals
\[\mc{I}_k := (k^\gamma, (k + 1)^\gamma).\]
It follows from the mean value theorem that for all $0 < x < y$, we have 
\begin{equation}\label{eqn:mean_value_thm}
    \gamma y^{\gamma - 1}|x - y| \leq |x^\gamma - y^\gamma| \leq \gamma x^{\gamma - 1}|x - y|.
\end{equation}
From this fact, it follows in particular that
\begin{equation}\label{eqn:standard_length}
    |\mc{I}_k| \approx_\gamma (k+1)^{\gamma - 1}.
\end{equation}
For $0 < s < 1$, we say that an interval $I = (a, b) \subset [0, \infty)$ is $s$-good if 
\[a \in E \quad \text{and} \quad a + (1 - s)(b - a) \in E.\]
See \cref{fig:good_interval}(i). \par
The goal is to estimate $\sq{\mc{L}}_s(I)$ for all intervals $I \subset (0, \infty)$. We do this for three types of intervals in the following three lemmas. \par
It turns out that $\sq{\mc{L}}_s(I)$ is quite easy to compute for $s$-good intervals.
\begin{lemma}\label{lem:example_good}
    Let $I_0 = (a, b)$ be an $s$-good interval. Then 
    \[\sq{\mc{L}}_s(I_0) \approx_\gamma (a + (1 - s)(b - a))^{(\gamma - 1) / \gamma}.\]
\end{lemma}

\begin{proof}
    We first show that $\sq{\mc{L}}_s(I_0) \gtrsim_\gamma (a + (1 - s)(b - a))^{(\gamma - 1) / \gamma}$. Set 
    \[\delta_0 = C_\gamma (a + (1 - s)(b - a))^{(\gamma - 1) / \gamma}\]
    where $C_\gamma$ will be fixed later. We find disjoint $E$-free intervals $\{I_k\}$ such that $|I_k| \geq \delta_0$ and $\sum|I_k| \geq (1 - s)|I_0|$. Since $I_0$ is an $s$-good interval, $n_0 = a^{1/\gamma}$ and $n_1 = (a + (1 - s)(b - a))^{1/\gamma}$ are natural numbers. We consider the standard intervals $\{\mc{I}_{k}\}_{k = n_0}^{n_1 - 1}$ (see \cref{fig:good_interval}(ii)). Then by \cref{eqn:standard_length}, 
    \[|\mc{I}_k| \gtrsim_\gamma (k + 1)^{\gamma - 1} \geq  n_1^{\gamma - 1}.\]
    We can now choose $C_\gamma$ small enough so that $|\mc{I}_k| \geq \delta_0$ for such $n_0 \leq k \leq n_1 - 1$. We have
    \[\sum_{k = n_0}^{n_1 - 1}|\mc{I}_k| = a + (1 - s)(b - a) - a = (1 - s)|I_0|.\]
    Therefore 
    \[\sq{\mc{L}}_s(I_0) \geq \delta_0 \gtrsim_\gamma (a + (1 - s)(b - a))^{(\gamma - 1) / \gamma}.\]
    We now show that 
    \[\sq{\mc{L}}_s(I_0) \lesssim (a + (1 - s)(b - a))^{(\gamma - 1) / \gamma}.\]
    Suppose that there are $E$-free disjoint intervals $\{I_k\}$ such that $|I_k| \geq \delta$ and $\sum |I_k| \geq (1 - s)|I_0|$. By enlarging and combining the $I_k$, we may assume that they are of the form 
    \[I_k = \mc{I}_k\]
    for $k = n_0, ..., \sq{n}_1 - 1$ for some $\sq{n}_1 \in \N$. From the same computation as above, 
    \[\sum_{k = n_0}^{n_1 - 1}|\mc{I}_k| = (1 - s)|I_0| \leq \sum_{k = n_0}^{\sq{n}_1 - 1}|\mc{I}_k|\]
    which implies that $\sq{n}_1 \geq n_1$. Similarly to above, we estimate,
    \[\delta \leq |\mc{I}_{\sq{n}_1 - 1}| \lesssim_\gamma \sq{n}_1^{\gamma - 1} \leq n_1^{\gamma - 1} = (a + (1 - s)(b - a))^{(\gamma - 1) / \gamma}.\]
    Taking the supremum over all $\delta$ gives 
    \[\sq{\mc{L}}_s(I_0) \lesssim_\gamma (a + (1 - s)(b - a))^{(\gamma - 1) / \gamma}\]
    which concludes the proof of the lemma.
\end{proof}

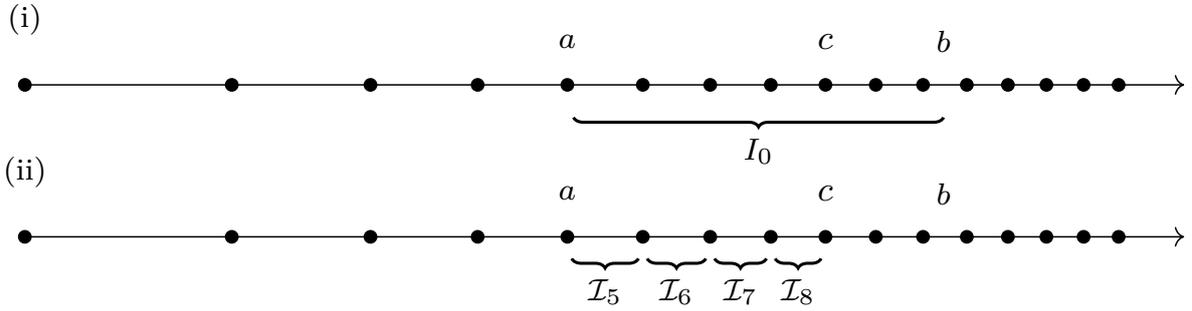
\begin{figure}[!ht]
    \centering
    \resizebox{1\textwidth}{!}{%
    \begin{tikzpicture}[scale=2]
    \tikzstyle{every node}=[font=\small]
    \draw[->] (5,8.5) to[short] (10.3,8.5);
    \node [font=\SMALL] at (5,8.8) {(i)};
    \node at (5,8.5) [circ] {};
    \node at (5.946,8.5) [circ] {};
    \node at (6.58,8.5) [circ] {};
    \node at (7.07,8.5) [circ] {};
    \node (pt1) at (7.48,8.5) [circ] {};
    \node [font=\SMALL] at (7.48,8.7) {$a$};
    \node at (7.825,8.5) [circ] {};
    \node at (8.133,8.5) [circ] {};
    \node at (8.41,8.5) [circ] {};
    \node at (8.66,8.5) [circ] {};
    \node [font=\Small] at (8.66,8.7) {$c$};
    \node at (8.89,8.5) [circ] {};
    \node at (9.106,8.5) [circ] {};
    \node (pt2) at (9.306,8.5) [circ] {};
    \node [font=\SMALL] at (9.2,8.7) {$b$};
    \node at (9.494,8.5) [circ] {};
    \node at (9.67,8.5) [circ] {};
    \node at (9.84,8.5) [circ] {};
    \node at (10,8.5) [circ] {};
    \draw [
    thick,
    decoration={
        brace,
        mirror,
        raise=0.3cm
    },
    decorate
    ] (pt1) -- (9.2,8.5)
    node [pos=0.5,anchor=north,yshift=-0.35cm,font=\SMALL] {$I_0$};

    \draw[->] (5,7.8) to[short] (10.3,7.8);
    \node [font=\SMALL] at (5,8.1) {(ii)};
    \node at (5,7.8) [circ] {};
    \node at (5.946,7.8) [circ] {};
    \node at (6.58,7.8) [circ] {};
    \node at (7.07,7.8) [circ] {};
    \node (pt1) at (7.48,7.8) [circ] {};
    \node [font=\SMALL] at (7.48,8) {$a$};
    \node at (7.825,7.8) [circ] {};
    \node at (8.133,7.8) [circ] {};
    \node at (8.41,7.8) [circ] {};
    \node at (8.66,7.8) [circ] {};
    \node [font=\Small] at (8.66,8) {$c$};
    \node at (8.89,7.8) [circ] {};
    \node at (9.106,7.8) [circ] {};
    \node (pt2) at (9.306,7.8) [circ] {};
    \node [font=\SMALL] at (9.2,8) {$b$};
    \node at (9.494,7.8) [circ] {};
    \node at (9.67,7.8) [circ] {};
    \node at (9.84,7.8) [circ] {};
    \node at (10,7.8) [circ] {};
    
    \draw [
    thick,
    decoration={
        brace,
        mirror,
        raise=0.2cm
    },
    decorate
    ] (7.5,7.8) -- (7.805,7.8)
    node [pos=0.5,anchor=north,yshift=-0.25cm,font=\SMALL] {$\mc{I}_{5}$};

    \draw [
    thick,
    decoration={
        brace,
        mirror,
        raise=0.2cm
    },
    decorate
    ] (7.845,7.8) -- (8.113,7.8)
    node [pos=0.5,anchor=north,yshift=-0.25cm,font=\SMALL] {$\mc{I}_{6}$};

    \draw [
    thick,
    decoration={
        brace,
        mirror,
        raise=0.2cm
    },
    decorate
    ] (8.153,7.8) -- (8.39,7.8)
    node [pos=0.5,anchor=north,yshift=-0.25cm,font=\SMALL] {$\mc{I}_{7}$};

    \draw [
    thick,
    decoration={
        brace,
        mirror,
        raise=0.2cm
    },
    decorate
    ] (8.43,7.8) -- (8.64,7.8)
    node [pos=0.5,anchor=north,yshift=-0.25cm,font=\SMALL] {$\mc{I}_{8}$};

    \end{tikzpicture}
    }%
    \caption{(i) A good interval $I_0 = (a, b)$ where $c = a + (1 - s)(b - a)$. (ii) The choice of intervals $\{\mc{I}_k\}$ such that $\sum |\mc{I}_k| \geq (1 - s)|I_0|$.}
    \label{fig:good_interval}
\end{figure}

The next lemma says that if we restrict ourselves to considering intervals that contain a uniformly bounded number of points of $E$, then the median porosity condition \cref{eqn:example_hypothesis} is trivial.

\begin{lemma}\label{lem:small_N}
    Fix a natural number $m \geq 1$. Then if $I_0 = (a, b)$ contains $m$ points of $E$ we have 
    \[\sq{\mc{L}}_1(I_0) \lesssim_{\gamma, m} \sq{\mc{L}}_{\frac{2}{3}}(I_0) \leq \sq{\mc{L}}_1(I_0).\]
\end{lemma}

\begin{proof}
    The second inequality follows immediately from the definitions so we only need to prove the first inequality. \par
    \textbf{Case 1:} First assume that $m = 1$. Then we can write the interval 
    \[I_0 = I_1 \cup \{x_0\} \cup I_2\]
    where $I_1$ and $I_2$ are $E$-free and $x_0 \in E$. Clearly $|I_j| \geq (1/2)|I_0|$ for $j=1$ or $j=2$. For this $I_j$, we have that $\sq{\mc{L}}_1(I_0) = |I_j|$. It follows that 
    \[\sq{\mc{L}}_1(I_0) = |I_j| \leq  \mc{\sq{L}}_{1/2}(I_0) \leq \mc{\sq{L}}_{\frac{2}{3}}(I_0)\]
    which concludes the proof of the lemma for $m = 1$. \par
    \textbf{Case 2:} Now assume that $m \geq 2$. Write $n_l = n_l(I_0)$ and $n_r = n_r(I_0)$ so that $m = n_r - n_l + 1$. Let $\delta_0 := n_r^\gamma - (n_r - 1)^\gamma$ be the length of the smallest standard interval fully contained in $I_0$. If $(n_l^\gamma - a) \leq (n_r^\gamma - n_l^\gamma) $, then we choose the $m - 1$ standard intervals $\mc{I}_{n_l}, ..., \mc{I}_{n_r - 1}$ contained in $I_0$. See \cref{fig:small_N}. We claim that 
    \begin{equation}\label{eqn:small_N_size}
        \sum_{j = n_l}^{n_r - 1}|\mc{I}_j| \geq \frac{1}{3}|I_0|.
    \end{equation}
    Indeed, we have 
    \begin{align*}
        \frac{|I_0|}{\sum_{j = n_l}^{n_r - 1}|\mc{I}_j|} & = \frac{b - a}{n_r^\gamma - n_l^\gamma} \\
         & = 1 + \frac{b - n_r^\gamma}{n_r^\gamma - n_l^\gamma} + \frac{n_l^\gamma - a}{n_r^\gamma - n_l^\gamma} \\
         & \leq 3
    \end{align*}
    which gives \cref{eqn:small_N_size}. It follows that $\sq{\mc{L}}_{\frac{2}{3}}(I_0) \geq \delta_0$. If $(n_l^\gamma - a) > (n_r^\gamma - n_l^\gamma)$, then a similar proof works but we add the interval $\sq{\mc{I}_{0}} = (a, n_l^\gamma)$ to the collection $\mc{I}_{n_l}, ..., \mc{I}_{n_{r}-1}$, obtaining \cref{eqn:small_N_size} with constant $\frac{1}{2}$.
    Finally, we observe that the lengths of the smallest and largest standard intervals contained in $I_0$ are proportional (with constants depending on $m$ and $\gamma$): If $n_l > 1$ then
    \begin{align}\label{RatioLargestToSmallestStandardIntervals}
        \frac{n_l^\gamma - (n_l - 1)^\gamma}{n_r^\gamma - (n_r - 1)^\gamma} & \leq \frac{\gamma (n_l - 1)^{\gamma - 1}}{\gamma n_r^{\gamma - 1}} \\ \nonumber
        & = \left(\frac{n_r}{n_l - 1}\right)^{1 - \gamma} \\ \nonumber
        & = \left(1 + \frac{m}{n_l - 1}\right)^{1 - \gamma} \\ \nonumber
        & \leq (1 + m)^{1 - \gamma}.
    \end{align}
    We conclude that
    \[\sq{\mc{L}}_{\frac{2}{3}}(I_0) \geq \delta_0 = n_r^\gamma - (n_r - 1)^\gamma \gtrsim_{\gamma, m} n_l^\gamma - (n_l - 1)^\gamma \geq \sq{\mc{L}}_1(I_0)\]
    as wanted. If $n_l = 1$, a similar reasoning gives the same result. \par
    Indeed, if $n_l = 1$ and $(n_l^\gamma - a) \leq ((n_{l}+1)^\gamma - n_l^\gamma)$ and $m=2$, then there is a unique standard interval contained in $I_0$, namely $\mc{I}_{1}$, and $\sq{\mc{L}}_{\frac{2}{3}}(I_0) = \sq{\mc{L}}_1(I_0) = |\mc{I}_{1}|$ and we are done. \par
    Next, if $n_l = 1$ and $(n_l^\gamma - a) \leq (n_{l+1}^\gamma - n_l^\gamma)$ and $m>2$, then the analogue of \eqref{RatioLargestToSmallestStandardIntervals} becomes
    \[
    \frac{2^\gamma - 1^\gamma}{n_r^\gamma - (n_r - 1)^\gamma} \leq
    \frac{\gamma}{\gamma (n_r)^{\gamma-1}}= (n_r)^{1-\gamma} \leq m^{1-\gamma}, 
    \]
    so we conclude as in the $n_l > 1$ case. \par
    Finally, if $n_l = 1$ and $(n_l^\gamma - a) > (n_{l+1}^\gamma - n_l^\gamma)$, we add the interval $\sq{\mc{I}_{0}} = (a, n_l^\gamma) = (a,1)$, and then the analogue of \eqref{RatioLargestToSmallestStandardIntervals} becomes
    \[
    \frac{n_l^\gamma - a}{n_r^\gamma - (n_r - 1)^\gamma} \leq 
    \frac{1}{\gamma (n_r)^{\gamma-1}} \leq \frac{m^{1-\gamma}}{\gamma},
     \]
    and again we conclude as in the $n_l > 1$ case.   
\end{proof}

\begin{figure}[!ht]
    \centering
    \resizebox{1\textwidth}{!}{%
    \begin{tikzpicture}[scale=2]
    \tikzstyle{every node}=[font=\small]
    \draw[->] (5,8.5) to[short] (10.3,8.5);
    \node at (5,8.5) [circ] {};
    \node at (5.946,8.5) [circ] {};
    \node [font=\SMALL] at (6.4,8.7) {$a$};
    \node at (6.58,8.5) [circ] {};
    \node at (7.07,8.5) [circ] {};
    \node at (7.48,8.5) [circ] {};
    \node at (7.825,8.5) [circ] {};
    \node at (8.133,8.5) [circ] {};
    \node [font=\SMALL] at (8.3,8.7) {$b$};
    \node at (8.41,8.5) [circ] {};
    \node at (8.66,8.5) [circ] {};
    \node at (8.89,8.5) [circ] {};
    \node at (9.106,8.5) [circ] {};
    \node at (9.494,8.5) [circ] {};
    \node at (9.306,8.5) [circ] {};
    \node at (9.67,8.5) [circ] {};
    \node at (9.84,8.5) [circ] {};
    \node at (10,8.5) [circ] {};

    \draw [
    thick,
    decoration={
        brace,
        mirror,
        raise=0.2cm
    },
    decorate
    ] (6.6,8.5) -- (7.05,8.5)
    node [pos=0.5,anchor=north,yshift=-0.25cm,font=\SMALL] {$I_3$};

    \draw [
    thick,
    decoration={
        brace,
        mirror,
        raise=0.2cm
    },
    decorate
    ] (7.09,8.5) -- (7.46,8.5)
    node [pos=0.5,anchor=north,yshift=-0.25cm,font=\SMALL] {$I_4$};

    \draw [
    thick,
    decoration={
        brace,
        mirror,
        raise=0.2cm
    },
    decorate
    ] (7.5,8.5) -- (7.805,8.5)
    node [pos=0.5,anchor=north,yshift=-0.25cm,font=\SMALL] {$I_5$};
    
    \draw [
    thick,
    decoration={
        brace,
        mirror,
        raise=0.2cm
    },
    decorate
    ] (7.825,8.5) -- (8.133,8.5)
    node [pos=0.5,anchor=north,yshift=-0.25cm,font=\SMALL] {$I_6$};
    \end{tikzpicture}
    }%
    \caption{The choice of intervals $\{\mc{I}_k\}$ when $I_0$ contains $m = 5$ points.}
    \label{fig:small_N}
\end{figure}
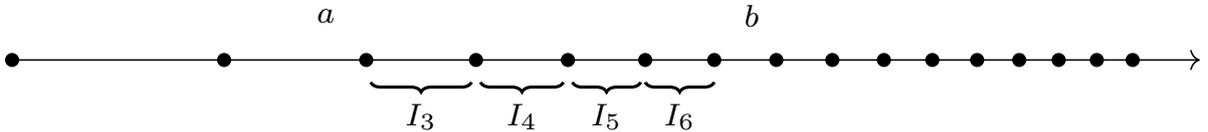

Finally, if $I_0$ contains a large number of points of $E$, we can approximate $I_0$ by good intervals and apply \cref{lem:example_good}.

\begin{lemma}\label{lem:large_N}
    Let $0 < s < 1$. Then for any $\epsilon > 0$, there is some large $N \geq 10$ such that if $I_0 = (a, b)$ contains at least $N$ points of $E$, then 
    \[\mc{\sq{L}}_{s + \epsilon}(I_0) \gtrsim_{\gamma, s}(a + (1 - s)(b - a))^{(\gamma - 1) / \gamma}\]
    and 
    \begin{equation}\label{UpperBoundForLTildeSMinusEpsilon}
    \mc{\sq{L}}_{s - \epsilon}(I_0) \lesssim_{\gamma, s} (a + (1 - s)(b - a))^{(\gamma - 1) / \gamma}.
    \end{equation}
    As a consequence, if $\epsilon < \min\{ s, 1-s\}$, then  \begin{equation}\label{eqn:LsBehavesAsWithGoodIntervalsIfNL>1AndContainManyPoints}
        \mc{\sq{L}}_{s}(I_0) \approx_{\gamma, s, \epsilon}(a + (1 - s)(b - a))^{(\gamma - 1) / \gamma}.
    \end{equation}
\end{lemma}

\begin{proof}
 
    Fix some $N$ that will be made large later and suppose that $I_0$ contains at least $N$ points of $E$. Write $n_l = n_l(I_0)$ and $n_r = n_r(I_0)$ so that $n_r - n_l + 1 \geq N$. We first claim that there is some function $v = v_{s,\gamma}(N)$ such that $v(N) \to 0$ as $N \to \infty$ and such that
    \begin{equation}\label{eqn:LengthOfMaximalFreeIntervalInI0Estimate}
    \sq{\mc{L}}_1(I_0) \leq v(N)|I_0|.
    \end{equation}
    If $n_l = 1$ then $a \leq 1$ so that 
    \[\sq{\mc{L}}_1(I_0) \leq 1 \leq \frac{|I_0|}{b - 1} \leq \frac{|I_0|}{N^\gamma - 1}.\]
    On the other hand, if $n_l > 1$,
    \begin{align*}
        \sq{\mc{L}}_1(I_0) & \leq n_l^\gamma - (n_l - 1)^\gamma \\ & \leq \gamma (n_l - 1)^{\gamma - 1} \\
        & \lesssim_\gamma \frac{(n_l - 1)^{\gamma - 1}}{n_r^\gamma - n_l^\gamma}|I_0| \\
        & \lesssim_\gamma \frac{(n_l - 1)^{\gamma - 1}}{(n_r - n_l)n_r^{\gamma - 1}}|I_0| \\
        & = \frac{1}{n_r - n_l}\left(\frac{n_r - n_l}{n_l - 1} + \frac{n_l}{n_l - 1}\right)^{1 - \gamma}|I_0| \\
        & \leq \left\{ \frac{1}{(n_r - n_l)^\gamma} + \frac{2^{1 - \gamma}}{n_r - n_l} \right\}|I_0| \\
        & \lesssim (N - 1)^{-\gamma}|I_0|.
    \end{align*}
    Therefore by setting $v(N) = C_{\gamma}\max((N^\gamma - 1)^{-1}, (N - 1)^{-\gamma})$, we have the claim. \par Now let $J_1$ be the largest $s$-good interval contained in $I_0$ and let $J_2$ be the smallest $(s + 1) / 2$-good interval containing $I_0$. See \cref{fig:large_N}. Write $J_j = (a_j, b_j)$ where $a_1 = n_l^\gamma$ and $a_2 = (n_l - 1)^\gamma$. Observe that 
    \[b_1 = \max\{b' \in [a_1, b] : a_1 + (1 - s)(b' - a_1) \in E\}.\]
    Therefore the interval $(a_1 + (1 - s)(b_1 - a_1), a_1 + (1 - s)(b - a_1))$ is $E$-free. This implies that its length is at most $\sq{\mc{L}}_1(I_0)$. Thus 
    \[(b - a_1) - (b_1 - a_1) \leq \frac{n_l^\gamma - (n_l - 1)^\gamma}{1 - s} \lesssim_s v(N)|I_0|.\]
    Therefore 
    \[|I_0| = b - a = b_1 - a_1 + (b - b_1) + (a_1 - a) \leq |J_1|(1 + C_{s}v(N)).\]
    In a similar way, we can show that 
    \[
    |J_2| \leq |I_0|(1 + C_{s, \gamma} v(N)).
    \]
    Indeed, since adjacent standard intervals have comparable length (with constants depending on $\gamma$), we can easily reduce to the case where the largest $E$-free interval in $J_2$ is contained in $I_0$, so that \cref{eqn:LengthOfMaximalFreeIntervalInI0Estimate} applies. The rest of the proof is completely analogous, with the small remark that, in the case that $n_l = 1$, just for the purposes of this estimate, we can consider $a_2 = 0 \in E$, so that $J_2$ has left endpoint in $E$, and then \cref{lem:example_good} is valid for $J_2$, with the same proof (with $n_0 = 0$). 

    Now let $\{I_j\}$ be a sequence of $E$-free sub-intervals of $J_1$ such that $|I_j| \geq \mc{\sq{L}}_s(J_1) / 2$ and
    \[\sum |I_j| \geq (1 - s)|J_1|.\]
    Then 
    \[\sum |I_j| \geq (1 - s)(1 + C_{s}v(N))^{-1}|I_0|.\]
    Therefore if we take $N$ large enough so that $(1 - s)(1 + C_{s}v(N))^{-1} \geq 1 - s - \epsilon$, we have by \cref{lem:example_good}
    \[\mc{\sq{L}}_{s+\epsilon}(I_0) \gtrsim \mc{\sq{L}}_{s }(J_1) \gtrsim_\gamma (n_l^\gamma + (1 - s)|J_1|)^{(\gamma - 1) / \gamma} 
    \gtrsim_{\gamma,s} (a + (1 - s)(b - a))^{(\gamma - 1) / \gamma}.
    \]
    The inequality \eqref{UpperBoundForLTildeSMinusEpsilon}, but replacing $(1 - s)$ with $\left( 1 - \frac{s+1}{2} \right)$ in the right-hand side, is obtained similarly. 
    Finally, note that if $0 < s', t' < 1$, then
    \[
    \left(\frac{a + (1 - s')(b - a)}{a + (1 - t')(b - a)}\right)^{(\gamma - 1) / \gamma} \approx_{\gamma, s', t'} 1, \]
    so both \eqref{UpperBoundForLTildeSMinusEpsilon} and \cref{eqn:LsBehavesAsWithGoodIntervalsIfNL>1AndContainManyPoints} follow readily, with constants independent of $a$ and $b$.

\end{proof}

\begin{figure}[!ht]
    \centering
    \resizebox{1\textwidth}{!}{%
    \begin{tikzpicture}[scale=2]
    \tikzstyle{every node}=[font=\small]
    \draw[->] (5,8.5) to[short] (10.3,8.5);
    \node at (5,8.5) [circ] {};
    \node at (5.946,8.5) [circ] {};
    \node [font=\SMALL] at (6.2,8.7) {$a$};
    \node at (6.58,8.5) [circ] {};
    \node at (7.07,8.5) [circ] {};
    \node at (7.48,8.5) [circ] {};
    \node at (7.825,8.5) [circ] {};
    \node at (8.133,8.5) [circ] {};
    \node [font=\SMALL] at (7.825,8.7) {$c_2$};
    \node at (8.41,8.5) [circ] {};
    \node [font=\SMALL] at (8.41,8.7) {$c_1$};
    \node at (8.66,8.5) [circ] {};
    \node at (8.89,8.5) [circ] {};
    \node at (9.106,8.5) [circ] {};
    \node [font=\SMALL] at (9.2,8.7) {$b$};
    \node at (9.306,8.5) [circ] {};
    \node at (9.494,8.5) [circ] {};
    \node at (9.67,8.5) [circ] {};
    \node at (9.84,8.5) [circ] {};
    \node at (10,8.5) [circ] {};

    \draw [
    thick,
    decoration={
        brace,
        mirror,
        raise=0.2cm
    },
    decorate
    ] (6.58,8.5) -- (8.8,8.5)
    node [pos=0.5,anchor=north,yshift=-0.25cm,font=\SMALL] {$J_1$};

    \draw [
    thick,
    decoration={
        brace,
        mirror,
        raise=0.2cm
    },
    decorate
    ] (5.946,8.2) -- (9.494,8.2)
    node [pos=0.5,anchor=north,yshift=-0.25cm,font=\SMALL] {$J_2$};

    \end{tikzpicture}
    }%
    \caption{The intervals $J_i = (a_i, b_i)$ with the points $a_i, c_i \in E$ where $c_1 = a_1 + (1 - s)|J_1|$ and $c_2 = a_2 + (1 - (\frac{s+1}{2}))|J_2|$.}
    \label{fig:large_N}
\end{figure}
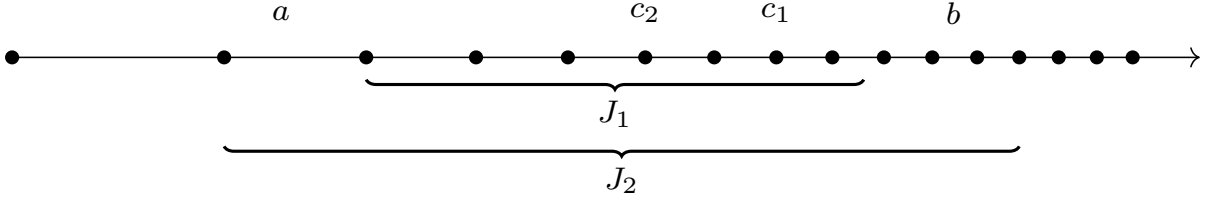

\begin{proof}[Proof of \cref{thm:example}]
    We start by showing that $E$ is not weakly porous. Suppose to the contrary that $E$ is weakly porous with constants $0 < s, \delta < 1$. Let $I_m = [0, m^\gamma)$. Then for all $m$, $\mc{\sq{L}}_1(I_m) = 1$ since the maximal $E$-free subinterval of $I_m$ is just $[0, 1)$. However, if $I \in \mc{S}_\delta^1(I_m)$ then we must have 
    \[
    |I| \geq \delta \, \mc{L}_1(I_m) \approx \delta.
    \]
    But $E$-free subintervals of length at least $C\delta$ can only fill a fixed  length of $[0, \infty)$ since by \cref{eqn:standard_length}, $|\mc{I}_k| \lesssim (k + 1)^{\gamma - 1} \leq \delta$ for large enough $k$. This contradicts the weakly porous assumption.
    \par
    We now prove that $E$ is median porous. It suffices to show that for some $0 < s < t < 1$, we have \cref{eqn:example_hypothesis}
    for all intervals $I$ that intersect $E$. Let $(s, t) = (3/4, 7/8)$ and choose $\epsilon = 1/16$ and $N$ large enough as in \cref{lem:large_N}. 
    \par 
    
    Let $I_0 = (a, b) \subset [0, \infty)$ be an arbitrary interval. First suppose that 
    $I_0$ contains $m \leq N$ points. Then by \cref{lem:small_N}, we have
    \[\sq{\mc{L}}_t(I_0) \leq \sq{\mc{L}}_{1}(I_0) \lesssim_{\gamma, N} \sq{\mc{L}}_s(I_0).\]
    
    Now suppose that $I_0$ contains at least $N$ points of $E$. Then by \cref{lem:large_N}, we have 
    \begin{equation}\label{eqn:LsAndLtComparableWhenGetEstimateSimilarToSGoodIntervals}
        \frac{\sq{\mc{L}}_t(I_0)}{\sq{\mc{L}}_s(I_0)} \approx_\gamma  
        \left(\frac{a + \frac{1}{8}(b - a)}{a + \frac{1}{4}(b - a)}\right)^{(\gamma - 1) / \gamma} 
        \approx_\gamma 1,
    \end{equation}
    independently of $a$ and $b$.

    Therefore \cref{eqn:example_hypothesis} holds for all intervals $I_0 \subset [0, \infty)$ and we conclude that $E$ is median porous.
\end{proof}

\section{Hardy-Sobolev Inequalities}\label{sec:hardy_sobolev}

We conclude this article with a discussion on Hardy-Sobolev inequalities. These are Sobolev embeddings on weighted spaces where the weights are $\dist(\cdot, E)$ raised to some power. Let $E \subset \R^n$ and consider the following Hardy-Sobolev inequality:
\begin{equation}\label{eqn:hardy-sobolev}
\left(\int_{\R^n}|f(x)|^{q}d^{\beta_2}(x) \, dx\right)^{1/q} \leq C\left(\int_{\R^n}|\nabla f(x)|^{p}d^{\beta_1}(x) \, dx\right)^{1/p}
\end{equation}
for all $f \in C_c^\infty(\R^n)$ where $d(x) = \dist(x, E)$, $1 \leq p \leq q \leq np / (n - p)$, $\beta_1 \in \R$ and 
\[\beta_2 = \frac{q}{p}(n - p + \beta_1) - n.\]
Under some dimensional assumptions on $E$,  \cref{eqn:hardy-sobolev} can hold. There has been considerable interest in studying these inequalities, in particular the questions of which domains support Hardy-Sobolev inequalities, and properties such as self-improvement of those inequalities (see e.g. \cite{KoskelaZhongHardyInequalityAndBoundarySize}, \cite{KoskelaLehrback_WeightedPointwiseHardy}, \cite{Lehrbäck2013dimension}, \cite{LehrbackWeightedHardyInequalitiesSizeOfBoundary}, \cite{LehrbackVahakangasInbetweenSobolevandHardy}, \cite{HurriSyrjanenVahakangas_FractionalSobolevPoincareAndFractionalHardyUnboundedJohnDomains}, \cite{ErikssonBiqueLehrbackVahakangas_SelfImprovementOfWeightedPointwiseInequalitiesOnOpenSets}, \cite{ErikssonBiqueKlineSelfImprovementOfFractionalHardyInequalities}, and the references therein), through various methods such as isoperimetric inequalities, Sobolev and Riesz capacity estimates, hyperbolic fillings, etc. 

Most recently, the validity of Hardy-Sobolev inequalities for domains $\R^n \setminus E$ has been pursued via Muckenhoupt $A_p$ properties of powers of the distance function (see e.g.~\cite{dyda2017muckenhoupt}, again \cite{LehrbackVahakangasInbetweenSobolevandHardy}, and \cite{KinnunenLehrbackVahakangasBook}). However, to date this approach has been limited in its development due to a lack of characterization of which sets $E$, and for which powers of the distance function $\dist(\cdot, E)$, such powers $w = \dist(\cdot, E)^{-\alpha}$ are Muckenhoupt $A_p$ weights. So the previous studies had to use side conditions such as porosity of the set $E$. E.g.~in \cite{LehrbackVahakangasInbetweenSobolevandHardy}, building on \cite{HoriuchiImbeddingTheoremsForWeightedSobolevSpaces2}, the isoperimetric inequality was used in combination with porosity and $A_1$ properties of powers of the distance function; in \cite{dyda2017muckenhoupt} the $A_p$ method was combined with Riesz potentials and porosity. The recent work \cite{ErikssonBiqueKlineSelfImprovementOfFractionalHardyInequalities} takes another approach, namely hyperbolic fillings, to prove self-improvement properties of fractional Hardy inequalities, which is a property clearly reminiscent of Muckenhoupt $A_p$ weights. 

In our case, while we do not completely remove the side conditions, we substantially weaken them in the critical case, from porosity to median porosity. Since \cref{thm:BMO_char_dist} shows that median porosity is equivalent to the existence of some power of the distance function being a Muckenhoupt $A_p$ weight, the median porosity side condition seems to be the limit of the $A_p$ technique.

More concretely, in \cite{dyda2017muckenhoupt}, the following theorem was proved:
\begin{theorem}\label{thm:hardy_sobolev_old}
    Let $1 < p \leq q \leq np / (n - p) < \infty$, (so, in particular, $p<n$), $\beta_1 \in \R$, and $\beta_2 = q/p(n - p + \beta_1) - n$. If 
   \[\underline{\textup{co\,dim}}_A(E) > \max \left(-\beta_2, \frac{\beta_1}{p - 1}\right)\]
   then \cref{eqn:hardy-sobolev} holds for all $f \in C_c^\infty(\R^n)$.
\end{theorem} 
As an application of \cref{thm:BMO_char_dist}, we can improve the sufficient conditions for \cref{eqn:hardy-sobolev} in the critical case. We have the following theorem.

\begin{theorem}\label{cor:hardy_sobolev}
Let $\beta_1$, $\beta_2$, $p$, and $q$ satisfy the conditions in \cref{thm:hardy_sobolev_old} and let $n \geq 2$. Also, assume that 
\[q = \frac{np}{n - p}.\]
If 
    \begin{equation}\label{DimensionalHypothesesForSufficientConditonsInCriticalCaseInHardySobolev}
    \textup{Mu}_\infty(E) > \max \left\{ -\beta_2, \frac{\beta_1}{p - 1} \right\} \geq \min \left\{ -\beta_2, \frac{\beta_1}{p - 1} \right\} > - \sup_{s > 1}(s - 1)\textup{Mu}_{s'}(E) 
    \end{equation}
then there is a constant $C$ such that 
\[\left(\int_{\R^n}|f(x)|^{q}d^{\beta_2}(x) \, dx\right)^{1/q} \leq C\left(\int_{\R^n}|\nabla f(x)|^{p}d^{\beta_1}(x) \, dx\right)^{1/p}\]
for all $f \in C_c^\infty(\R^n)$.
\end{theorem}

\begin{remark}\label{WeActuallyImproveSufficientConditionsForHardySobolev}
Theorem \ref{cor:hardy_sobolev} is indeed an improvement of \cref{thm:hardy_sobolev_old}. To be sure, when $q = \frac{np}{n - p}$, then $\beta_2 = \frac{n\beta_1}{n - p}$, which implies that $\beta_1$ and $\beta_2$ have the same sign. Thus, the hypotheses of \cref{thm:hardy_sobolev_old} imply in particular that $\underline{{\textup{co\,dim}}}_A(E) > 0$ and hence that $E$ is porous by \cref{eqn:porosity_equiv_codim}. 

For a porous set $E$, we have that $\textup{Mu}_1(E) = \underline{{\textup{co\,dim}}}_A(E)$, as observed in \cite[p.17, lines 1-2]{anderson2022weakly}, and also that $E$ is weakly porous by \cref{thm:muck_weakly_porous}. Then, by Corollary \ref{PropertiesOfMu_pFunction}, $\text{Mu}_1(E) = \text{Mu}_p(E) = \text{Mu}_\infty(E) >0$ for all $1 < p < \infty$. Consequently, $\sup_{s > 1}(s - 1)\textup{Mu}_{s'}(E) = +\infty$, and $\text{Mu}_\infty(E) = \underline{{\textup{co\,dim}}}_A(E)$, so the hypotheses in 
\eqref{DimensionalHypothesesForSufficientConditonsInCriticalCaseInHardySobolev} are satisfied.

But of course \eqref{DimensionalHypothesesForSufficientConditonsInCriticalCaseInHardySobolev} allows for sets that are not porous, and not even weakly porous.
\end{remark}

\begin{proof}
    To prove the assertion, we use the classical result of Muckenhoupt and Wheeden \cite{1974Muckenhoupt} which says that if a weight $w$ satisfies 
    \begin{equation}\label{eqn:ApTypeConditionInMuckenhouptWheeden}
         \left(\dashint_{Q} w \, dx\right)\left(\dashint_{Q} w^{-p' / q} \, dx\right)^{q/p'} \leq C
    \end{equation}
    where $q = np / (n - p)$ (i.e.~$w \in A_{1 + (q/p')}$) and $1 < p < n$, then there is a constant $C$ such that 
    \begin{equation}\label{eqn:WeightedHardySobolevInequalityCriticalCase}
    \left(\int_{\R^n}|f(x)|^{q}w(x) \, dx\right)^{1/q} \leq C\left(\int_{\R^n}|\nabla f(x)|^{p}w(x)^{p/q} \, dx\right)^{1/p}
    \end{equation}
    for all $f \in C_c^\infty(\R^n)$. \par
    Set $w = d^{\beta_2}$. Note that when $q = np / (n - p)$,  
    \[\beta_2 = \frac{n\beta_1}{n - p}.\]
    We use the fact that $w \in A_{r}$ if and only if $w \in A_\infty$ and $w^{-1/(r - 1)} \in A_\infty$ (see e.g. \cite[Chapter IV, Thm. 2.17]{garcia1985weighted}). In our case, $r = 1 + (q / p')$ so that 
    \[-\frac{1}{r - 1} = - \frac{p'}{q}.\]
    By \cref{thm:Ap_range}, $w \in A_\infty$ if and only if 
    \[
    \text{Mu}_\infty(E) > - \beta_2 > -\sup_{s > 1}(s - 1)\textup{Mu}_{s'}(E) ,
    \] 
    and $w^{-p' / q} \in A_\infty$ if and only if 
    \[
    \text{Mu}_\infty(E) > \frac{\beta_2 p'}{q} = \frac{\beta_1}{p - 1} > -\sup_{s > 1}(s - 1)\textup{Mu}_{s'}(E).
    \]  
    Both of these conditions holding simultaneously are equivalent to the hypothesis of the theorem. Therefore $w \in A_{1 + (q/p')}$ which proves the assertion.
\end{proof}

The main idea in the proof of Theorem \ref{cor:hardy_sobolev} is to guarantee that the appropriate  $s$-Riesz potential (namely $s=1$ in our case)
\begin{equation}\label{RieszPotential}
    \mathcal{I}_s(f)(x) = c_n \int_{\R^n} \frac{f(y)}{|x-y|^{n-s}} dy
\end{equation}
satisfies
\begin{equation}\label{RieszPotentialBound}
   \| \mathcal{I}_s(f)w^{\frac{1}{q}} \|_{L^q} \lesssim \| f w^{\frac{1}{q}}  \|_{L^p} ,
\end{equation}
for $0 < s < n$, $1 < p < \frac{n}{s}$ and $\frac{1}{q} = \frac{1}{p} - \frac{s}{n}$. Indeed, \cite{1974Muckenhoupt} proves that \eqref{RieszPotentialBound} is precisely equivalent to \cref{eqn:ApTypeConditionInMuckenhouptWheeden}. From \eqref{RieszPotentialBound}, \cref{eqn:WeightedHardySobolevInequalityCriticalCase} follows easily. 

This strategy is precisely the one followed in \cite{dyda2017muckenhoupt}, both for first order Hardy-Sobolev inequalities, and fractional Hardy-Sobolev inequalities. If we get closer to the implementation of the strategy in \cite{dyda2017muckenhoupt} than our proof of Theorem \ref{cor:hardy_sobolev}, we get Theorem \ref{FractionalAndFirstOrderHardySobolevInequalitiesViaRieszPotentials} below. Part (a) in Theorem \ref{FractionalAndFirstOrderHardySobolevInequalitiesViaRieszPotentials} is completely equivalent to Theorem \ref{cor:hardy_sobolev} (see \cref{rem:ComparisonOfHypothesesOfSufficientConditionsForHardySobolevInCriticalCase}), but it also yields with no extra effort sufficient conditions for fractional Hardy-Sobolev inequalities (part (b)), that improve the results in \cite{dyda2017muckenhoupt} for the critical exponents. 

As mentioned in the introduction, we focus on the case $\beta_1 < 0$ because it is the only case where conditions strictly weaker than porosity seem to be able to contribute to the characterization of domains supporting a Hardy-Sobolev inequality. 

Indeed, on the one hand, the first order ($s=1$ in Theorem \ref{FractionalAndFirstOrderHardySobolevInequalitiesViaRieszPotentials}) weighted Hardy-Sobolev inequalities for $\beta_1 = 0$ (in which case $-p \leq \beta_2 \leq 0$, and thus $\dist(x, E)^{\beta_1}$ does not blow up locally and $\dist(x, E)^{\beta_2}$ blows up locally in a controlled manner in \eqref{eqn:hardy-sobolev}), the sets $E \subset \R^n$ for which \cref{eqn:HardySobolevIneqIntroduction} is true have already been completely characterized in \cite[Theorem 1.1]{LehrbackVahakangasInbetweenSobolevandHardy} precisely by the condition that $\textup{\underline{co\,dim}}_{A}(E) > n - \frac{q}{p}(n - p) = - \beta_2 \geq 0$, which in turn implies porosity of the set $E$ by \cref{eqn:porosity_equiv_codim}.

On the other hand, for $\beta_1 > 0$ (in which case $\frac{q}{p} \beta_1 - p \leq \beta_2 \leq  \frac{q}{p} \beta_1$, and thus $\dist(x, E)^{\beta_1}$ does not blow up locally, whereas $\dist(x, E)^{\beta_2}$ either does not blow up locally, if $\beta_2 \geq 0$, or does blow up locally but in a controlled manner in \eqref{eqn:hardy-sobolev}), the necessary condition for the Hardy-Sobolev inequality given in \cite[Theorem 6.1]{LehrbackVahakangasInbetweenSobolevandHardy} is $\textup{\underline{co\,dim}}_{A}(E) > n - \frac{q}{p}(n - p + \beta_1) = - \beta_2$, which implies porosity of the set $E$ when $- \beta_2 > 0$ by \cref{eqn:porosity_equiv_codim} and, moreover, it almost matches the sufficient conditions for the Hardy-Sobolev inequality given in \cite[Theorem 6.1]{dyda2017muckenhoupt}, namely that $\textup{\underline{co\,dim}}_{A}(E) > \max \left\{ n - \frac{q}{p}(n - p + \beta_1), \frac{\beta_1}{p-1} \right\}$ (see also \cite[Remark 6.2]{dyda2017muckenhoupt}). So it seems very likely that for the case $\beta_1 > 0$, the characterization of weighted Hardy-Sobolev inequalities will also be in terms of porosity.

However, the case $\beta_1 < 0$ is much more technical. Indeed, in this case, $\beta_2 \leq \frac{q}{p} \beta_1 $ (with equality precisely in the critical case $q = p^* = \frac{np}{n-p}$), so $\beta_2<\beta_1<0$. Thus, both $\dist(x, E)^{\beta_1}$ and $\dist(x, E)^{\beta_2}$ are singular near $E$ in \eqref{eqn:hardy-sobolev}, with the latter having the stronger singularity. 

Now note that both all the necessary conditions and all the sufficient conditions currently known in this $\beta_1 < 0$ case assume porosity of the set $E$ (see e.g. \cite{dyda2017muckenhoupt} and \cite{LehrbackVahakangasInbetweenSobolevandHardy} and the references therein). Our Theorems in this section are, to our knowledge, the first instance in the literature characterizing weighted Hardy-Sobolev inequalities in the critical case $q = p^* = np / (n - p)$. 

So, getting back to implementing the Riesz potential strategy, we first get the corresponding bound for Riesz potentials, which improves on \cite[Theorem 4.1]{dyda2017muckenhoupt} in the critical case $q = \frac{np}{n - sp}$, since instead of porosity, we only assume median porosity. 

\begin{theorem}\label{WeightedInequalityForRieszPotentialsForMedianPorousSets}
Let $s>0$, and let $\emptyset \neq E \subset \R^n$ be a closed set. Let $1 < p \leq q = \frac{np}{n - sp} < \infty$ and $\beta_1<0$ be such that
\begin{equation}\label{MuckenhouptExponentHypothesesForRieszPotentialTheorem1}
    \textup{Mu}_\infty(E) > n - \frac{q}{p}(n - sp + \beta_1) =: - \beta_{2,s}
\end{equation}
and
\begin{equation}\label{MuckenhouptExponentHypothesesForRieszPotentialTheorem2}
    \textup{Mu}_p(E) > - \beta_1.
\end{equation}
Then there is a constant $C>0$ such that the inequality
\begin{equation}\label{RieszPotentialWeightedBound}
    \left(\int_{\R^n}\mathcal{I}_s(f)(x)^{q}d^{\beta_{2,s}}(x) \, dx\right)^{1/q} \leq C\left(\int_{\R^n}f(x)^{p}d^{\beta_1}(x) \, dx\right)^{1/p}
\end{equation}
holds for all measurable functions $f \geq 0$.
\end{theorem}

\begin{proof}
    As in the case of \cite[Theorem 4.1]{dyda2017muckenhoupt}, the proof consists basically of checking, for the particular case we need, the hypotheses of \cite[Theorem 4.2]{dyda2017muckenhoupt} which in turn is heavily based on \cite{PerezWheedenPotentialOperatorsMaximalFunctionsAndGeneralizationsOfAInfty}. 
    
    Indeed, consider the weights $w(x) = d^{\beta_{2,s}}(x)$, $v(x) = d^{\beta_1}(x)$, and $h(x) = d^{\frac{-\beta_1}{p-1}}(x)$. We first check that all of them are $A_\infty$ weights. Note that $q \leq \frac{np}{n - sp}$ implies that $- \beta_{2,s} = n - \frac{q}{p}(n - sp + \beta_1) \geq - \frac{q}{p} \beta_1 > - \beta_1 > 0$. Consequently, by \eqref{MuckenhouptExponentHypothesesForRieszPotentialTheorem1} and \cref{thm:Ap_range_2}, both $w$ and $v$ are in $A_\infty$. Regarding $h(x)$, note that by \eqref{MuckenhouptExponentHypothesesForRieszPotentialTheorem2}, $  \textup{Mu}_p(E) > - \beta_1 > 0$, which by \cref{thm:Ap_range_2} is equivalent to $d^{\beta_1} \in A_p$, which in turn, by the standard duality ($u \in A_p \Longleftrightarrow u^{\frac{-1}{p-1}} \in A_{p'}$), is equivalent to $h(x) = d^{\frac{-\beta_1}{p-1}} \in A_{p'} \subset A_\infty$.

    Next we check that there exists a constant $K>0$ such that the inequality \begin{equation}\label{ConditionForRieszPotentialToMapLpToLqInDydaEtAl}
        w(B)^{\frac{1}{q}} h(B)^{\frac{p-1}{p}} \leq K \ \textup{rad}(B)^{-s} \ |B| 
    \end{equation}
    holds for all balls $B \subset \R^n$. Since all weights involved are doubling (being $A_\infty$ weights), it suffices to check the analogue of \eqref{ConditionForRieszPotentialToMapLpToLqInDydaEtAl} for cubes.
    
    On the one hand, by \eqref{MuckenhouptExponentHypothesesForRieszPotentialTheorem1} and as in \eqref{EstimatingApConditionFactor1}, we get
    \begin{equation}\label{WQEstimate}
     w(Q)^{\frac{p}{q}} = \left( \int_Q d^{\beta_{2,s}}(y) \ dy   \right)^{\frac{p}{q}} \lesssim \left( |Q| \ \mc{L}_{s'}(Q)^{\frac{q}{p}(n - sp + \beta_1) - n} \right)^{\frac{p}{q}},
    \end{equation}
    and on the other hand, by \eqref{MuckenhouptExponentHypothesesForRieszPotentialTheorem2} and as in \eqref{EstimatingApConditionFactor2}, we get
    \begin{equation}\label{HQEstimate}
     h(Q)^{p-1} = \left( \int_Q d^{\frac{- \beta_1}{p-1}}(y) \ dy   \right)^{p-1} \lesssim \left( |Q| \ \mc{L}_{s'}(Q)^{\frac{- \beta_1}{p-1}} \right)^{p-1} .
    \end{equation}
    Thus,
    \begin{equation}\label{WQTimesHQEstimate}
    w(Q)^{\frac{p}{q}} h(Q)^{p-1} \lesssim |Q|^{\frac{p}{q} + p-1 } \mc{L}_{s'}(Q)^{- \beta_1} \mc{L}_{s'}(Q)^{ n - sp + \beta_1 - \frac{p}{q}n} = |Q|^p \ l(Q)^{-sp} \ \left[ \frac{\mc{L}_{s'}(Q)}{l(Q)} \right]^{n \left( 1 - \frac{p}{q} \right) - sp}.
    \end{equation}

    Now $\left[ \frac{\mc{L}_{s'}(Q)}{l(Q)} \right] \leq 1$, so we get \eqref{ConditionForRieszPotentialToMapLpToLqInDydaEtAl} precisely when $n \left( 1 - \frac{p}{q} \right) - sp \geq 0$, which happens precisely when $q \geq \frac{np}{n - sp}$.
\end{proof}

\begin{remark}\label{rem:WeGetComparabilityInTestingConditionForRieszPotential} 
Note that in \eqref{WQEstimate} and in \eqref{HQEstimate}, we get comparability of all terms (i.e.~``$\approx$" instead of just ``$\lesssim$") by \cref{prop:median_dist}, and thus also in \eqref{WQTimesHQEstimate}. This has two consequences.

The first consequence is that, if $E$ is porous, our proof reduces essentially to that of \cite[Theorem 4.1]{dyda2017muckenhoupt}. Indeed, by \cref{rem:AnyValueOfSWorksForDefinitionOfMedianPorous}, $\mc{L}_{s'}(Q) \approx \mc{L}_{1}(Q) \approx l(Q)$, in which case 
\eqref{ConditionForRieszPotentialToMapLpToLqInDydaEtAl} holds irrespective of the sign of $n \left( 1 - \frac{p}{q} \right) - sp$.

It is more convenient to state the second consequence further below (see \cref{rem:WeReachedTheLimitOfTheRieszPotentialMethod}). 
\end{remark}

Once we have the $s$-Riesz potential bound \eqref{RieszPotentialWeightedBound}, the next step is pointwise bounds of a function $f$ by the $s$-fractional integral of the appropriate version of the $s$-derivative of $f$. We recall these known bounds for the convenience of the reader, adapted to our setting. 

\begin{lemma}\label{RecoverFFromItsSDerivative}\cite[Lemma 5.4]{dyda2017muckenhoupt} Let $0 < s < 1$. Then if $1 \leq t < \infty$, there exists $C>0$ such that for all compactly supported Lipschitz functions $f$ and all $x \in \R^n$, we have that
\[
|f(x)| \leq C \mathcal{I}_s(g)(x),
\]
where for all $y \in \R^n$,
\[
g(y) = \left( \int_{\R^n} \frac{|f(y) - f(z)|^t}{|y - z|^{n + st}} \ dz \right)^{\frac{1}{t}}.
\]
\end{lemma}

The corresponding result for $s=1$, which we include next, is very well-known. The reader will find a proof e.g.~in \cite[Section 6]{dyda2017muckenhoupt} in a much more general context than ours.

\begin{lemma}\label{RecoverFFromItsFirstDerivative}
    There exists $C>0$ such that for any compactly supported Lipschitz function $f$ defined on $\R^n$ and all $x \in \R^n$,
    \[
    |f(x)| \leq C \mathcal{I}_1(|\nabla f|)(x).
    \]    
\end{lemma}

Combining Theorem \ref{RieszPotentialWeightedBound}, Lemma \ref{RecoverFFromItsSDerivative}, and Lemma \ref{RecoverFFromItsFirstDerivative}, we readily get the following Theorem. We will later see (\cref{rem:ComparisonOfHypothesesOfSufficientConditionsForHardySobolevInCriticalCase}) that Theorem \ref{FractionalAndFirstOrderHardySobolevInequalitiesViaRieszPotentials}(a) is \emph{exactly} the same as Theorem \ref{cor:hardy_sobolev}. However, there is some use to restating it under hypotheses that are only apparently different: one version might be easier to check than the other, and the proofs are slightly different.

\begin{theorem}\label{FractionalAndFirstOrderHardySobolevInequalitiesViaRieszPotentials}

Let $\emptyset \neq E \subset \R^n$ be a closed set. 

\begin{itemize}

    \item [(a)] Let $1 < p \leq q = \frac{np}{n - p} < \infty$ and $\beta_1<0$ be such that
\begin{equation}\label{MuckenhouptExponentHypothesesForRieszPotentialTheorem1v2}
    \textup{Mu}_\infty(E) > n - \frac{q}{p}(n - p + \beta_1) =: - \beta_{2}
\end{equation}
and
\begin{equation}\label{MuckenhouptExponentHypothesesForRieszPotentialTheorem2v2}
    \textup{Mu}_p(E) > - \beta_1.
\end{equation}
Then there is a constant $C>0$ such that the first order Hardy-Sobolev inequality
\[\left(\int_{\R^n}|f(x)|^{q}d^{\beta_2}(x) \, dx\right)^{1/q} \leq C\left(\int_{\R^n}|\nabla f(x)|^{p}d^{\beta_1}(x) \, dx\right)^{1/p}\]
holds for all $f \in C_c^\infty(\R^n)$.

    \item [(b)] Let $0 < s < 1$, and let $1 < p \leq q = \frac{np}{n - sp} < \infty$ and $\beta_1<0$ be such that
\begin{equation}\label{MuckenhouptExponentHypothesesForRieszPotentialTheorem1v3}
    \textup{Mu}_\infty(E) > n - \frac{q}{p}(n - sp + \beta_1) =: - \beta_{2,s}
\end{equation}
and
\begin{equation}\label{MuckenhouptExponentHypothesesForRieszPotentialTheorem2v3}
    \textup{Mu}_p(E) > - \beta_1.
\end{equation}
Then if $1 \leq t < \infty$, there is a constant $C>0$ such that the fractional Hardy-Sobolev inequality
    \[\left(\int_{\R^n}|f(x)|^{q}d^{\beta_{2,s}}(x) \, dx\right)^{1/q} \leq C\left(\int_{\R^n}
    \left( \int_{\R^n} \frac{|f(y) - f(z)|^t}{|y - z|^{n + st}} \ dz \right)^{\frac{p}{t}}
    d^{\beta_1}(y) \, dy\right)^{1/p}\]
holds for all $f \in C_c^\infty(\R^n)$.
    
\end{itemize}

\end{theorem}

\begin{remark}\label{rem:ComparisonOfHypothesesOfSufficientConditionsForHardySobolevInCriticalCase}
In the critical case $q = \frac{np}{n - sp}$ (both for $0<s<1$ and for $s=1$), when $\beta_1 < 0$, which is the situation Theorem \ref{FractionalAndFirstOrderHardySobolevInequalitiesViaRieszPotentials} deals with, it improves on the corresponding Theorems 5.3 and 6.1 in \cite{dyda2017muckenhoupt}. Indeed, in that case, since $n - \frac{q}{p}(n -sp + \beta_1) = - \frac{q}{p}\beta_1 >0$, the dimensional hypothesis in Theorems 5.3 and 6.1 in \cite{dyda2017muckenhoupt} reduce to $\textup{\underline{co\,dim}}_{A}(E) > - \frac{q}{p}\beta_1 >0$, so in particular $E$ is porous, and as in Remark \ref{WeActuallyImproveSufficientConditionsForHardySobolev}, $ \textup{\underline{co\,dim}}_{A}(E) = \textup{Mu}_\infty(E) = \textup{Mu}_p(E) > -\beta_{2,s} = - \frac{q}{p}\beta_1$ (note $\beta_{2,s} = \beta_2$), so the dimensional hypotheses of Theorem \ref{FractionalAndFirstOrderHardySobolevInequalitiesViaRieszPotentials} are satisfied.

The comparison between the hypotheses of Theorem \ref{FractionalAndFirstOrderHardySobolevInequalitiesViaRieszPotentials}(a) and Theorem \ref{cor:hardy_sobolev} (in the critical case $q = \frac{np}{n - p}$) is more subtle, but reveals that they are equivalent hypotheses.

Indeed, first assume the hypotheses of Theorem \ref{FractionalAndFirstOrderHardySobolevInequalitiesViaRieszPotentials}(a). Then $\textup{Mu}_p(E) > - \beta_1 > 0$, which by \cref{thm:Ap_range_2} is equivalent to $d^{\beta_1} \in A_p$, which in turn, by the standard duality ($u \in A_p \Longleftrightarrow u^{\frac{-1}{p-1}} \in A_{p'}$), is equivalent to $d^{\frac{-\beta_1}{p-1}} \in A_{p'}$, which again by \cref{thm:Ap_range_2} implies $\frac{\beta_1}{p-1} > -\sup_{s > 1}(s - 1)\textup{Mu}_{s'}(E)$. So the hypotheses of Theorem \ref{cor:hardy_sobolev} are satisfied.

Now assume the hypotheses of Theorem \ref{cor:hardy_sobolev} (for $q = \frac{np}{n - p}$). From \eqref{DimensionalHypothesesForSufficientConditonsInCriticalCaseInHardySobolev}, and \cref{thm:Ap_range_2} we get that $d^{\frac{- \beta_1}{p-1}} \in A_\infty$. From the fact that $- \beta_2 = n - \frac{q}{p}(n -p + \beta_1) = - \frac{q}{p}\beta_1 >0$, and \cref{thm:Ap_range_2} we get that $d^{\frac{q}{p}\beta_1} \in A_\infty$. In turn, we claim that this implies that $d^{\beta_1} \in A_\infty$. Indeed, note that by H\"older's inequality, and taking $f = \log d$,
\[
\left(\dashint_{Q}d^{\beta_1}\, dx\right) \leq
\left(\dashint_{Q}d^{(q/p)\beta_1}\, dx\right)^{p/q}  
\lesssim \exp(\ang{f}_Q)^{\beta_1},
\]
where the last inequality follows from the fact (see e.g. \cite[p.218, 6.3]{SteinFatBook}) that $w \in A_\infty$ if and only if
    \begin{equation}\label{ExpLogCharacterizationOfAInfty3}
        \dashint_{Q} w \, dx \lesssim \exp\left(\dashint_Q \log w \, dx\right).
    \end{equation}  
But then we get that $d^{\beta_1} \in A_\infty$ (again by \eqref{ExpLogCharacterizationOfAInfty3}). 

But now, both $d^{\frac{- \beta_1}{p-1}}$ and $d^{\beta_1}$ belong to $A_\infty$, which is equivalent to $d^{\beta_1} \in A_p$ (\cite[Chapter IV, Thm. 2.17]{garcia1985weighted}), which again by \cref{thm:Ap_range_2} implies that $\textup{Mu}_p(E) > - \beta_1$, i.e.~\eqref{MuckenhouptExponentHypothesesForRieszPotentialTheorem2v2}. So the hypotheses of Theorem \ref{FractionalAndFirstOrderHardySobolevInequalitiesViaRieszPotentials}(a) are satisfied. 

To sum everything up, the hypotheses of Theorem \ref{cor:hardy_sobolev} (for $q = \frac{np}{n - p}$) are completely equivalent to the hypotheses of Theorem \ref{FractionalAndFirstOrderHardySobolevInequalitiesViaRieszPotentials}(a).
\end{remark}

\begin{remark}\label{rem:WeReachedTheLimitOfTheRieszPotentialMethod}
The proofs of Theorem \ref{cor:hardy_sobolev} and Theorem \ref{FractionalAndFirstOrderHardySobolevInequalitiesViaRieszPotentials} show that we have essentially reached the limit of the Riesz potential method (as it appears e.g.~in \cite{PerezWheedenPotentialOperatorsMaximalFunctionsAndGeneralizationsOfAInfty}, \cite{1974Muckenhoupt}, \cite{dyda2017muckenhoupt}). Indeed, this method (say for $s = 1$ for convenience of notation) basically proves and glues together each one of the inequalities
\begin{equation}\label{eqn:FunctionDominatedByRieszPotentialOfGradientV2}
    |f(x)| \leq C \mathcal{I}_1(|\nabla f|)(x) ,
\end{equation}
and 
\begin{equation}\label{eqn:ChainOfInequalitiesForRieszPotentialMethod}
\left(\int_{\R^n}\mathcal{I}_s(f)(x)^{q}d^{\beta_{2,s}}(x) \, dx\right)^{1/q} 
\leq C\left(\int_{\R^n}f(x)^{p}d^{\beta_1}(x) \, dx\right)^{1/p},
\end{equation}
Inequality \eqref{eqn:FunctionDominatedByRieszPotentialOfGradientV2} is ``always" true and well-known. 
The issue is the inequality
\cref{eqn:ChainOfInequalitiesForRieszPotentialMethod}. Indeed, we proved such inequality, following \cite{PerezWheedenPotentialOperatorsMaximalFunctionsAndGeneralizationsOfAInfty} and \cite{dyda2017muckenhoupt}, by proving \eqref{ConditionForRieszPotentialToMapLpToLqInDydaEtAl}. But such proof is a $T1$-type theorem, meaning that condition \eqref{ConditionForRieszPotentialToMapLpToLqInDydaEtAl} is actually a \emph{necessary condition} for the 
inequality
\cref{eqn:ChainOfInequalitiesForRieszPotentialMethod} to hold, as is easy to see by testing the second inequality in \cref{eqn:ChainOfInequalitiesForRieszPotentialMethod} on the characteristic function of a ball (as noted after the statement of Theorem 2.4 in \cite{PerezWheedenPotentialOperatorsMaximalFunctionsAndGeneralizationsOfAInfty}).

Now we get to the second consequence of  \cref{rem:WeGetComparabilityInTestingConditionForRieszPotential}. Indeed, for a set $E$ that is not porous but is median porous (or even weakly porous), the quantity $\left[ \frac{\mc{L}_{s'}(Q)}{l(Q)} \right]$ in \eqref{WQTimesHQEstimate} will not be bounded away from zero. This means that, by the comparability stated in \cref{rem:WeGetComparabilityInTestingConditionForRieszPotential}, \eqref{ConditionForRieszPotentialToMapLpToLqInDydaEtAl} will not hold for $q < \frac{np}{n - sp}$, and thus the
inequality 
\cref{eqn:ChainOfInequalitiesForRieszPotentialMethod} will also not be true. Exactly the same argument is valid for the proof of Theorem \ref{cor:hardy_sobolev}, since the condition in \cite{1974Muckenhoupt} to prove the second inequality in \cref{eqn:ChainOfInequalitiesForRieszPotentialMethod} is also a testing condition, thus necessary for the inequality to hold.

So the Riesz potential method breaks down for median porous, but non-porous, sets $E$ in the subcritical range $q < \frac{np}{n - sp}$ ($0 < s \leq 1$). Of course this does not mean that corresponding statements for $q < \frac{np}{n - sp}$ in terms of Muckenhoupt exponents, for Theorem \ref{cor:hardy_sobolev} and Theorem \ref{FractionalAndFirstOrderHardySobolevInequalitiesViaRieszPotentials} are false. We believe they are true, but some other method of proof must be found. 
\end{remark}

\begin{remark}\label{rem:WeightedPoincareInequalities}
By the methods of proof in this section, we can immediately improve known distance weighted Poincar\'e inequalities. E.g. 
\begin{enumerate}
\item[$\bullet$] Substituting the hypothesis $1 < p < n - \overline{\dim}_A (E)$ with $1 < p < \textup{Mu}_p(E)$, we obtain the analogue of  Theorem 10.28 in \cite{KinnunenLehrbackVahakangasBook} for median porous sets.

\item[$\bullet$] With the same proof as Theorem \ref{WeightedInequalityForRieszPotentialsForMedianPorousSets}, for $q = \frac{np}{n-p}$, and substituting the hypothesis $\overline{\dim}_A (E) < \frac{q}{p}(n-p)$ with \eqref{MuckenhouptExponentHypothesesForRieszPotentialTheorem1} (but in this case there is no need for \eqref{MuckenhouptExponentHypothesesForRieszPotentialTheorem2}), we obtain the analogue of Theorem 10.29 in \cite{KinnunenLehrbackVahakangasBook} for median porous sets. 

\item[$\bullet$] Substituting the hypothesis $\overline{\dim}_A (\partial \Omega) < \frac{q}{p}(n-p)$ with $n - \frac{q}{p}(n-p) < \textup{Mu}_\infty(\partial \Omega)$, we obtain the analogue of Theorem 10.30 in \cite{KinnunenLehrbackVahakangasBook} for median porous sets. 
\end{enumerate}

We leave the details to the reader.
\end{remark}

Next we improve necessary conditions for \cref{eqn:hardy-sobolev} in the critical case.

\begin{theorem}\label{NecessaryConditionsForHardySobolevThm}
    Let $1 \leq p < q = np / (n - p) < \infty$, $\beta_1 < 0$ and let $\beta_2 = q/p(n - p + \beta_1) - n$. Assume additionally that $d^{\beta_1}$ is locally integrable. Also assume that $E$ is a median porous set. 
    If the Hardy-Sobolev inequality 
    \begin{equation}\label{eqn:HardySobolevInequalityFor NecessaryConditions}
            \left(\int_{\R^n}|f|^{q}d^{\beta_2} \, dx\right)^{1/q} \leq C\left(\int_{\R^n}|\nabla f|^{p}d^{\beta_1} \, dx\right)^{1/p}
    \end{equation}
    holds for all $f \in C_c^\infty(\R^n)$ then 
    \begin{equation}\label{NecessaryConditionForHardySobolevWithBeta2}
        \textup{Mu}_\infty(E) > -\beta_2 > 0. 
    \end{equation}
\end{theorem}

\begin{remark}
    Conclusion \eqref{NecessaryConditionForHardySobolevWithBeta2} is an improvement of Theorem 6.2 in \cite{LehrbackVahakangasInbetweenSobolevandHardy} (see also Remark 6.5 in \cite{LehrbackVahakangasInbetweenSobolevandHardy}) where the same theorem is proved under the stronger assumption that $E$ is porous. Indeed, \cite{LehrbackVahakangasInbetweenSobolevandHardy} phrased the conclusion as  $\underline{\textup{co\,dim}}_A(E) > -\beta_2$, but if $E$ is porous, $\underline{\textup{co\,dim}}_A(E)  = \textup{Mu}_\infty(E) > 0$. 

\end{remark}

\begin{proof}
    First note that, by elementary computations, $\beta_2 = \frac{q}{p}(n - p + \beta_1) - n \leq \frac{q}{p} \beta_1$ if and only if $q \leq \frac{np}{n-p}$, which holds by hypothesis. Now, since both $q > p$ and $\beta_1 < 0$ by hypothesis, then $\frac{q}{p} \beta_1 < \beta_1 < 0$. So the hypotheses imply that $\beta_2 < \beta_1 < 0$.
    
    We follow the proof in Theorem 6.2 in \cite{LehrbackVahakangasInbetweenSobolevandHardy}, skipping most of the details since they are basically the same as in that proof.
    
    Equation $(25)$ in \cite{LehrbackVahakangasInbetweenSobolevandHardy} is valid in our case, as well as equation $(27)$, with the same proof. (As in Remark 6.5 in \cite{LehrbackVahakangasInbetweenSobolevandHardy}, the reasoning in the paragraph containing equation $(26)$ at the bottom of p.354 -only the part in p.354- is not needed, since we assume that $d^{\beta_1}$ is locally integrable.)

    At this point we depart from \cite{LehrbackVahakangasInbetweenSobolevandHardy}. Indeed, since $E$ is a median porous set, we can choose $\epsilon > 0$ small enough such that $d^{\epsilon \beta_1} \in A_\infty$. In particular, $d^{\epsilon \beta_1}$ is a doubling weight. Therefore, if $f = \log d$, we get
    \begin{equation}\label{eqn:ImprovementOfExponentInAInftyWeightsViaExpLogAndReverseHolder}
        \left(\dashint_{Q}d^{(q/p)\beta_1}\, dx\right)^{p/q} \leq C \left(\dashint_{2Q}d^{\epsilon\beta_1}\, dx\right)^{1/\epsilon} \leq C \left(\dashint_{Q}d^{\epsilon\beta_1}\, dx\right)^{1/\epsilon} \leq \exp(\ang{f}_Q)^{\beta_1},
    \end{equation}
    where the first inequality is equation $(27)$ in \cite{LehrbackVahakangasInbetweenSobolevandHardy}, and the last inequality follows from the already mentioned fact (see e.g. \cite[p.218, 6.3]{SteinFatBook}) that $w \in A_\infty$ if and only if
    \begin{equation}\label{ExpLogCharacterizationOfAInfty2}
        \dashint_{Q} w(x) \, dx \lesssim \exp\left(\dashint_Q \log w(x) \, dx\right).
    \end{equation}
    This implies that $d^{(q/p)\beta_1} \in A_\infty$ (again by \eqref{ExpLogCharacterizationOfAInfty2}). 
    Until now, we have not used the value of $q$.  Since we know $q = \frac{np}{n - p}$, then 
    \[\frac{q\beta_1}{p} = q \frac{n-p}{np}\beta_2 = \beta_2,\]
    and we have that $\text{Mu}_\infty(E) > -\beta_2$ by \cref{thm:Ap_range}.
\end{proof} 

\begin{remark}
The proof of Theorem \ref{NecessaryConditionsForHardySobolevThm}  only uses $q = \frac{np}{n - p}$ at the very end. It follows from the same proof that, in the range $q < \frac{np}{n - p}$, if we  only assume that $E$ is a median porous set, instead of a porous set, a necessary condition for the Hardy-Sobolev inequality \cref{eqn:hardy-sobolev} to hold is  
\begin{equation}\label{WeakerNecessaryConditionForHardySobolevInSubcriticalRange}
    \textup{Mu}_\infty(E) > - \beta_1 \frac{q}{p}, 
\end{equation}
which was previously unknown, but it does not appear to be sharp.

Indeed, in \cite[Thm. 6.2]{LehrbackVahakangasInbetweenSobolevandHardy}, assuming $E$ is porous, the resulting necessary condition for the Hardy-Sobolev inequality \cref{eqn:hardy-sobolev} to hold is
\begin{equation}\label{NecessaryConditionForHardySobolevInSubcriticalRangeLehrbackVahakangas}
\underline{\textup{co\,dim}}_A(E) > - \beta_1 \frac{q}{p} + \left( n - \frac{q}{p}n + q \right).
\end{equation}

As mentioned in Remark \ref{WeActuallyImproveSufficientConditionsForHardySobolev}, if a set is porous, then $0 < \underline{\textup{co\,dim}}_A(E) = \textup{Mu}_\infty(E)$. Since the term $\left( n - \frac{q}{p}n + q \right) > 0 $ if and only if $q < \frac{np}{n - p}$, the conclusion \eqref{NecessaryConditionForHardySobolevInSubcriticalRangeLehrbackVahakangas} in \cite[Thm. 6.2]{LehrbackVahakangasInbetweenSobolevandHardy} is stronger than \eqref{WeakerNecessaryConditionForHardySobolevInSubcriticalRange} if we assume the set $E$ is porous. 
\end{remark}

\bibliographystyle{amsalpha}
\bibliography{ref}

\bigskip
\noindent
\textsc{Department of Mathematics, Brown University, 75 Waterman St, Providence, RI, United States 02912} \\
\textit{Email address}: \texttt{marcus\_pasquariello@brown.edu}

\bigskip
\noindent
\textsc{Department of Mathematics, University of Toronto, Room 6290, 40 St. George Street, Toronto, Ontario, Canada M5S 2E4} \\
\textit{Email address}: \texttt{ignacio.uriartetuero@utoronto.ca}

\end{document}